\documentclass[11pt,reqno]{amsart}
\RequirePackage[table]	{xcolor}			
\definecolor{green1}{RGB}{0,107,28}
\definecolor{green2}{RGB}{17,85,35}
\definecolor{green3}{RGB}{0,77,20}
\definecolor{green4}{RGB}{33,166,68}
\definecolor{green5}{RGB}{60,166,88}

\definecolor{blue1}{RGB}{3,56,91}
\definecolor{blue2}{RGB}{16,51,73}
\definecolor{blue3}{RGB}{1,40,66}
\definecolor{blue4}{RGB}{35,109,157}
\definecolor{blue5}{RGB}{59,118,157}
\RequirePackage[
	final,							
	colorlinks=true,
	urlcolor=blue3,						
	linkcolor=blue3,					
	citecolor=green4,					
	hypertexnames=false,					
	hyperfootnotes=false,					
	]{hyperref}						

\usepackage{a4wide}
\usepackage{latexsym}
\usepackage{amsopn,amscd}
\usepackage{amsfonts}
\usepackage{amssymb}
\usepackage{amsthm}
\usepackage{amsmath}
\usepackage[all]{xy}
\usepackage{scalefnt}
\usepackage[capitalise,nameinlink,noabbrev]{cleveref}
\usepackage{hyperref}

\usepackage{tikz}

\makeatletter
\newcommand*{\boxwedge}{
  \mathbin{
    \mathpalette\@boxwedge{}
  }
}
\newcommand*{\@boxwedge}[2]{
  \sbox0{$#1\boxplus\m@th$}
  \dimen2=.5\dimexpr\wd0-\ht0-\dp0\relax 
  \dimen@=\dimexpr\ht0+\dp0\relax
  \def\lw{.06}
  \kern\dimen2 
  \tikz[
    line width=\lw\dimen@,
    line join=round,
    x=\dimen@,
    y=\dimen@,
  ]
  \draw
    (\lw/2,0) rectangle (1-\lw,1-\lw)
    (\lw,0) -- (.5,1-\lw-\lw/2) -- (1-\lw-\lw/2 ,0)
  ;
  \kern\dimen2 
}

\DeclareFontFamily{OT1}{pzc}{}
\DeclareFontShape{OT1}{pzc}{m}{it}{<-> s * [1.200] pzcmi7t}{}
\DeclareMathAlphabet{\mathpzc}{OT1}{pzc}{m}{it}

\def\mathcat{\mathpzc}

\def\Alg{\mathcat{Alg}}

\def\red{\mathrm{red}}

\def\CcC{\mathcat C}

\def\ev{\mathrm{ev}}

\def\LAL{\Lambda}

\def\hh{\mathfrak h}

\def\Ho{\mathrm H}

\def\gg{\mathfrak g}

\def\qq{\mathfrak q}
\def\GG{G}

\def\TT{\mathrm T}
\def\UU{\mathrm U}
\def\Vect{\mathrm{Vect}}
\def\RR{\mathbb R}
\def\ZZ{\mathbb Z}

\def\Hom{\mathrm{Hom}}

\def\Alt{\mathrm{Alt}}

\def\Der{\mathrm{Der}}

\def\TTT{\mathrm T}
\def\TTTc{\TTT^{\mathrm c}}

\def\Sigm{\mathrm S}

\def\LALc{\LAL^{\mathrm {c}}}

\def\bra{[\,\cdot\, , \,\cdot\,]}
\def\pbra{\{\,\cdot\, , \,\cdot\,\}}

\def\MH{\mathcal H}

\def\Lc{\Lambda^{\mathrm{c}}}

\def\Ho{\mathrm{H}}

\def\modu{\mathcat M}

\def\CC{\mathbb C}
\def\AA{\mathcal A}

\def\can{\mathrm{can}}

\def\fund{\mathrm{fund}}

\def\mrc{\mathrm c}

\def\ad{\mathrm{ad}}

\def\Id{\mathrm{Id}}

\def\hinn{\langle\,\cdot\, , \,\cdot \,\rangle}

\def\Rep{\mathrm{Rep}}
\def\Ad{\mathrm{Ad}}

\def\dR{\Omega}

\def\form{\,\bullet\,}
\def\fus{\mathrm{fus}}
\def\Form{\Omega}
\def\pr{\mathrm{pr}}
\def\mult{\mathrm{mult}}

\def\inv{\mathrm{inv}}

\def\Pii{F}
\def\diag{\mathrm{diag}}
\def\conj{\mathrm {conj}}
\def\bK{\mathbb K}
\def\bR{\mathcat R}

\def\oomega{\mathcat H}
\def\varotheta{{\psi^\odot}}
\def\psidot{{\psi^{\,\bullet}}}
\def\Map{\mathrm{Map}}
\def\eeta{\eta}
\def\twist{\mathrm{twist}}

\def\prin{\xi}

\def\Prin{\mathcal P}

\def\ovomega{\overline \omega}
\def\ovoomega{\widetilde \oomega}
\def\ofund{\widetilde \fund}

\def\oGG{\widetilde \GG}

\def\ogg{\widetilde \gg}
\def\omult{\widetilde \mult}
\def\oform{\widetilde \form}

\def\olambda{\widetilde \lambda}

\def\mmu{\Phi}

\def\fflat{\flat}

\def\pp{\mathbf p}

\def\ran{\mathrm{ran}}
\def\qqp{\mathbf q}
\def\pairi{\langle \,\cdot\, , \,\cdot\,\rangle}
\def\Gr{\mathrm {Gr}}

\newcommand{\Conn}{\mathcal A_{\prin}}

\newcommand{\cgroup}{G}

\newtheorem{thm}{Theorem}[section]
\newtheorem{cor}[thm]{Corollary}
\newtheorem{lem}[thm]{Lemma}
\newtheorem{prop}[thm]{Proposition}

\newtheorem{examp}[thm]{Example}

\newtheorem{rema}[thm]{Remark}

\newtheorem{concl}[thm]{Conclusion}
\newtheorem{compl}[thm]{Complement}

\long
\def\MSC#1\EndMSC{\def\arg{#1}\ifx\arg\empty\relax\else
      {\par\narrower\noindent
      2020 Mathematics Subject Classification. #1\par}\fi}

\long
\def\KEY#1\EndKEY{\def\arg{#1}\ifx\arg\empty\relax\else
    {\par\narrower\noindent
      Keywords and Phrases: #1\par}\fi}

\title[Quasi  structures and Poisson geometry of moduli spaces]
{Quasi Poisson structures, weakly quasi Hamiltonian structures, and
Poisson geometry of  various moduli spaces}

\author{Johannes Huebschmann  }
\address{
\newline
Universit\'e de Lille - Sciences et Technologies 
\\
D\'epartement de Math\'ematiques\\
\newline CNRS-UMR 8524,
Labex CEMPI (ANR-11-LABX-0007-01)
\\
\newline
59655 Villeneuve d'Ascq Cedex, France\\
\newline
Johannes.Huebschmann@univ-lille.fr
 }

\date{\today}
\numberwithin{equation}{section}

\begin{document}
\setcounter{page}{1}

\begin{abstract}
\noindent
Let $\GG$ be a Lie group and  $\gg$  its Lie algebra.
The aim here is to  develop  
a theory 
 of quasi Poisson structures
relative to a not necessarily non-degenerate $\Ad$-invariant 
 symmetric $2$-tensor in $\gg \otimes \gg$
and one of general not necessarily non-degenerate
quasi Hamiltonian structures relative to 
a not necessarily  non-degenerate
$\Ad$-invariant
symmetric bilinear form on $\gg$, a quasi Poisson structure being
given by a skew bracket of two variables
such that suitable data defined in terms of 
$\GG$ as symmetry group involving the $2$-tensor in the tensor square  
of the Lie algebra $\gg$
 measure how that bracket
fails to satisfy  the Jacobi identity.
The present approach  involves a novel concept of momentum mapping
and yields, in the non-degenerate case,
a bijective correspondence,
in terms of explicit algebraic expressions,
between 
non-degenerate quasi Poisson structures and
non-degenerate
quasi Hamiltonian structures.
The new theory applies to  various 
not necessarily non-singular moduli spaces
and yields thereupon, via reduction with respect to
an appropriately defined  momentum mapping, 
not necessarily non-degenerate ordinary
Poisson structures arising from the data including the
symmetric 2-tensor in $\gg\otimes \gg$.
Among these moduli spaces
are representation spaces, possibly twisted,
of the fundamental group of a Riemann surface,
possibly punctured, and
moduli spaces of semistable holomorphic vector bundles
as well as Higgs bundle moduli spaces.
In the non-degenerate case,
such a Poisson structure comes down to a stratified symplectic one
of the kind explored in the literature and recovers, e.g.,
the symplectic part of a K\"ahler structure
introduced by Narasimhan and Seshadri for moduli spaces
of stable holomorphic vector bundles on a curve.
In the algebraic setting,  these moduli spaces
arise as not necessarily non-singular
affine not necessarily non-degenerate Poisson varieties.
Along the way, a side result is an explicit equivalence
between extended moduli spaces
and quasi Hamiltonian spaces independently of gauge theory.
\end{abstract}

\maketitle

\MSC 

\noindent
Primary: 53D30 

\noindent
Secondary: 14D21 14L24 14H60 32S60 53D17 53D20 58D27 81T13
\EndMSC

\KEY
Quasi Poisson structure, weakly quasi Hamiltonian structure,
 extended moduli space,
momentum mapping in quasi setting,
representation space,
analytic representation variety,
algebraic representation variety,
analytic quasi Hamiltonian reduction,
algebraic quasi Hamiltonian reduction,
analytic quasi Poisson reduction,
algebraic quasi Poisson reduction, affine Poisson variety,
Dirac structure.

 \EndKEY

{\tableofcontents}
\section{Introduction}
Let $\GG$ be a Lie group and let $\gg$ denote its Lie algebra.
In the realm of $\GG$-manifolds,
extending an approach in \cite{MR1880957} 
for the case where $\GG$ is compact,
we develop  
a theory  of quasi Poisson structures
relative to a not necessarily non-degenerate $\Ad$-invariant 
 symmetric $2$-tensor in $\gg \otimes \gg$
and one of
general not necessarily non-degenerate
quasi Hamiltonian structures relative to 
a not necessarily  non-degenerate
$\Ad$-invariant 
symmetric bilinear form on $\gg$.
Our approach involves a novel concept of momentum mapping.
Using this theory,  we endow various 
not necessarily non-singular moduli spaces
with 
not necessarily non-degenerate
Poisson structures arising from the data including the
symmetric 2-tensor in $\gg \otimes \gg$. This generalizes
earlier constructions in the literature
for the special case where the 
symmetric bilinear form on $\gg$
(and hence the $2$-tensor on $\gg$
it determines) are  non-degenerate
and the Poisson structure is smooth and non-degenerate.

The papers \cite{MR1460627}, \cite{MR1370113}, \cite {MR1277051}, 
\cite{MR1470732}, 
building on \cite{MR1112494}
and \cite{MR1362845},
offer a purely finite-dimensional construction
of the moduli spaces of semistable holomorphic vector bundles
on a Riemann surface, possibly punctured,
and of generalizations thereof
including moduli spaces of principal bundles
as stratified symplectic spaces in the sense of
\cite{MR1127479}, realized,
according to the {\em Hitchin-Kobayashi\/} correspondence
for principal bundles on a Riemann surface,
 as spaces of twisted representations
of the fundamental group in a compact Lie group.
The construction
proceeds by ordinary symplectic reduction applied to
a finite-dimensional {\em extended moduli space\/}
arising from a  product of $2\ell$ copies of the Lie group
(the group $\UU(n)$ for the case of holomorphic  rank $n$  vector bundles)
where $\ell$ is the genus of the surface or,
in the presence of punctures, from a variant thereof.
This structure depends on the  Lie group, 
a choice of an invariant inner product on its Lie algebra,
and the topology of a corresponding bundle, 
but is independent of any complex structure on the Riemann surface. 
It also makes sense for a not necessarily compact  Lie group 
and an $\Ad$-invariant 
not necessarily non-degenerate
nor positive
symmetric bilinear form on the Lie algebra thereof.

The paper \cite{MR1638045} establishes a theory of  quasi Hamiltonian spaces
with respect to a compact Lie group $\GG$
and an $\Ad$-invariant positive definite symmetric bilinear form
on its Lie algebra  $\gg$
and  reworks  
the above extended moduli spaces
in terms of quasi Hamiltonian structures.
This does not recover
the complete story, however, 
since, in the presence of singularities, 
quasi Hamiltonian reduction does not recover
the entire moduli space resulting from
the extended moduli space;
see Remark \ref{qhreduction} below.
The paper \cite{MR1880957}
develops a theory of  quasi Poisson  spaces
with respect to a compact Lie group $\GG$
and an $\Ad$-invariant positive definite symmetric bilinear form
on the  Lie algebra  $\gg$ of $\GG$
and extends the classical
bijective correspondence between
symplectic structures and non-degenerate Poisson structures
to the quasi case with respect to a compact Lie group,
and  \cite{MR2642360} reworks the theory for
a Lie group $\GG$ with a non-degenerate $\Ad$-invariant
symmetric bilinear form on its Lie algebra $\gg$.

One aim of the present paper is
to build
a theory of 
general 
quasi Poisson structures,  a quasi Poisson structure being
given by a skew bracket of two variables
such that suitable data defined in terms of 
a symmetry group involving the $2$-tensor over the Lie algebra
 measure how that bracket
fails to satisfy  the Jacobi identity, and one of
general not necessarily non-degenerate
quasi Hamiltonian structures,
and to use this theory to rework those moduli space constructions 
in an algebraic manner
over a general algebraically closed field of
characteristic zero. 
Another aim is, as a reaction to
a referee request, 
 to make precise
 the equivalence between the extended moduli spaces
and  the corresponding constructions in
\cite{MR1638045} that lead
to moduli spaces or twisted representation spaces over Riemann surfaces.

This paper is addressed  to the expert.
We therefore explain some of the technical details here in the introduction.
We proceed as follows:

Let $\bK$ denote the base field,
a field of characteristic zero,
mainly the reals or the complex numbers,
but also the field of definition in an algebraic context.
View $\GG$ as a $\GG$-manifold with respect to conjugation.
In an algebraic context we take  $\GG$ to be algebraic.
Let $M$ be a $\GG$-manifold,
a non-singular  affine $\GG$-variety in an algebraic context.
Let $\fund_M \colon M \times \gg \to \TT M$ denote the
fundamental vector field map associated with the $\GG$-action
on $M$.
Consider a $\GG$-equivariant  map $\mmu\colon M \to \GG$,
let $\TT_\Phi \GG \to M$ 
be the vector bundle on $M$ which $\Phi$ induces from the tangent bundle
 of $\GG$ and factor 
the  derivative $d\Phi\colon \TT M \to \TT G$
through the resulting morphism  
$(d \Phi)_M\colon \TT M \to \TT_\Phi G$
of vector bundles on $M$.
Left and right translation in $\GG$ induce trivializations
$L_\Phi, R_\Phi \colon M \times \gg \to \TT_\Phi \GG$
of $\TT_\Phi \GG \to M$.
Write the tensor product of vector bundles on $M$ as
$\otimes_M$.
Let $\TT^2 M \to M$ denote the tensor square 
$\TT M \otimes_M \TT M \to M$
of the tangent bundle $\TT M \to M$ of $M$
and interpret, in the obvious way,
 $M \times (\gg \otimes \gg) \to M$,
 $\TT M \otimes \gg \to M$ and
$\TT_\Phi \GG \otimes \gg \to M$
as the respective vector bundles
$(M \times \gg) \otimes_M (M \times\gg) \to M$,
$\TT M \otimes_M(M \times \gg) \to M$, and
$\TT_\Phi \GG \otimes_M(M \times \gg) \to M$.

Let $\,\form\,$ be an $\Ad$-invariant not necessarily non-degenerate
symmetric bilinear form on  $\gg$.
Furthermore, let $\sigma$ be a $\GG$-invariant $2$-form on  $M$.
We define $\Phi$
to be a
$\GG$-{\em momentum mapping for\/} $\sigma$
{\em relative to\/} $\,\form\,$
when $(d \Phi)_M$  renders the diagram
\begin{equation}
\begin{gathered}
\xymatrixcolsep{6.5pc}
\xymatrix{
\gg \otimes \TT M  \ar[d]_{\Id \otimes_M (d \mmu)_M} \ar[r]^{\fund_M \otimes_M \Id} & \TT^2 M \ar[r]^{ \sigma} &\bK \ar@{=}[d]
\\
\gg \otimes (\TT_\mmu\GG) \ar[r]_{\Id \otimes_M
\tfrac 12
\left(L_\mmu^{-1} + R_\mmu^{-1} \right) \phantom{aaa}}
& M\times(\gg \otimes \gg) \ar[r]_{\phantom{aaa}\form} &\bK
}
\end{gathered}
\label{momintvar}
\end{equation}
commutative.
Here 
the notation $ \fund_M\otimes_M \Id $ signifies
\begin{equation}
\fund_M\otimes_M \Id \colon (M \times \gg) \otimes _M \TT M
\longrightarrow  \TT M \otimes _M \TT M,
\end{equation}
 $\Id \otimes (d \mmu)_M$ signifies
\begin{equation}
\Id \otimes (d \mmu)_M \colon (M \times \gg) \otimes_M \TT M 
\longrightarrow  (M \times \gg) \otimes_M \TT_\Phi \GG, 
\end{equation}
and $\Id \otimes (L^{-1}_\Phi+R^{-1}_\Phi)$
is a short hand notation for 
\begin{equation}
(M \times \gg) \otimes_M(\TT_\mmu\GG) 
\stackrel{\Id \otimes_M (L^{-1}_\Phi + R^{-1}_\Phi)} \longrightarrow
(M \times \gg)  \otimes _M(M \times \gg). 
\end{equation}
This definition does not involve the cotangent bundle of $M$ explicitly
and hence still applies to situations where
there is no canonical cotangent bundle,
e.g., in an infinite dimensional situation.
Diagram \eqref{momintvar} looks, perhaps, unnecessarily complicated;
indeed, a similar diagram 
with $d\Phi$ and $\TT \GG$ instead of
$(d\Phi)_M$ and $\TT_\Phi \GG$
characterizes the momentum property as well, but \eqref{momintvar}
renders the comparison with 
\eqref{PMPhi} below straightforward.
The reader will notice when we exchange the order of 
the two tensor factors
$\fund_M$ and $\Id$
in the upper row of \eqref{momintvar}
and make the corresponding change in the lower row,
the resulting diagram characterizes the negative of $\sigma$.

When $M$ is a conjugacy class $\CcC$ in $\GG$, viewed as
a $\GG$-manifold relative to conjugation,
 and $\Phi$ the inclusion,
diagram \eqref{momintvar}
characterizes a $\GG$-invariant
$2$-form $\sigma$ on $\CcC$, necessarily alternating 
since $\,\form\,$ is $\Ad$-invariant,
and the inclusion into $\GG$ is a $\GG$-momentum mapping
for  $\sigma$ relative to $\,\form\,$.
This $2$-form plays a major role in the theory
of moduli spaces; we explain this below.

The $2$-form $\,\form\,$ on $\gg$ 
determines a biinvariant $3$-form $\lambda$ on $\GG$
(the Cartan form). 
We define $\sigma$
 to be
$\Phi$-{\em quasi closed 
relative to\/} $\,\form\,$
when
$d \sigma= \Phi^* \lambda$.
A pair $(\sigma,\Phi)$ of this kind constitutes
a {\em weakly $\GG$-quasi Hamiltonian structure\/}
on $M$ {\em relative to\/}
$\,\form\,$ when
 $\sigma$ is $\Phi$-quasi closed and $\Phi$
 a $\GG$-momentum mapping for $\sigma$ relative to $\,\form\,$;
the pair $(\sigma,\Phi)$  is
a {\em $\GG$-quasi Hamiltonian structure\/}
on $M$ {\em relative to\/} $\,\form\,$
when it satisfies, furthermore, a non-degeneracy constraint.
We develop a theory of 
weak $\GG$-quasi Hamiltonian structures
generalizing that of
$\GG$-quasi Hamiltonian structures
in \cite{MR1638045} for the case where
$\GG$ is compact and $\,\form\,$ positive definite.

Let $\oomega$ be an $\Ad$-invariant not 
necessarily non-degenerate symmetric $2$-tensor in $\gg \otimes \gg$.
In a similar vein, we show the $\Ad$-invariant  symmetric $2$-tensor
$\oomega$ in $\gg \otimes \gg$
determines
a totally  antisymmetric $3$-tensor $\phi_\oomega$ in $\gg \otimes \gg$;
see Section \ref{quasipoiss} for details.
This $3$-tensor 
represents a class in the third Lie algebra homology group of $\gg$,
necessarily non-zero 
when $\oomega$ is non-trivial.
Say $\oomega$ is {\em non-degenerate\/}
when so is the $\Ad$-invariant symmetric bilinear form
on $\gg$ which $\oomega$ induces.
In this case, the 
$3$-tensor $\phi_\oomega$ is dual to the $3$-form due to E. Cartan.
However, there are interesting
cases where $\oomega$ is not non-degenerate:
Consider, e.g.,  a Lie algebra $\gg$ which decomposes as
the sum $\gg_1 \oplus \gg_2$ of two  ideals,
required to be $\Ad$-invariant when $\GG$ is not connected, and
take $\oomega$ non-degenerate over $\gg_1$ and
zero over $\gg_2$.
Another example is the Lie algebra of infinitesimal gauge transformations
of a principal bundle 
(interpreted naively, e.g., in the Fr\'echet topology, but not in terms
of suitable Sobolev structures)
such that the Lie algebra of its structure group
carries an $\Ad$-invariant  symmetric $2$-tensor
$\oomega$; even when $\oomega$ is non-degenerate,
the resulting  symmetric $2$-tensor on
the Lie algebra of infinitesimal gauge transformations
is no longer (naively) non-degenerate;
see Example \ref{examp3} below.
Perhaps one can 
use this structure to
rebuild 
gauge theory 
in the Fr\'echet topology.
This applies in particular to the loop group,
and we discuss possible consequences in Subsection \ref{loop}.

Let $P$ be an antisymmetric 
 $\GG$-invariant $2$-tensor on a $\GG$-manifold $M$.
We define a  $\GG$-equivariant 
map $\Phi \colon M \to \GG$ 
to be a $\GG$-{\em momentum mapping for\/}
$P$ 
{\em relative to\/} $\oomega$
when the diagram
\begin{equation}
\begin{gathered}
\xymatrixcolsep{4pc}
\xymatrix{
M \ar[r]^{P}\ar@{=}[d]  
& 
\TT^2 M\ar[r]^{(d \Phi)_M \otimes_M \Id\phantom{aaaa}}
& 
(\TT_\Phi \GG) \otimes_M \TT M 
\\
M \ar[r]_{\oomega\phantom {aaaa}}&M \times (\gg \otimes \gg) 
\ar[r]_{\phantom{a}\tfrac 12(L_\Phi+R_\Phi)\otimes_M \Id}&
(\TT_\Phi\GG) \otimes \gg 
\ar[u]|-{\Id \otimes_M \fund_M}
}
\end{gathered}
\label{PMPhi}
\end{equation}
commutes.
Here the notation 
$(L_\Phi+R_\Phi)\otimes_M \Id$
is a short hand notation for 
\begin{equation}
(M \times \gg) \otimes_M(M \times \gg)
\stackrel{(L_\Phi + R_\Phi)\otimes_M \Id} \longrightarrow
 \TT_\Phi \GG \otimes _M(M \times \gg),
\end{equation}
and we do not distinguish in notation between a $2$-tensor
on $M$ and the section of the vector bundle $\TT^2 M \to M$
it defines.
Again this definition does not involve the cotangent bundle of $M$ 
explicitly
and hence still applies to situations where
there is no canonical cotangent bundle.
The reader will notice the formal similarity
between \eqref{momintvar} and \eqref{PMPhi}.
Also, the notion of Poisson momentum mapping
can be cast in a similar diagrammatic description,
with the identity of the corresponding Lie algebra $\gg$,
viewed as a member of $\gg^* \otimes \gg$, playing the role of $\oomega$;
see Remark \ref{rema1} below.

When $M$ is the group $\GG$ itself, taken as a $\GG$-manifold
relative to conjugation, and $\Phi$ the identity,
diagram \eqref{PMPhi} characterizes a $\GG$-invariant
 $2$-tensor $P_\GG \in \TT^2 \GG$,
necessarily antisymmetric since $\oomega$ is $\Ad$-invariant,
and the identity of $\GG$ is a $\GG$-momentum mapping
for $P_\GG$ relative to $\oomega$.
It induces an $\GG$-invariant antisymmetric
 $2$-tensor 
on each conjugacy class in $\GG$
such that the inclusion
is a $\GG$-momentum mapping
for it  relative to $\oomega$.
These $2$-tensors play a major role in the theory.

We say the $\GG$-invariant antisymmetric $2$-tensor $P$ 
on the $\GG$-manifold $M$
is $\Phi$-{\em quasi Poisson relative to\/} 
$\oomega$
when
the Schouten square $[P,P]$ of $P$, by construction a $3$-tensor on  $M$, 
coincides with a (suitably defined) multiple of
the image of $\phi_\oomega$ 
under the  infinitesimal $\gg$-action on $M$, and we say
$(P,\Phi)$
is $\GG$-{\em Hamiltonian\/}
when $\Phi$ is a $\GG$-momentum mapping for $P$ relative to $\oomega$.
For this concept of
Hamiltonian
 $\GG$-quasi Poisson  structure,
we build a theory
generalizing that of
quasi Poisson  structures
in  \cite{MR1880957} 
for the case where
$\GG$ is compact;
in that paper,
the corresponding  $\Ad$-invariant  symmetric $2$-tensor
$\oomega$ in $\gg \otimes \gg$
is that which arises
from an  
$\Ad$-invariant positive
definite symmetric bilinear form
on $\gg$ in the obvious way.
The reasoning in  \cite{MR1880957} and  \cite{MR1638045} 
involves the positivity
of the  bilinear form
on $\gg$ explicitly.
The paper \cite{MR2103001} redevelops
quasi Poisson  structures 
in terms of Dirac structures
without the positivity assumption
and \cite{MR2642360}
pushes these ideas further 
but both papers 
start from an $\Ad$-invariant non-degenerate
symmetric bilinear form on $\gg$
and work with the resulting
$\Ad$-invariant symmetric $2$-tensor.
In our more general setting,
a completely new approach
is necessary, we
 cannot blindly exploit 
the constructions and observations 
\cite{MR1638045, MR1880957, 
MR2068969,
MR2103001,
MR2642360},
and the Dirac structure description is no longer available
when the $2$-tensor $\oomega$ is not non-degenerate.
Also, \cite{MR2068969} 
includes a characterization of
quasi Hamiltonian structures
(with respect to a non-degenerate 
$\Ad$-invariant 
symmetric bilinear form on $\gg$)
in terms of Dirac structures,
but this characterization does not extend to 
weakly quasi Hamiltonian structures
and in particular is not available
with respect to a degenerate 
$\Ad$-invariant 
symmetric bilinear form on $\gg$.
The paper \cite{MR2806566} reworks and extends
the relationship between Dirac  and quasi Poisson structures, 
still relying on the non-degeneracy of the corresponding
$2$-form on the Lie algebra under discussion, and there is little
overlap between that paper and the present one.

In Section \ref{reps} we recall, for later reference in this paper,  
a description of the moduli spaces
of interest here.
After some preparations in Section \ref{prelimi}
we develop, in Section \ref{qhs}, the theory of
weakly quasi Hamiltonian structures
in our general setting relative to a not necessarily
non-degenerate invariant symmetric bilinear form on the Lie algebra, 
with an eye towards
the comparison between
extended moduli spaces and weakly quasi Hamiltonian spaces.
We spell out explicitly the comparison as
Conclusions \ref{concl2} and \ref{concl3}.
In Sections \ref{quasipoiss} and \ref{fusion}, we build the
theory of quasi Poisson structures in our general setting
 relative to a general
invariant symmetric $2$-tensor over the Lie algebra.
In Section \ref{comparison}  we 
develop  notions of momentum duality
and non-degeneracy;
our notions of non-degeneracy
elaborate on the notions of non-degeneracy
in \cite{MR1638045} 
for quasi Hamiltonian structures
relative to a compact group
and in \cite{MR1880957} for quasi Poisson structures
relative to a compact group.
The main result in Section \ref{comparison},  Theorem \ref{existence},
 establishes,
in our general setting,
a bijective correspondence
in terms of explicit algebraic expressions
between (non-degenerate)
quasi Hamiltonian structures and
non-degenerate quasi Poisson structures,
and we refer to this correspondence as {\em momentum duality\/}.
This Theorem enables us to conclude that,
when the quasi Hamiltonian structure is algebraic,
its momentum dual quasi Poisson structure is algebraic, and vice versa,
cf. Corollary \ref{momdalg}.
The proof of Theorem \ref{existence}
substantially relies on material from the theory
of Dirac structures
\cite{MR2068969},
\cite{MR2103001},
\cite{MR2642360};
indeed Theorems \ref{dvsq1} and \ref{dvsq2} 
summarize versions of the crucial
relationships between twisted Dirac structures
and quasi Hamiltonian structures
\cite{MR2068969} and 
 between twisted Dirac structures
and Hamiltonian quasi Poisson structures
\cite{MR2103001}, respectively.
Theorem \ref{dvsq3} 
gives a Dirac theoretic version of  Theorem \ref{existence}.
In Section \ref{msrrev} we apply the results in previous sections to
the moduli spaces under discussion.
In particular, in the algebraic setting,
we obtain these moduli spaces as 
not necessarily non-singular
affine Poisson varieties, cf. Theorem \ref{ms3}.
Here the construction of the 
Poisson structure as an algebraic object is
the major issue.

Affine algebraic Poisson 
structures on \lq\lq wild character varieties\rq\rq
are in \cite{MR3126570};
see \cite{MR3931781} and the literature there.
 We explain the connections between such moduli spaces 
and our results
in Subsections \ref{alternateq} and \ref{stokesd} below.
Suffice it to mention here that our results go beyond those in
 \cite{MR3126570} in two ways: they also cover 
the singular \lq\lq wild character varieties\rq\rq\ 
of the kind $\Phi^{-1}(\CcC)// \mathbf H$
(notation as in  Subsection \ref{alternateq})
 and, moreover, 
 yield Poisson structures on \lq\lq wild character varieties\rq\rq\ 
more general than those in \cite{MR3126570}; see
 Subsection \ref{stokesd} for details.
It is, perhaps, worthwhile noting that the singular case is the typical case.
Also, Atiyah and Bott emphatically raised the issue
of understanding such singularities \cite{MR702806}.

Our exposition is in the spirit of a tradition that goes back 
to Saunders Mac Lane:
Favor  commutative diagrams over complicated formulas.
This is, perhaps, not the most concise approach
but has the advantage of being  categorical, at least in principle.
We apologize for the length of the paper;
working out the comparison between 
extended moduli spaces and quasi Hamiltonian structures
turns out to be a Sisyphean task, as is
the proof of Theorem \ref{existence}
since we cannot simply quote the requisite results from
\cite{MR2068969},
\cite{MR2103001},
\cite{MR2642360}.

\section{Representations of surface groups}
\label{reps}

The main application of the theory we build is to
moduli spaces of representations. In view of the Kobayashi-Hitchin 
correspondence and the non-abelian Hodge correspondence
\cite {MR965220, MR887285, MR887284, MR1179076, MR1307297, MR1320603},
this recovers moduli spaces of semistable holomorphic vector bundles
as well as Higgs bundles moduli spaces.
For ease of exposition, 
we recall the following notation and terminology from
\cite{MR3836789, MR1460627, MR1370113}: 

Consider the standard presentation
\begin{equation}
\mathcal P 
= \langle x_1,y_1,\dots,x_\ell,y_\ell, z_1,\ldots,z_n; r\rangle,
\quad r = \Pi [x_j,y_j] z_1 \cdots z_n,
\label{standpre2}
\end{equation}
of the fundamental group $\pi$ of a (real) compact surface
of genus $\ell \geq 0$ with $n\geq 0$ boundary 
circles, and suppose $n \geq 3$ when $\ell = 0$,
cf. \cite[(2.1)~p.~381]{MR1460627}. 
(An assumption of the kind  $\ell + n >0$ avoids inconsistencies, 
and the case $(\ell,n)=(0,2)$ is special and not interesting.)
In the literature, it is also common to consider
punctures rather than boundary circles.
Let
$\Pii$ be the free group on the generators
$x_1,y_1,\dots,x_\ell,y_\ell, z_1,\ldots,z_n$.

Let $\GG$ be an algebraic group, 
let ${\mathbf C} = \{\CcC_1,\ldots, \CcC_n\}$ 
be a family of $n$ conjugacy classes
in $\GG$, let $\Hom(F,\GG)_{\mathbf C}$
denote the space of homomorphisms from $F$ to
$\GG$ for which the value of the generator $z_j$
lies in $\CcC_j$ ($1 \leq j \leq n$),
and let $\Hom(\pi,\GG)_{\mathbf C}$
denote the space of homomorphisms from $\pi$ to
$\GG$ for which the value of the generator $z_j$
lies in $\CcC_j$ ($1 \leq j \leq n$).
The relator $r$ induces an  algebraic  map
\begin{equation}
r \colon \Hom(F,\GG)_{\mathbf C} \cong \cgroup^{2 \ell}
\times \CcC_1 \times \ldots \times \CcC_n \longrightarrow \cgroup,
\label{relmap2}
\end{equation}
the canonical projection $F \to \pi$ induces 
an embedding $\Hom(\pi,\GG)_{\mathbf C} \to \Hom(F,\GG)_{\mathbf C}$,
and this embedding realizes
 $\Hom(\pi,\GG)_{\mathbf C}$ as the  affine variety $r^{-1}(e)$ in  $\Hom(F,\GG)_{\mathbf C}$.
We then use the notation
$\mathrm{Rep}(\pi,\GG)_{\mathbf C}$
for a suitably defined $\GG$-quotient
of $\Hom(\pi,\GG)_{\mathbf C}$
 \cite[Section 9 p.~403]{MR1460627},
the ordinary orbit space, when $\GG$ is compact
(and the ground field is that of the reals), and a suitable 
affine algebraic quotient 
when $\GG$ is a reductive algebraic group.

Consider the special case $n=0$. The group $\pi$
is  now the fundamental group  of a closed Riemann surface $\Sigma$
of genus 
$\ell \geq 1$. Let $N$ be the normal closure
of $r$ in $F$, and let $\Gamma = F/[F,N]$.
The image $[r] \in \Gamma$ of $r$ generates a free 
cyclic central subgroup $\ZZ\langle[r]\rangle$ of $\Gamma$
\cite[\S 6] {MR702806}, \cite[(2.2)]{MR3836789}.
Let $X$ be a point of the center of $\gg$ such that
$\exp(X)$ lies in the center of $\GG$ and such that 
$r^{-1}(\exp(X))$ is non-empty. 
Let $\Hom_X(\Gamma,\GG)$ be the subspace of $\Hom(\Gamma,\GG)$
that consists of the homomorphisms $\chi \colon \Gamma \to \GG$
that enjoy the property $\chi([r]) =\exp(X) \in \GG$.
The canonical projection $F \to \Gamma$
induces an injection $\Hom(\Gamma,\GG) \to \Hom(F,\GG) \cong \GG^{2\ell}$,
and this injection
identifies the space
$\Hom_X(\Gamma,\GG)$
with the subvariety  $r^{-1}(\exp(X))$ of $\Hom(F,\GG)$.
Under these circumstances, we use the notation
$\mathrm{Rep}_X(\Gamma,\GG)$
for the $\GG$-quotient
of $\Hom_X(\Gamma,\GG)$
\cite[Section 6 p.~754]{MR1370113},
the ordinary orbit space, when $\GG$ is compact
(and the ground field is that of the reals), and a suitable 
affine algebraic quotient 
when $\GG$ is a reductive algebraic group.
As a space,
 $\Hom_X(\Gamma,\GG)$ and hence
$\mathrm{Rep}_X(\Gamma,\GG)$
depends only on the value $\exp(X) \in \GG$
(and hence it makes sense to take the affine algebraic quotient)
but, when we establish the link with principal bundles,
the member $X$ of the center of 
$\gg$ actually recovers a topological characteristic class
of the corresponding bundle.

Let $\widehat \Sigma$ arise from $\Sigma$ by removing an open disk.
The presentation
\begin{equation}
 \langle x_1,y_1,\dots,x_\ell,y_\ell, z; r\rangle,
\quad r = \Pi [x_j,y_j] z,
\end{equation}
of the fundamental group $\widehat \pi$ of $\widehat \Sigma$
 entails  an equivalent description of $\mathrm{Rep}_X(\Gamma,\GG)$:
Let $\CcC =\{\exp (X)\}\subseteq \GG$,
the conjugacy class of the member  $\exp (X)$ of the center of $\GG$.
The canonical projection $\widehat \pi \to \Gamma$
induces an isomorphism
of algebraic varieties from
$\Hom_X(\Gamma,\GG)$ to $\Hom(\widehat \pi,\GG)_{\{\CcC\}}$,
and hence the quotients
$\Rep_X(\Gamma,\GG)$ and
$\Rep(\widehat \pi,\GG)_{\{\CcC\}}$
essentially coincide as spaces.
The construction of the 
additional structure on
$\Rep_X(\Gamma,\GG)$ which we carry out in this paper
is considerably less involved than that
on a space of the kind
$\Rep(\widehat \pi,\GG)_{\{\CcC\}}$, however.

\section{Preliminaries}
\label{prelimi}

We write the ground field as $\bK$, a field of characteristic zero,
mostly the reals $\RR$ or the complex numbers
$\CC$ but, in the algebraic case,
we use the notation $\bK$ for the field of definition as well if need be.
At times we also work over a more general commutative ring
$\bR$, and we then suppose that $\bR$ is an algebra over the rationals.

Below, the terms manifold and group refer to a smooth manifold and a Lie 
group, or  to an analytic  manifold and a Lie group, taken as an analytic 
group, or  to an affine algebraic  manifold (non-singular affine variety) 
and an algebraic group,  defined over a not necessarily algebraically 
closed field of characteristic zero.
Consider a manifold $M$. We denote by $\mathcat A[M]$
the structure algebra of functions on $M$
(smooth, analytic or algebraic, as the case may be);
thus in the algebraic case,
$\mathcat A[M]$ refers to the coordinate ring of $M$.
We refer to $\mathcat A[M]$ as the 
{\em algebra of admissible functions\/} on $M$
and to the members of
$\mathcat A[M]$ as {\em admissible functions\/}. 
We use the notation
 $\Vect(M) $ for the vector fields
on $M$, viewed as a module over $\mathcat A[M]$ and, as usual, we identify
 $\Vect(M) $ with the derivations of 
$\mathcat A[M]$. 
We write a canonically arising map as $\can$.

For two $\bR$-modules $\modu_1$ and $\modu_2$, we use the notation
$\twist \colon \modu_1 \otimes \modu_2 \to \modu_2 \otimes \modu_1$
for the interchange map.
Recall a {\em coordinate system\/}
of an $\bR$-module $\modu$ consists of
a family $(e_j)$ of members of $\modu$ and  
a family $(\eeta^k)$ of members of the $\bR$-dual $\modu^*= \Hom(\modu, \bR)$
such that the canonical morphism
\begin{equation}
\modu^* \otimes \modu \longrightarrow \Hom(\modu,\modu)
\end{equation}
of $\bR$-modules 
sends $\eeta^k \otimes e_k$ to the identity of $\modu$.
An $\bR$-module 
has a coordinate system
if and only if it is finitely
generated and projective.

For a manifold $M$ and a vector space $V$, occasionally we do not distinguish
in notation between a $V$-valued differential operator on $M$ and
the corresponding map $\TT^* M \to V$.
For two vector bundles $\alpha_1 \colon E_1 \to M$
and $\alpha_2 \colon E_2 \to M$ on $M$,
we write their tensor product 
in the category of vector bundles on $M$
as
\begin{equation}
\alpha_1 \otimes_M \alpha_2\colon E_1\otimes_M E_2 \longrightarrow M
\end{equation}
and
their Whitney sum
as
\begin{equation}
\alpha_1 \oplus_M \alpha_2\colon E_1\oplus_M E_2 \longrightarrow M.
\end{equation}
For two manifolds $M$ and $N$, we write
$
(\TT M) \otimes (\TT N) 
$
for the total space
of the tensor product 
\begin{equation}
((\TT M) \times N) \otimes (M \times (\TT N))
\longrightarrow M \times N
\end{equation}
of the induced vector bundles on $M\times N$
as indicated by the notation.

The material in the present section is completely standard.
We spell out details to set the stage
and to fix signs.

\subsection{Lie group actions on manifolds} 
Let $\GG$ be a Lie group.
Denote by $\gg$
the Lie algebra of left 
invariant vector fields on $\GG$,
with the Lie bracket $\gg$ inherits from
the Lie bracket of vector fields on $\GG$.
Let $X \in \gg$.
Right and left translation 
induce
 the vector fields
$X^L$ and $X^R$, respectively, as
\begin{align}
X^L_q &= \tfrac d {dt}|_{t=0}(q \exp(t X)),\ q \in \GG,
\\
X^R_q &= \tfrac d {dt}|_{t=0}(\exp(t X)q),\ q \in \GG.
\end{align}
We write the resulting 
 linear 
maps  as
\begin{equation}
L \colon \gg \longrightarrow \Vect (\GG),\  L(X) =X^L,\quad
R \colon \gg \longrightarrow \Vect (\GG),\  R(X) =X^R
\label{resulting1}
\end{equation}
and, with an abuse of notation, we write the corresponding adjoints as
\begin{equation}
\begin{gathered}
L \colon \GG \times \gg \longrightarrow \TT \GG,\quad
R \colon \gg  \times \GG \longrightarrow \TT \GG\quad 
\text{or} \quad 
R \colon \GG  \times \gg \longrightarrow \TT \GG
\\
L(q,X) = q X,\ R(q,X) = Xq,\quad X \in \gg, \ q \in \GG.
\end{gathered}
\label{resulting2}
\end{equation}
Thus $X^L = X$ is the
left invariant 
and 
$X^R$ the right invariant vector field
which $X \in \gg$ generates,
and
$X_q^R = \Ad_q (X_q)$ at  the point $q$ of $\GG$.
The notation
$q X$ etc. ($X \in \gg$, $q \in \GG$) 
 is perfectly rigorous in terms of the semidirect product group
$ \gg \rtimes \GG$ (relative to the adjoint action of $\GG$ on $\gg$).
We also  write  the obvious extensions
of $L$ and $R$ to $\TT \GG$ as
\begin{equation}
L \colon \GG \times \TT \GG \longrightarrow \TT \GG,\quad
R \colon \GG  \times \TT \GG \longrightarrow \TT \GG .
\label{resulting3}
\end{equation}
These admit an obvious interpretation in terms of
the semidirect product group
$ \gg \rtimes \GG$.
Accordingly, for a homogeneous member $\beta$ of $\Lc[\gg]$,
we denote the left-invariant 
multivector field on $\GG$ it generates
by $\beta^L$ and the right-invariant
one by $\beta^R$.
We take a $\GG$-manifold to be a manifold $M$ together with
a $\GG$-action from the left.
Accordingly,
for $X \in \gg$,
 the {\em fundamental vector field\/} $X_M$ on $M$ 
which $X$ {\em generates\/}
is the vector field
\begin{equation}
X_{M,q} =\tfrac d {dt}|_{t=0}(\exp(-t X)q),\ q \in M.
\label{funda}
\end{equation}
Then the induced map $\gg \to \Vect(M)$
is a morphism of Lie algebras.
With these preparations out of the way,
the vector field $-X^R$
is the fundamental vector field
on $\GG$ which $X \in \gg$ generates relative to the left translation action
of $\GG$ on itself, and
the vector field $X^L$
is the fundamental vector field
on $\GG$ which $X \in \gg$ generates relative to the right translation action
$\GG \times \GG \to \GG$
viewed as a $\GG$-action from the left via
\begin{equation}
\GG \times \GG \longrightarrow \GG,\ (y,q) \mapsto q y^{-1}.
\end{equation} 
Accordingly, 
the vector field $\fund^\conj_\GG(X)=X^L -X^R$ is
the fundamental vector field on $\GG$
which $X \in \gg$ generates relative to the conjugation action
of $\GG$ on itself,
viewed as a $\GG$-action from the left.

\subsection{Vector bundle induced from the tangent bundle
of a Lie group as target}

Let $M$ be a $\GG$-manifold and $\Phi \colon M \to \GG$ an admissible 
map.
Consider  the vector bundle  $\tau_\Phi=\Phi^* \tau_\GG \colon \TT_\Phi \GG \to M$ 
on $M$ which $\Phi$ induces from the tangent bundle
$\tau_\GG \colon \TT \GG \to \GG$ of $\GG$.
The derivative  $d \Phi \colon \TT M \to \TT \GG$  
of $\Phi$ determines a morphism 
$(d \Phi)_M\colon  \TT M \to \TT_\Phi \GG$ of vector bundles  on $M$
such that this derivative 
factors 
as
\begin{equation}
\TT M \stackrel{(d \Phi)_M} \longrightarrow 
\TT_\Phi \GG \stackrel{\can}\longrightarrow \TT \GG. 
\end{equation}
While, in general, it does not make sense to write
$d \Phi(X)$ for a vector field $X$ on $M$,
the section 
$(d \Phi)_M(X) = (d \Phi)_M \circ X$
of $\tau_\Phi\colon \TT_\Phi \GG \to M$ makes perfect sense.
This enables us to write the replacement 
$(d \Phi)_M$ for the 
derivative of $\Phi$ in a purely algebraic fashion as
\begin{equation}
(d \Phi)_M \colon \Vect(M) \longrightarrow 
\mathcal A[M] \otimes_{\mathcal A[\GG]} \Vect(\GG). 
\end{equation}
This is the appropriate interpretation of
$(d \Phi)_M$ in the affine algebraic setting,
with 
$\mathcal A[M]= \bK[M]$ and $\mathcal A[\GG] = \bK[\GG]$
being the respective affine coordinate rings.

The operations $L,R\colon \GG \times \gg \to \TT \GG$
of left and right translation on $\TT\GG$ induce operations
$L_\Phi,R_\Phi\colon M \times \gg \to \TT_\Phi \GG$
of {\em left and right translation
for\/} $\TT_\Phi \GG$, as displayed in the commutative diagrams
\begin{equation}
\begin{gathered}
\xymatrix{
M \times \gg 
\ar@/_/[ddr]_{\pr_M}
\ar@{.>}[dr]|-{L_\Phi\phantom{}} \ar[r]^{\Phi \times \Id}& \GG \times \gg \ar[dr]^L&
\\
&\TT _\Phi \GG \ar[r]^\can\ar[d]^{\tau_\Phi} & \TT \GG\ar[d]^{\tau_\GG}
\\
&M \ar[r]_\Phi &\GG
}
\quad
\xymatrix{
M \times \gg 
\ar@/_/[ddr]_{\pr_M}
\ar@{.>}[dr]|-{R_\Phi\phantom{}} \ar[r]^{\Phi \times \Id}& \GG \times \gg \ar[dr]^R&
\\
&\TT _\Phi \GG \ar[r]^\can\ar[d]^{\tau_\Phi} & \TT \GG\ar[d]^{\tau_\GG}
\\
&M \ar[r]_\Phi &\GG ;
}
\end{gathered}
\label{display}
\end{equation}
these operations are necessarily isomorphisms of vector bundles on $M$
and hence trivialize $\tau_\Phi$.

\subsection{Product of two Lie groups}
\label{prodtg}
For later reference, we introduce some notational device:
Let $\GG^1$ and $\GG^2$ be Lie groups, and
write $\GG^ \times =\GG^1 \times \GG^2$ and
$\gg^\times = \gg^1 \oplus \gg^2$.
The total space  $\TT^2 \GG^\times$
of the second tensor square of the tangent bundle of
$\GG^\times$
decomposes as the Whitney sum
\begin{equation}
\TT^2 \GG^\times =(\TT^2 \GG^1) \times \GG^2
\oplus _{\GG^\times}
(\TT\GG^1) \otimes (\TT \GG^2)
\oplus_{\GG^\times} 
(\TT\GG^2) \otimes (\TT \GG^1)
\oplus_{\GG^\times} 
\GG^1 \times \TT^2 (\GG^2).
\label{decomp1}
\end{equation}
Let $M$ be a smooth 
$\GG^\times$-manifold.
We write
$\fund_M^1 \colon M \times \gg^1 \to \TT M$ and
$\fund_M^2 \colon M \times \gg^2\to \TT M$
for 
the fundamental vector field maps which the restrictions
of the $(\GG^1 \times \GG^2)$-action
to the first and second factor induce.
Accordingly, we use the notation
$L^1,R^1\colon \GG^1 \times \GG^2 \times \gg^1 \to \TT \GG^1 \times \TT \GG^2$
and
$L^2,R^2\colon \GG^1 \times \GG^2 \times \gg^2 \to \TT \GG^1 \times \TT \GG^2$.

Introduce the notation
\begin{align}
L^\times,R^\times &\colon \GG^\times \times \gg^\times \to \TT \GG^\times
\\
L^{1,\times} &\colon   \GG^1 \times \GG^2 \times \gg^1 \times \gg^2 
\stackrel{\pr}\longrightarrow
\GG^1 \times  \GG^2 \times \gg^1 
\stackrel{L^1} \longrightarrow \TT \GG^1 \times \TT\GG^2
\label{view1}
\\
L^{2,\times} &\colon   \GG^1 \times \GG^2 \times \gg^1 \times \gg^2 
\stackrel{\pr}\longrightarrow
\GG^1 \times  \GG^2 \times \gg^2 
\stackrel{L^2} \longrightarrow \TT \GG^1 \times \TT\GG^2
\label{view2}
\\
R^{1,\times} &\colon   \GG^1 \times \GG^2 \times \gg^1 \times \gg^2 
\stackrel{\pr}\longrightarrow
\GG^1 \times  \GG^2 \times \gg^1 
\stackrel{R^1} \longrightarrow \TT \GG^1 \times \TT\GG^2
\label{view3}
\\
R^{2,\times} &\colon   \GG^1 \times \GG^2 \times \gg^1 \times \gg^2 
\stackrel{\pr}\longrightarrow
\GG^1 \times  \GG^2 \times \gg^2 
\stackrel{R^2} \longrightarrow \TT \GG^1 \times \TT\GG^2 .
\label{view4}
\end{align}
In terms of this notation,
\begin{align}
L^\times &= L^{1,\times} + L^{2,\times}\colon \GG^1 \times \GG^2 
\times (\gg^1 \oplus \gg^2) \longrightarrow \TT \GG^1 \times \TT \GG^2
\label{Ltimes}
\\
R^\times &= R^{1,\times} + R^{2,\times}\colon \GG^1 \times \GG^2 
\times (\gg^1 \oplus \gg^2) \longrightarrow \TT \GG^1 \times \TT \GG^2 .
\label{Rtimes}
\end{align}
Let $\iota^1\colon \gg \to \gg^1 \oplus \gg^2$
and
$\iota^2\colon \gg \to \gg^1 \oplus \gg^2$
denote the injections
that identify
$\gg$ with the first and second copy of $\gg$
 in $\gg^\times = \gg^1 \oplus \gg^2$, respectively.
By \eqref{Ltimes} and \eqref{Rtimes},
\begin{align}
(L^\times + R^\times) \otimes \Id_{\gg^\times}
&=
\begin{cases}
\phantom{+}
(L^{1,\times} + R^{1,\times}) \otimes 
\iota^1
\\ 
+
(L^{2,\times} + R^{2,\times}) \otimes \iota^2 
\\
+
(L^{2,\times} + R^{2,\times}) \otimes 
\iota^1
\\ 
+
(L^{1,\times} + R^{1,\times}) \otimes \iota^2 .
\end{cases}
\label{constit1}
\end{align}

\subsection{Graded coalgebras and Hopf algebras}
\label{gracoh}
Over a general ring $\bR$ (which is supposed to be merely
an algebra over the rationals),
let $\modu$ be an $\bR$-module, and
consider the graded  tensor coalgebra
$\TTTc[\modu]$ cogenerated 
by the canonical epimorphism $\TTTc[\modu]\to \modu $.
Recall the diagonal
on its degree $n$ constituent
$\TT^{\mrc,n}[\modu] = \modu^{\otimes n}$
($n \geq 0$
is the sum of the canonical iosomorphisms
$\modu^{\otimes n} \to (\modu^{\otimes k}) \otimes (\modu^{\otimes n-k})$.
On $\TT^{\mrc,n}[\modu]$, the symmetric group
$S_n$ on $n$ letters acts in the obvious way
and by  signed transpositions
$x\otimes y \mapsto -y \otimes x$ ($x,y \in \modu$).
The invariants
 $\LAL^{\mrc,n}[\modu] = (\modu^{\otimes n})^{S_n}$
under
the second action 
constitute 
the 
homogeneous degree $n$ constituent
of the graded exterior coalgebra $\Lc[\modu]$
cogenerated
by the canonical epimorphism $\Lc[\modu]\to \modu $.
Indeed, under the canonical injection
 $\Lc[\modu] \to \TTTc[\modu]$,
the diagonal of
$\TTTc[\modu]$ induces a diagonal for
 $\Lc[\modu]$,
and the counit induces a counit
for $\Lc[\modu]$.
Moreover,
the operation of addition on $\modu$
induces multiplication maps 
on 
$\TTTc[\modu]$  and
 $\Lc[\modu]$
that turn each of these graded coalgebras  into a graded Hopf algebra,
with 
$\Lc[\modu]$ graded commutative and graded cocommutative.
For $n \geq 0$, 
 we refer to
the degree $n$ constituent
$\LAL^{\mrc,n}[\modu]$
of $\Lc[\modu]$
as the $n$th {\em exterior copower\/} of $\modu$.

We use the notation $\TT [\modu]$ and $\LAL[\modu]$
for  the ordinary graded tensor  and graded exterior
$\bR$-algebra over $\modu$, respectively,  and,
for $n \geq 0$,
we write
$\TT^n [\modu]$ and$\LAL^n[\modu]$
for the corresponding homogeneous degree $n$ constituent.
The diagonal map of $\modu$
induces diagonal maps for
$\TT [\modu]$ and $\LAL[\modu]$, 
the familiar {\em shuffle diagonal\/},
and these, together
with the canonical counits, turn
$\TT [\modu]$ and $\LAL[\modu]$
into graded $\bR$-Hopf algebras.
The universal property of the cogenerating morphism
$\mathrm{cog}\colon \Lc[\modu] \to \modu$ of $\Lc[\modu]$
and the canonical projection  $\pp_\modu\colon \LAL[\modu] \to \modu$
of graded $\bR$-modules
determine 
a canonical morphism
$\can \colon \LAL[\modu] \to \Lc[\modu]$
of graded $\bR$-coalgebras such that the composite
$\mathrm{cog} \circ \can \colon \LAL[\modu] \to \modu$
coincides with $\pp_\modu$.
Furthermore, the composite
$ \can \circ \mathrm{gen} \colon \modu \to \Lc[\modu]$
of $\can$ with 
the generating morphism
$\mathrm{gen}\colon \modu \to \LAL[\modu]$
of graded $\bR$-modules 
coincides with
 canonical injection
$\mathrm{inj}_\modu\colon \modu \to \Lc[\modu]$
of graded $\bR$-modules,
This implies that
$\can \colon \LAL[\modu] \to \Lc[\modu]$  is a
morphisms of graded 
$\bR$-Hopf algebras, necessarily an isomorphisms.
 Accordingly,
we write the multiplication operation
of $\Lc[\modu]$ as $\wedge$.

\begin{rema}
\label{aforms}
{\rm
Let $\modu$ be an $\bR$-module. 
The canonical projection $\TT[\modu] \to \LAL[\modu]$
induces an injection
$\Hom(\LAL[\modu],\bR) \to \Hom(\TT[\modu],\bR)$
which identifies
$\Hom(\LAL[\modu],\bR)$
with the graded 
$\bR$-module
$\Alt(\modu,\bR)$ of alternating $\bR$-valued multilinear forms
on $\modu$,
and the diagonal of $\LAL[\modu]$ induces 
the familiar
 multiplication
of forms on 
$\Alt(\modu,\bR)$
turning the latter into 
a graded commutative $\bR$-algebra.
When $\modu$ is finitely generated and projective,
with respect to the multiplication maps,
the canonical morphism
$\Lc[\modu^*] \to \Alt(\modu,\bR)$
of graded $\bR$-modules
is an isomorphism of graded $\bR$-algebras.
}
\end{rema}

We need an explicit description of
the second exterior copower $\LAL^{\mrc,2}[\modu^1 \oplus \modu^2]$
of the direct sum
$\modu^1 \oplus \modu^2$ of two $\bR$-modules.
The multiplication map
\begin{equation}
\wedge  \colon  
\modu^1 \otimes \modu^2 \longrightarrow 
\LAL^{\mrc,2}[\modu^1 \oplus \modu^2] \subseteq
\modu^1 \otimes \modu^2 \oplus \modu^2 \otimes \modu^1
\end{equation}
takes the form
\begin{equation}
\wedge =(\Id, -\twist) \colon  
\modu^1 \otimes \modu^2 \longrightarrow \modu^1 \otimes \modu^2 \oplus \modu^2 \otimes \modu^1
\end{equation}
and identifies 
$\modu^1 \otimes \modu^2$
with its isomorphic image
\begin{equation}
\modu^1 \boxwedge \modu^2 =\modu^1 \otimes \modu^2 \oplus \modu^2 \otimes \modu^1
\end{equation}
 in
the target, and it is entirely classical that
the second exterior copower $\LAL^{\mrc,2}[\modu^1 \oplus \modu^2]$
decomposes canonically as
\begin{equation}
\begin{aligned}
\LAL^{\mrc,2}[\modu^1 \oplus \modu^2] &= \LAL^{\mrc,2}[\modu^1]
\oplus \modu^1 \boxwedge \modu^2  \oplus \LAL^{\mrc,2}[\modu^2] 
\\
{}
&\subseteq 
\modu^1 \otimes \modu^1 \oplus
\modu^1 \otimes \modu^2 \oplus  \modu^2 \otimes \modu^1
 \oplus  \modu^2 \otimes \modu^2.
\end{aligned}
\label{canonically1}
\end{equation}

These structures extend to vector bundles in an obvious manner.
Accordingly, over a manifold $M$, for $ k \geq 0$,
we write the total spaces of the  corresponding vector bundles
arising from the tangent bundle
$\tau_M\colon \TT M \to M$ 
of $M$ as
 $\TT^k[ M]$, $\TT^{\mrc,k}[M]$,
$\LAL^k[M]$,
$\LAL^{\mrc,k}[M]$,
etc.
With the algebra 
$\mathcat A[M]$
of functions
on $M$ 
(smooth, analytic, or algebraic, as the case may be)
substituted for $\bR$,
we write the corresponding 
$\mathcat A[M]$-modules of
$k$-vector fields 
(spaces of sections of the corresponding vector bundle)
as well as as $\TT^k[ M]$, $\TT^{\mrc,k}[M]$,
$\LAL^k[M]$,
$\LAL^{\mrc,k}[M]$,
etc.,\
with a slight abuse of notation.

\subsection{Gerstenhaber brackets}
Recall a {\em Gerstenhaber bracket\/} on a graded commutative
algebra $A$
is a homogeneous bracket $\bra \colon A \otimes A \to A$
of the kind
$\bra \colon A^j \otimes A^k \to A^{j+k-1}$
($j+k-1 \geq 0$
that turns $A$, regraded down by $1$,
into a graded Lie algebra
and
is a derivation in each variable of $\bra$ in the sense that,
for a homogeneous member $\alpha$ of $A$,
the operation $[\alpha,\,\cdot\,]$ on $A$
is a graded derivation of degree $|\alpha|-1$.
Thus
\begin{align}
[\alpha,\beta]&=-(-1)^{(|\alpha|-1)(|\beta|-1)}[\beta,\alpha]
\label{ger1}
\\
[\alpha,[\beta,\gamma]]&= [[\alpha,\beta],\gamma] 
+(-1)^{(|\alpha|-1)(|\beta|-1)}[\beta,[\alpha,\gamma]]
\label{ger2}
\\
[\alpha,\beta \gamma]&=[\alpha,\beta] \gamma +(-1)^{(|\alpha|-1)|\beta|}
\beta[\alpha,\gamma] .
\label{ger3}
\end{align}
A {\em Gerstenhaber algebra\/}
is a graded commutative algebra together with a Gerstenhaber bracket.

Consider an $\bR$-Lie algebra $\gg$.
In view of \eqref{ger3},
the Lie bracket $\bra$ of $\gg$ extends to  a Gerstenhaber bracket
\begin{equation}
\bra \colon \LAL^j[\gg] \otimes \LAL^k[\gg] \to \LAL^{j+k-1}[\gg]
\quad
(j+k-1 \geq 0)
\end{equation}
on $\LAL[\gg]$.
The canonical isomorphism
 $\LAL^c[\gg] \to \LAL[\gg]$
of $\bR$-Hopf algebras
then carries this Gerstenhaber bracket
to a Gerstenhaber bracket
$\bra\colon 
\Lc[\gg] \otimes \Lc[\gg] \to \Lc[\gg]
$
on $\Lc[\gg]$ relative to its 
graded commutative
$\bR$-algebra structure.

\section{Quasi Hamiltonian structure}
\label{qhs}

The paper \cite{MR1638045} reinterpretes the extended moduli space
construction in \cite[Section 1]{MR1370113}, \cite {MR1277051}, 
\cite{MR1470732} of a twisted representation space---equivalently, that
of the corresponding moduli space---in terms of 
a quasi Hamiltonian space.

The quasi Hamiltonian setting makes sense for smooth,
analytic and algebraic manifolds.
Thus, in this section, $\GG$ denotes a group
and $\gg$ its Lie algebra,
either analytic or algebraic.
A $\GG$-manifold
is a smooth, analytic or algebraic $\GG$-manifold,
as the case may be.
Further,
$\,\form\,$ denotes an $\Ad$-invariant symmetric bilinear form
on $\gg$, not necessarily non-degenerate.
For a $\GG$-manifold $M$ and a member $X$ of $\gg$,
the notation $X_M$ refers to
the fundamental vector field on $M$ which $X$ generates,
cf. \eqref{funda}.
We denote by $\Form$  the de Rham functor and by
$\omega \in \Form^1(\GG, \gg)$
 the left invariant Maurer-Cartan form
on $\GG$. Let
\begin{align}
\lambda &=\tfrac 1 {12}[\omega,\omega]\form \omega  \in \Form^3(\GG).
\label{lambda}
\end{align}

\subsection{Weakly quasi Hamiltonian structures}
\label{wqH}
For our purposes it is convenient to
downplay the non-degeneracy constraint 
(B3) in the original definition
 \cite[Definition 2.2]{MR1638045},
see also \cite[Definition 10.1]{MR1880957}
and Subsection \ref{nondegeneracy} below,
of a  quasi Hamiltonian structure.

Let $M$ be a $\GG$-manifold and
$\sigma$  a $\GG$-invariant $2$-form on $M$.
We denote the adjoint of $\sigma$ with respect to the first variable
by $\sigma^\fflat\colon \TT M \to \TT^* M$, so that
$\sigma^\fflat(X)(Y)= \sigma(X,Y)$, for $X,Y \in  \TT M$, and
by $\TT M^{\sigma}$ the kernel of
$\sigma^\fflat$,
a distribution over $M$, not necessarily the total space of a vector bundle.

Let $\Phi \colon M \to \GG$ 
an admissible $\GG$-equivariant map and
$\,\form\,$ an $\Ad$-invariant symmetric bilinear form on $\gg$.
We define  the $2$-form $\sigma$ on $M$
to be $\Phi$-{\em quasi closed relative to\/} 
$\,\form\,$ when
\begin{align}
d\sigma &= \Phi^*\lambda .
\label{qh1}
\end{align}
Recall from the introduction that
$\mmu\colon M \to \GG$ 
is a
$\GG$-{\em momentum mapping for\/} $\sigma$
{\em relative to\/} $\,\form\,$
when it renders diagram
\eqref{momintvar} commutative.

\begin{prop}
\label{moms}
Let 
$\mmu \colon M \to \GG$ be a $\GG$-equivariant
admissible map.
For a $\GG$-invariant $2$-form $\sigma$ on $M$,
the following are equivalent.
\begin{enumerate}
\item
The map $\mmu \colon M \to \GG$
is a $\GG$-momentum mapping for $\sigma$
 relative to $\,\form\,$.

\item For any member $X$ of $\gg$,
\begin{equation}
\sigma(X_M,\,\cdot\,) 
= \tfrac 12 \mmu^*(X \form (\omega + \ovomega))).
\label{qh2var}
\end{equation}

 \item
The diagram
\begin{equation}
\begin{gathered}
\xymatrixcolsep{4pc}
\xymatrix{
\TT M \ar[d]_{-\sigma^\fflat}
 \ar[r]^{(d\mmu)_M} & \TT_\Phi \GG 
\ar[r]^{\tfrac 12 \left(L_\Phi^{-1} + R_\Phi^{-1}\right)} & M \times \gg
\ar[d]^{\Id \times \psidot}
\\
\TT^* M \ar[rr]_{\fund^*_M}
&
&
M \times \gg^* 
}
\end{gathered}
\label{PMu}
\end{equation}
is commutative.

\item
The diagram 
\begin{equation}
\begin{gathered}
\xymatrixcolsep{7pc}
\xymatrix{
\TT^* M 
& \ar[l]_{(d\mmu)^*_M}  \TT^*_\mmu \GG &
\ar[l]_{\tfrac 12\left(L^{*,-1}_\mmu + R^{*,-1}_\mmu\right)}
  M \times \gg^*
\\
\TT M\ar[u]^{\sigma^\fflat}&& \ar[ll]^{\fund_M}
M \times \gg\ar[u]_{\Id \times \psidot} 
}
\end{gathered}
\label{PMudual}
\end{equation}
is commutative.

\item
The diagram
\begin{equation}
\begin{gathered}
\xymatrixcolsep{4pc}
\xymatrix{
\TT^* M 
& \ar[l]_{(d\mmu)^*_M}  \TT^*_\mmu \GG &
\ar[l]_{L^{*,-1}_\mmu}
M \times \gg^* & \ar[l]_{\Id \times \psidot} M \times \gg
\\
\TT M\ar[u]^{\sigma^\fflat}&&& \ar[lll]^{\fund_M}
M \times \gg\ar[u]_{\tfrac 12 \Id_M \times (\Id_\gg + \Ad_\mmu^{-1} )} 
}
\end{gathered}
\label{PMudual2}
\end{equation}
is commutative. 
\end{enumerate}
\end{prop}

\begin{proof}
This is straightforward.
We only note that  the dual
$\sigma^{\fflat,*}\colon \TT M \to \TT^* M$
of $\sigma^\fflat$ coincides with $-\sigma^\fflat$.
\end{proof}

We define a  {\em weakly\/}
$\GG$-{\em quasi Hamiltonian structure\/}
on a $\GG$-manifold $M$
relative to $\,\form\,$
to consist of a $\GG$-invariant $2$-form
$\sigma$ on $M$ and
a $\GG$-equivariant map
$\mmu \colon M \to \GG$
such that $\sigma$ is $\mmu$-quasi closed relative to $\,\form\,$
and that $\mmu$ is a $\GG$-momentum mapping for
$\sigma$ relative to $\,\form\,$.
A  weakly $\cgroup$-{\em quasi Hamiltonian\/} manifold
is a  $\cgroup$-manifold $M$ together with a 
weakly $\cgroup$-{\em quasi Hamiltonian\/} structure
(relative to some $\Ad$-invariant symmetric bilinear form
on the Lie algebra $\gg$ of $\GG$).

\begin{rema}
\label{sign1}
{\rm
In \cite[Definition 2.2]{MR1638045}, 
identity \eqref{qh1} occurs with a minus sign.
In \cite[Definition 10.1]{MR1880957}, \eqref{qh1}
does not carry a minus sign
and the sign of \eqref {qh2var} is maintained. 
The present signs are 
consistent with the those for the double, see
Subsection \ref{predouble} below.
}
\end{rema}

\subsection{Non-degeneracy}
\label{nondegeneracy}
We define a $\GG$-invariant $2$-form $\sigma$ on $M$
to be $\mmu$-{\em non-degenerate\/}
relative to an admissible $\GG$-equivariant
map $\mmu\colon M \to \GG$
when the morphism
\begin{equation}
\xymatrixcolsep{4pc}
\xymatrix{
\TT M \ar[r]^{(\sigma^\fflat,(d\mmu)_M) \phantom{aaaaa}}
& \TT^* M \oplus_M \TT_\mmu \GG
}
\label{inj1}
\end{equation}
of vector bundles
on $M$ is a monomorphism
or, equivalently, when
the morphism
\begin{equation}
\xymatrixcolsep{4pc}
\xymatrix{
\TT^* M& \ar[l]_{(\sigma^\fflat,(d\mmu)^*_M) \phantom{aaaaa}}
 \TT M \oplus_M \TT^*_\mmu \GG
}
\label{surj11}
\end{equation}
of vector bundles
on $M$ is an epimorphism.

Suppose that the $\Ad$-invariant symmetric bilinear form $\,\form\,$
on $\gg$ is non-degenerate.
Let $\sigma$ be a $\GG$-invariant $2$-form on $M$ and $\mmu \colon M \to \GG$
a $\GG$-momentum mapping for $\sigma$  relative to $\,\form\,$.
Recall that
$\TT M^{\sigma}$ denotes the kernel of the adjoint
$\sigma^\fflat\colon \TT M \to \TT^* M$ of $\sigma$.
From the commutative diagrams  \eqref{PMu} and   \eqref{PMudual2},
we concoct the commutative diagram
\begin{equation}
\begin{gathered}
\scalefont{0.95}
{
\xymatrix{
&\ker(\Id \times (\Id + \Ad_\mmu^{-1} ))
\ar[r]
\ar@{>->}[d]
&\TT M^{\sigma}\ar@{>->}[d]\ar[rr] 
&
&\ker(L_\mmu^{-1}+ R_\mmu^{-1})\ar@{>->}[d] 
\\
M \times \gg 
\ar[d]^{-\Id \times \psidot}
&
\ar[l]_{\tfrac 12 \left(\Id  \times (\Id  + \Ad_\Phi^{-1})\right)} 
M \times \gg 
\ar[r]^{\fund_M} 
&\TT M \ar[d]_{-\sigma^\fflat}
 \ar[rr]^{(d\mmu)_M} 
&
&
 \TT_\mmu \GG 
\ar[d]_{\tfrac 12 (L_\mmu^{-1}+ R_\mmu^{-1})} 
\\
M \times \gg^*\ar[r]_{L^{*,-1}_\Phi}
&\TT_\Phi^* \GG
\ar[r]_{(d \mmu)_M^*}
&\TT^* M \ar[r]_{\fund^*_M}
&
M \times \gg^*
\ar[r]_{\Id \times \psidot^{-1}}
&
M \times \gg ,
}
}
\end{gathered}
\label{PMukkk}
\end{equation}
the lower right-hand rectangle being a variant of \eqref{PMu}
and the upper left-hand rectangle being commutative
since so is \eqref{PMudual2}.

\begin{prop}
\label{tech}
The $\Ad$-invariant symmetric 
$2$-form $\,\form\,$ on $\gg$ being non-degenerate,
the restriction of
$
2 L_\mmu \colon M \times \gg \longrightarrow  \TT_\mmu\GG 
$
to $\ker(\Id_M \times (\Id_\gg + \Ad_\mmu^{-1} ))$
yields the
upper row of {\em \eqref{PMukkk}}.
Hence
$\fund_M|\colon 
\ker(\Id_M \times (\Id_\gg + \Ad_\mmu^{-1} ))
\to \TT M^{\sigma}$
is a monomorphism of distributions
on $M$ and
$(d\mmu)_M|\colon 
 \TT M^{\sigma} \to \ker(L_\mmu^{-1}+ R_\mmu^{-1})$
is an epimorphism of distributions on $M$.
\end{prop}

\begin{proof}
Since $\mmu$ is $\GG$-equivariant, 
necessarily 
\begin{equation}
(d\mmu)_M\circ \fund_M=  L_\mmu - R_\mmu\colon M \times \gg \to \TT_{\mmu} \GG.
\end{equation}
This implies the claim.
\end{proof}

\begin{prop}
\label{nondegqh}
The $2$-form $\,\form\,$ on $\gg$ being non-degenerate, the 
following are equivalent.
\begin{enumerate}
\item
The morphism 
$\fund_M|\colon 
\ker(\Id_M \times (\Id_\gg + \Ad_\mmu^{-1} ))
\to \TT M^{\sigma}$
is an epimorphism of distributions on $M$.
\item
The morphism 
$\fund_M|\colon 
\ker(\Id_M \times (\Id_\gg + \Ad_\mmu^{-1} ))
\to \TT M^{\sigma}$
is an isomorphism of distributions on $M$.
\item
The intersection $\TT M^\sigma \cap \ker(d\mmu)$ is trivial.
\item
The $2$-form $\sigma$ is $\mmu$-non-degenerate.
\item
The morphism 
$d\mmu_M|\colon 
 \TT M^{\sigma} \to \ker(L_\mmu^{-1} + R_\mmu^{-1})$
is a monomorphism of distributions on $M$.
\item
The morphism $d\mmu_M|\colon 
 \TT M^{\sigma} \to \ker(L_\mmu^{-1} + R_\mmu^{-1})$
is an isomorphism of distributions on $M$.
\end{enumerate}
\end{prop}

\begin{proof}
It is immediate that 
(3), (4), and (5)  
are pairwise equivalent.
By Proposition \ref {tech},
 (1), (2), 
(5) and (6)
are pairwise equivalent.
\end{proof}

Define  a weakly $\GG$-quasi Hamiltonian structure
$(\sigma,\mmu)$ 
relative to an $\Ad$-invariant non-degenerate symmetric 
bilinear form $\,\form\,$ on $\gg$
to be {\em non-degenerate\/}
when $\sigma$ is $\mmu$-non-degenerate.
By Proposition \ref{nondegqh},
 a weakly $\GG$-quasi Hamiltonian structure
$(\sigma,\mmu)$ 
relative to  $\,\form\,$
is non-degenerate
if it enjoys one and hence each of the six equivalent properties
(1) -- (6) in that Proposition.
A $\GG$-{\em quasi Hamiltonian structure\/}
relative to an $\Ad$-invariant non-degenerate symmetric 
bilinear form $\,\form\,$ on $\gg$
is a non-degenerate weakly
$\GG$-quasi Hamiltonian structure
relative to $\,\form\,$.
A $\GG$-{\em quasi Hamiltonian\/} manifold
is a $G$-manifold $M$ together with a
$\GG$-quasi Hamiltonian structure
relative to some
 $\Ad$-invariant non-degenerate symmetric 
bilinear form $\,\form\,$ on $\gg$.

\begin{rema}
{\rm
Property (1) in Proposition \ref{nondegqh}  is 
\cite[(B3) Definition 2.2]{MR1638045}
for $\bK = \RR$
with respect to
a positive $ \Ad$-invariant  symmetric 
bilinear form on $\gg$
(and hence $\GG$ compact 
or more generally of compact type
\cite[Corollary 21.6]{MR0163331})
and
\cite[Definition 5.1, minimal degeneracy condition]{MR2642360}
for the case of a general
$\Ad$-invariant non-degenerate symmetric 
bilinear form on $\gg$; the notation in
\cite[Definition 5.1]{MR2642360}
for this form is $B$ (introduced in the Introduction
and discussed in Section 3.1).
}
\end{rema}

The following extends
\cite[Proposition 4.6]{MR1638045} to the present general case.

\begin{prop}
\label{4.6}
Let  $(M,\sigma,\mmu)$ be a $\GG$-quasi Hamiltonian manifold.
For every $\GG$-invariant $\bK$-valued admissible
function $f$ on $M$, 
there is a unique vector field 
$X_f$ on $M$ which under {\rm \eqref{inj1}}
goes to $(df, 0) \in \TT^* M \oplus \TT_\Phi\GG$.
The vector field $X_f$ is $\GG$-invariant and preserves
$\sigma$ and $\mmu$.
\end{prop}

\begin{proof}
\cite[Proposition 4.6]{MR1638045}
establishes the existence of the vector field $X_f$
associated with a $\GG$-invariant function $f$
for the case where the $\Ad$-invariant symmetric bilinear form
$\,\form\,$ on $\gg$ is positive.
A closer look reveals one can get away with
the positivity constraint.
We leave the details to the interested reader.
\end{proof}

\begin{rema}
{\rm Under the circumstances of Proposition \ref{4.6},
by Theorem \ref{existence} below,
choosing
$(X_0,\alpha) \in \TT M \oplus_M \TT_\Phi \GG$ 
such that
$df = \sigma^\fflat(X_0) + (d \mmu)_M^*(\alpha)$
(which is possible since \eqref{surj11} is an epimorphism of vector bundles on $M$)
and setting
\begin{equation}
X_f 
=X_0 + \fund_M
\left(
\tfrac 12 (L^{-1}_\Phi + R^{-1}_\Phi)(\psi_\Phi^{\form,-1}) (\alpha)
-\tfrac 14 (L^{-1}_\Phi - R^{-1}_\Phi)(d \mmu)_M(X_0)
\right)
\end{equation}
characterizes
the vector field 
$X_f$ on $M$.
}
\end{rema}

\begin{thm}
\label{unpois}
Let  $(M,\sigma,\mmu)$ be a $\GG$-quasi Hamiltonian manifold.
Setting, for two $\GG$-invariant admissible $\bK$-valued functions $f$ and $h$
on $M$,
\begin{equation}
\{f,h\}= X_f(h)= \sigma(X_h,X_f)
\label{unpois1}
\end{equation}
yields a Poisson bracket $\pbra$
on the algebra  $\mathcat A[M]^\GG$
of $\GG$-invariant admissible functions on $M$.
\end{thm}

\begin{proof}
It is immediate that
$\pbra$ is a derivation in each variable and that it is
skew. 
Moreover,
a standard calculation shows
\begin{equation} 
\{\{f,h\}, k\}  + \{\{h,k\},f\}\}+\{\{k, f\}, h\}\}
= d \sigma(X_f,X_h,X_k).
\end{equation}
Since $\sigma$ is $\mmu$-quasi closed and since, by Proposition \ref{moms},
$d\mmu(X_f)$ is zero,
\begin{equation*}
d \sigma(X_f,X_h,X_k) = \mmu^*\lambda (X_f,X_h,X_k) =0. \qedhere
\end{equation*}
\end{proof}

\subsection{Double}
\label{predouble}

As before, let $\,\form\,$ be an $\Ad$-invariant symmetric bilinear form
on the Lie algebra $\gg$ of $\GG$, not necessarily
non-degenerate.
Consider the product group $\GG \times \GG$.
We need two distinct copies of this product group.
As before, for book keeping purposes,
write the first copy of $\GG$ as $\GG^1$ and the second one as $\GG^2$;
then $\GG^1 \times \GG^2$ refers to the first copy of 
the product group to be discussed below,
and we write $\GG^\times$ for this copy for short
when there is no risk of confusion.
Next, write the first copy of $\GG$ as $\GG$ and the second one 
as $\oGG$;
then $\GG \times \oGG$ refers to the second copy of 
the product group under discussion below, and we always write this copy as
$\GG \times \oGG$.
We denote the Lie algebra of $\oGG$ 
by $\ogg$  
and
by $\,\oform\,$
the corresponding $\Ad$-invariant symmetric bilinear form on $\ogg$.
Then $\,\form\, + \,\oform\,$ is an 
 $\Ad$-invariant symmetric bilinear form on $\gg \oplus \ogg$
in an obvious way.

The actions
\begin{align}
\GG &\times (\GG^1 \times \GG^2)
\longrightarrow
\GG^1 \times \GG^2, \quad
(x,q_1,q_2) \longmapsto (x q_1 ,  q_2 x^{-1})
\label{act31}
\\
\oGG&\times (\GG^1 \times \GG^2)
\longrightarrow
\GG^1 \times \GG^2,\quad 
(y,q_1,q_2) \longmapsto (q_1 y^{-1},  y q_2)
\label{act32}
\end{align}
turn $\GG^\times$ into a  $\GG$- and into a $\oGG$-manifold.
Keeping in mind that
$\gg = \ogg = \gg^1 = \gg^2$, 
in terms of the notation $L^1$, $R^1$, $L^2$, $R^2$
introduced in Subsection \ref{prodtg},
we spell out
 the respective infinitesimal actions as
\begin{equation}
\begin{aligned}
\fund_{\GG^\times}=  L^2 - R^1
&\colon
\GG^1 \times \GG^2 \times \gg 
\longrightarrow \TT \GG^1 \times \TT \GG^2 
\\
(q_1,q_2,X)& \longmapsto \frac d {dt}\big|_{t=0}\left( \exp(-tX)q_1, q_2 \exp(tX)\right)
\\
(q_1,q_2,X)&\longmapsto  (q_1,q_2X)-(Xq_1,q_2) 
\label{mapsto1c}
\\
\ofund_{\GG^\times}= L^1 - R^2
&\colon
\GG^1 \times \GG^2 \times  \ogg
\longrightarrow \TT \GG^1 \times \TT \GG^2 .
\end{aligned}
\end{equation}
The two actions combine to an action
\begin{equation}
\begin{aligned}
(\GG \times \oGG)&\times (\GG^1 \times \GG^2)
\longrightarrow
\GG^1 \times \GG^2
\end{aligned}
\label{act3}
\end{equation}
that turns $\GG^\times$ into a
$(\GG \times \oGG)$-manifold
having the sum
\begin{equation}
\begin{aligned}
\fund_{\GG^\times} + \ofund_{\GG^\times}
&\colon
\GG^1 \times \GG^2 \times (\gg \oplus \ogg)
\longrightarrow \TT \GG^1 \times \TT \GG^2 
\end{aligned}
\end{equation}
as infinitesimal $(\gg \oplus \ogg)$-action.
Use the notation $\mult \colon \GG^1 \times \GG^2 \to \GG$
for the multiplication map of $\GG$ and let
$\omult$ denote the composite
\begin{equation}
\omult \colon \GG^1 \times \GG^2 
\stackrel{\inv \times \inv} \longrightarrow
 \GG^1 \times \GG^2 
\stackrel{\mult} \longrightarrow
\oGG.
\end{equation}

With respect to the decomposition \eqref{decomp1} of $\TT^2 \GG^\times$,
we use the notation 
\begin{equation}
L^{-1}_1 \wedge R^{-1}_2 \colon \TT^2 \GG^\times \longrightarrow 
\GG^\times \times (\gg \otimes \gg)
\end{equation}
for the  sum of
\begin{align*}
L^{-1}_1 \otimes R^{-1}_2 &\colon \TT \GG^1 \otimes \TT \GG^2 
\to \GG^\times \times (\gg^1 \otimes \gg^2)
=\GG^\times \times (\gg \otimes \gg)
\\
-R^{-1}_2 \otimes L^{-1}_1   &\colon \TT \GG ^2\otimes \TT \GG^1
\to \GG^\times \times (\gg^2 \otimes \gg^1)=\GG^\times \times (\gg \otimes \gg)
\end{align*}
and
\begin{equation}
R^{-1}_1 \wedge L^{-1}_2 \colon \TT^2 \GG^\times \longrightarrow 
\GG^\times \times (\gg \otimes \gg)
\end{equation}
for the  sum of
\begin{align*}
R^{-1}_1 \otimes L^{-1}_2 &\colon \TT \GG^1 \otimes \TT \GG^2 
\to \GG^\times \times (\gg^1 \otimes \gg^2)=\GG^\times \times (\gg \otimes \gg)
\\
-L^{-1}_2 \otimes R^{-1}_1   &\colon \TT \GG ^2\otimes \TT \GG^1
\to \GG^\times \times (\gg^2 \otimes \gg^1)=\GG^\times \times (\gg \otimes \gg) ,
\end{align*}
with the understanding that
$L^{-1}_1 \wedge R^{-1}_2$
and
$R^{-1}_1 \wedge L^{-1}_2$ are zero on the two other summands of
$\TT^2 \GG ^\times$.
To the reader, it might look more consistent to identify
$\GG^\times \times (\gg^2 \otimes \gg^1)$
with $
\GG^\times \times (\gg \otimes \gg)
$
through an additional interchange map
$\gg^2 \otimes \gg^1 \to \gg^1 \otimes \gg^2$
but, for our purposes,  this is not necessary 
since $\,\form\,$ is symmetric.

 The maps $\mult\colon \GG^1 \times \GG^2 \to \GG$ and
$\omult \colon \GG^1 \times \GG^2 \to \oGG$ 
induce the
 vector bundles
$\TT_\mult \GG \longrightarrow \GG^\times$
 and
$\TT_{\omult} \oGG \longrightarrow \GG^\times$
on $\GG^\times$
from the tangent bundles of  $\GG$ and $\oGG$, respectively,
and $(\mult,\omult) \colon \GG^\times \to \GG \times \oGG$ 
induces  the vector bundle 
$\TT _{(\mult,\omult)}(\GG\times \oGG) \to \GG^\times$
on $\GG^\times$ 
 from the tangent bundle of  $\GG \times \oGG$.
The following extends
\cite[Proposition 3.2]{MR1638045}
(for the case where $\GG$ is compact)
to the present general setting.

\begin{prop}
\label{momfusGs}
With respect to the action {\rm \eqref{act3}}
of the product group $ \GG\times \oGG$ on 
$\GG^1\times \GG^2$,
\begin{equation}
(\mult,\omult) \colon \GG^1 \times \GG^2 
\longrightarrow
\GG \times \oGG
\label{momfusGGoldd}
\end{equation}
is a $(\GG \times \oGG)$-momentum mapping relative to  
the $2$-form $\,\form\,  + \,\oform\,$ on $\gg \oplus \ogg$
for 
the  $2$-form  
\begin{equation}
\xymatrixcolsep{4.9pc}
\xymatrix{
 \sigma^\times_{\form}
 \colon\TT^2 \GG^\times
\ar[r(1.65)]^{\phantom{aaaaaaa} -\tfrac 12 
\left(L^{1,-1} \wedge R^{2,-1} + R^{1,-1} \wedge L^{2,-1}\right)}
&&\GG^\times \times (\gg \otimes \gg) 
\ar[r(0.6)]^{\phantom{aaaaaaa}\form}& 
\negthickspace
\negthickspace
\negthickspace
\negthickspace
\negthickspace
\negthickspace
\negthickspace
\negthickspace
\negthickspace
\negthickspace
\negthickspace
\negthickspace
\negthickspace
\negthickspace
\negthickspace
\negthickspace
\bK
}
\label{sigmatimes}
\end{equation}
on $\GG^\times$, 
the action of $\GG \times \oGG$ on itself being by conjugation,
and
the $2$-form  $\sigma^\times_{\form}$
is  $(\mult,\omult)$-quasi closed 
relative to  $\,\form\,  + \,\oform\,$.
When $\,\form\,$ is non-degenerate,
the $2$-form  $\sigma^\times_{\form}$ is
$(\mult,\omult)$-non-degenerate.
\end{prop}
The weakly  $(\GG \times \oGG)$-quasi Hamiltonian 
manifold $(\GG^ \times, \sigma^\times_{\form},(\mult,\omult))$ 
 relative to  
$\,\form\,  + \,\oform\,$ is the external
{\em weakly quasi Hamiltonian
 double\/} of $(\GG,\form)$,
when $\,\form\,$ is non-degenerate,
the external {\em quasi Hamiltonian
 double\/} of $(\GG,\form)$.
 
One can adapt the proof of \cite[Proposition 3.2]{MR1638045}
to the present situation.
We give a proof in the spirit of our approach.
To this end, we spell out the following:

\begin{prop}
\label{techn}
The diagram
\begin{equation}
\begin{gathered}
\scalefont{0.9}
{
\xymatrixcolsep{5.9pc}
\xymatrix{
\gg \otimes \TT \GG^\times  
\ar@/_6pc/[ddd]|-{\Id \otimes_{\GG^\times} (d \mult)_M}
\ar[d]|-{\Id \otimes_{\GG^\times} (d {\inv})} \ar[r]^{\fund_{\GG^\times}\otimes_{\GG^\times} \Id} & \TT^2 \GG^\times 
\ar[d]|-{d \inv}
\ar[r]^{-L^{1,-1} \wedge R^{2,-1} - R^{1,-1} \wedge L^{2,-1}\phantom{a}}
&
 \GG^\times\times(\gg \otimes \gg) 
 \ar[d]^{\inv \times \Id}
\\
\ogg \otimes \TT \GG^\times  
\ar[d]|-{\Id \otimes_{\GG^\times} (d {\omult})_{\GG^\times}} \ar[r]^{\ofund_{\GG^\times}\otimes_{\GG^\times} \Id} & \TT^2 \GG^\times 
\ar[r]^{-L^{1,-1} \wedge R^{2,-1} - R^{1,-1} \wedge L^{2,-1}\phantom{a}}
&
 \GG^\times\times(\gg \otimes \gg) 
 \ar[d]^{\form}
\\
 \ogg \otimes (\TT_{\omult}\GG) 
\ar[r]_{\phantom{aaa} \Id \otimes_{\GG^\times} \left(L_{\omult}^{-1} + R_{\omult}^{-1}\right)
\phantom{aaa}}
\ar[d]|-{\Id \otimes (d \inv)_M}
&(\ogg \otimes \ogg) \otimes \GG^\times \ar[r]_{\phantom{aaa}\oform} 
\ar[d]|-{\Id \otimes \inv}
&\bK \ar@{=}[d]
\\
 \gg \otimes (\TT_{\mult}\GG) 
\ar[r]_{\phantom{aaa} \Id \otimes_{\GG^\times} \left(L_{\mult}^{-1} + R_{\mult}^{-1}\right)
\phantom{aaa}}
&(\gg \otimes \gg) \otimes \GG^\times \ar[r]_{\phantom{aaa}\form} &\bK .
}
}
\end{gathered}
\label{momintomultvarr}
\end{equation}
is commutative.
\end{prop}

\begin{proof}
First we show the outermost diagram is commutative:
Let  $(q_1, q_2) \in \GG^\times = \GG_1 \times \GG_2$
and
$(X_{q_1}, X_{q_2}) \in \TT\GG_1 \times \TT \GG_2 \cong \TT \GG^\times$,
and let $X \in \gg$.
Since
\begin{equation*}
d \mult(X_{q_1}, X_{q_2})= X_{q_1} q_2  + q_1X_{q_2} \in \TT_{q_1 q_2}\GG,
\end{equation*}
necessarily
\begin{align*}
\left(L_{\mult}^{-1}  \circ d \mult\right)
(X_{q_1}, X_{q_2})&= 
\left(q_1,q_2,q_2^{-1} q_1^{-1} X_{q_1} q_2  + q_2^{-1}X_{q_2}\right)
\in \GG^\times \times \gg
\\
\left( R_{\mult}^{-1} \circ d \mult\right)
(X_{q_1}, X_{q_2})&= 
\left(q_1,q_1,
X_{q_1} q_1^{-1} + q_1 X_{q_2}q_2^{-1}q_1^{-1}\right) \in \GG^\times\times\gg.
\end{align*}
In view of \eqref{mapsto1c},
\begin{align*}
\left( \fund_{\GG^\times}\otimes_{\GG^\times} \Id\right)
(X \otimes (X_{q_1}, X_{q_2}))
&=\left (-Xq_1, q_2 X\right) \otimes \left(X_{q_1}, X_{q_2}\right) 
\\
&=
\begin{cases}
- X q_1 \otimes X_{q_1} 
+  q_2 X \otimes X_{q_1}
\\
- X q_1 \otimes X_{q_2} 
+ q_2X \otimes X_{q_2}.
\end{cases}
\end{align*}
Hence
\begin{align*}
\left(
\left(L^{1,-1} \wedge R^{2,-1}\right)\circ
\left( \fund_{\GG^\times}\otimes_{\GG^\times} \Id\right)\right)
(X \otimes (X_{q_1}, X_{q_2}))
&=
\begin{cases}
\phantom{+}\left(L^{1,-1} \otimes R^{2,-1}\right) (-X q_1 \otimes X_{q_2})
\\
-
\left(R^{2,-1} \otimes L^{1,-1}\right)(q_2 X \otimes X_{q_1})
\end{cases}
\\
&= \begin{cases}
-(q_1^{-1}X{q_1}) \otimes (X_{q_2} q_2^{-1})
\\
-
(q_2 Xq_2^{-1}) \otimes (q_1^{-1}X_{q_1})
\end{cases}
\\
\left( \left(R^{1,-1} \wedge L^{2,-1}\right)\circ
\left( \fund_{\GG^\times}\otimes_{\GG^\times} \Id\right)\right)
(X \otimes (X_{q_1}, X_{q_2}))
& =\begin{cases}
\phantom{+}
\left(R^{1,-1} \otimes L^{2,-1}\right) (-X q_1 \otimes X_{q_2})
\\
-
\left( L^{2,-1} \otimes R^{1,-1}\right)(q_2 X \otimes X_{q_1})
\end{cases}
\\
&=\begin{cases}
- X \otimes  q_2^{-1}X_{q_2}
\\
-X \otimes  X_{q_1}q_1^{-1} .
\end{cases}
\end{align*}
Since $\,\form\,$ is $\Ad$-invariant,
\begin{align*}
\begin{cases}
\phantom{+}
(q_1^{-1}X{q_1}) \form (X_{q_2} q_2^{-1})
\\
+
(q_2 Xq_2^{-1}) \form (q_1^{-1}X_{q_1})
\\
+ X \form  q_2^{-1}X_{q_2}
\\
+X \form  X_{q_1}q_1^{-1}
\end{cases}
&=\begin{cases}
\phantom{+}
X \form (q_1 X_{q_2} q_2^{-1} q_1^{-1})
\\
+
 X \form (q_2^{-1}q_1^{-1}X_{q_1}q_2)
\\
+ X \form  q_2^{-1}X_{q_2}
\\
+X \form  X_{q_1}q_1^{-1}
\end{cases}
=
\begin{cases}
\phantom{+}
X \form q_2^{-1} q_1^{-1} X_{q_1} q_2
\\
+ X \form q_2^{-1}X_{q_2}
\\
+ X \form X_{q_1} q_1^{-1} 
\\
+X \form
q_1 X_{q_2}q_2^{-1}q_1^{-1} .
\end{cases}
\end{align*}
This shows  the outermost diagram  is commutative.

Relative to the actions \eqref{act31} and \eqref{act32}, the diagram
\begin{equation*}
\begin{gathered}
\xymatrix{
\oGG \times (\GG^1 \times \GG^2) 
\ar[d]_{\Id \times (\inv \times \inv)}
\ar[r]^{\eqref{act32}} & \GG^1 \times \GG^2
\ar[d]^{\inv \times \inv}
\\
\oGG \times (\GG^1 \times \GG^2) 
\ar[r]^{\eqref{act31}} & \GG^1 \times \GG^2
}
\end{gathered}
\end{equation*}
is commutative. 
This implies that every subdiagram of
\eqref{momintomultvarr} except the innermost rectangle
is commutative.
Consequently the
innermost rectangle
is commutative as well.
\end{proof}

\begin{proof}[{Proof of Proposition {\rm \ref{momfusGs}}}]
The reader will readily verify 
that the 
map  
$(\mult,\omult)$
is $(\GG \times \oGG)$-equivariant.
The outermost diagram
of \eqref{momintomultvarr} 
being commutative says that the map
$\mult \colon \GG^\times \to \GG$
is a $\GG$-momentum mapping for
$\sigma^\times_{\form}$ relative to $\,\form\,$
and the innermost diagram
of \eqref{momintomultvarr} 
being commutative says that
$\omult \colon \GG^\times \to \oGG$
is a $\oGG$-momentum mapping for
$\sigma^\times_{\form}$ relative to $\,\oform\,$.
Consequently
\eqref{momfusGGoldd} 
is a $(\GG \times \oGG)$-momentum mapping for $\sigma^\times_{\form}$
relative to $\form + \oform$.
By construction
\begin{align*}
 \sigma^\times_{\form} &= 
-\tfrac 12 (\omega_1 \form \ovomega_2 + \ovomega_1 \form \omega_2)
\end{align*}
and, by equivariant Maurer-Cartan calculus, cf. \cite[(3.3)]{MR1362845},
\begin{align}
\tfrac 12 d(\omega_1 \form \ovomega_2)&=
\lambda_2 -\mult^* \lambda + \lambda_1
\label{cn1}
\\
\tfrac 12 d (\ovomega_1 \form \omega_2)&=
-\olambda_2 -\omult^* \olambda - \olambda_1 .
\end{align}
The tilde-notation being merely a notational device to distinguish two copies
of $\GG$, plainly
$\lambda _1 = \olambda_1 \in \GG^1$ and
$\lambda _2 = \olambda_2 \in \GG^2$.
Hence
the $2$-form  $\sigma^\times_{\form}$
is  $(\mult,\omult)$-quasi closed 
relative to  $\,\form\,  + \,\oform\,$.

When $\,\form\,$ is non-degenerate, 
the argument in the proof of
\cite[Proposition 3.2]{MR1638045}
shows that the $2$-form $\sigma^\times_{\form}$
is $(\mult,\omult)$-non-degenerate.
In view of Proposition \ref{compar32} (1) below,
the non-degeneracy claim is also a consequence of Proposition \ref{d.4}
below.
\end{proof}

\begin{rema}
\label{sign2}
{\rm
Consider  the special case where $\GG$ is compact
and the $2$-form$\,\form\,$ on $\gg$ positive.
The definition $\omega_D= 
\tfrac 12(a^* \theta, b^* \overline \theta) + \tfrac 12(a^* \overline \theta, b^* \theta)$
just before 
\cite[Proposition 3.2]{MR1638045}
yields the negative of the present $\sigma^\times$.
This is consistent with identity \eqref{qh1}
occuring in \cite[Def. 2.2 (B1)]{MR1638045}
with a minus sign, cf. Remark \ref{sign1}.
In \cite[Definition 10.1]{MR1880957}, \eqref{qh1}
does not come with a minus sign, 
what corresponds
in  \cite[Example 10.5]{MR1880957} 
 to $\omega_D$ carries a minus sign
and hence coincides with the present $\sigma^\times$,
and there is no minus sign in 
\eqref{qh2var} precisely as in the present approach.
}
\end{rema}

\subsection{Fusion} 
\label{fusq}
We extend this operation in
\cite[Section 6]{MR1638045}
for $\GG$-compact and positive $2$-form on its Lie algebra
to our general setting.

Consider the product group $\GG^\times = \GG \times \GG$
and, as before, write the first copy of $\GG$ as $\GG^1$ and
the second copy as $\GG^2$.
Let $\,\form\,$ be an $\Ad$-invariant symmetric bilinear form
on the Lie algebra $\gg$ of $\GG$, and let
$\form^{\negthinspace\times} = \form^{\negthinspace 1} 
+ \form^{\negthinspace 2}$ denote the 
corresponding $\Ad$-invariant symmetric bilinear form
on the Lie algebra $\gg^\times = \gg^1 \oplus \gg^2$ of $\GG^\times$.
The proof of \cite[Theorem 6.1]{MR1638045}, with signs adjusted,
establishes the following.
\begin{prop}
\label{fusmoms}
Let $M$ be a $\GG^\times$-manifold and
$\sigma^\times $   a $\GG^\times$-invariant $2$-form
on 
$M$. Further,  let 
$(\mmu^1,\mmu^2)\colon M \to \GG^1 \times \GG^2 =\GG^\times$ be an admissible
$\GG^\times$-equivariant map, and let
\begin{equation}
\sigma_\fus = \sigma^\times - \tfrac 12 (\mmu^1,\mmu^2)^*(\omega_1 \form \overline \omega_2) .
\label{sigmafus}
\end{equation}
\begin{enumerate}
\item
When $(\mmu^1,\mmu^2)$ is a 
$\GG^\times$-momentum mapping for $\sigma^\times $ 
relative to $\form^\times$, 
with respect to  the diagonal $\GG$-action on $M$, the product 
 $\mmu^1\mmu^2\colon M \to \GG$ is a $\GG$-momentum mapping
for $\sigma_\fus$ relative to $\,\form\,$.
\item
When $\sigma^\times$ is $(\mmu^1,\mmu^2)$-quasi closed
relative to $\form^\times$,
with respect to
the
diagonal $\GG$-action on $M$, the $2$-form
$\sigma_\fus$
is $\mmu^1 \mmu^2$-quasi closed relative to $\,\form\,$.
\item When $\sigma^\times$ is $(\mmu^1,\mmu^2)$-non-degenerate,
the $2$-form $\sigma_\fus$ is $\mmu^1\mmu^2$-non-degenerate. \qed
\end{enumerate}
\end{prop}

For illustration, 
suppose that $\sigma^\times$
is $(\mmu^1,\mmu^2)$-quasi closed. Then
\begin{align*}
d\sigma_\fus&= (\mmu^1,\mmu^2)^*(\lambda^1 + \lambda^2) - \tfrac 12 d(\mmu^1,\mmu^2)^*( \omega_1 \form \ovomega_2)
\\
&
= \mmu_1^* \lambda + \mmu_2^* \lambda - \tfrac 12(\mmu^1,\mmu^2)^*d( \omega_1 \form \ovomega_2)
\\
(\mmu^1\mmu^2)^* \lambda &= (\mmu^1,\mmu^2)^* \mult^* \lambda
= (\mmu^1,\mmu^2)^*(\lambda_1 + \lambda^2 - \tfrac 12 d ( \omega_1 \form \ovomega_2)) \ \text{by} \ \eqref{cn1}
\end{align*}
whence $\sigma_\fus$ is  $(\mmu^1\mmu^2)$-quasi closed.

\begin{rema}
{\rm
The kind of reasoning in the proof of Theorem \ref{fusmom} (1)
below yields a
 \lq\lq categorical\rq\rq\  proof
of  Proposition \ref{fusmoms} (1) as well.
}
\end{rema}

\subsection{Exponentiation}
\label{exponentiation}This operation is a crucial tool in the papers 
\cite{MR1460627}, \cite{MR1370113}, \cite{MR1670408}, \cite{MR1815112}, 
\cite {MR1277051}, \cite{MR1470732}.
Is also occurs in
 \cite[\S 3.3]{MR1638045},
\cite[\S 10 p.~23]{MR1880957}. 

Let $M$ be a $\GG$-manifold and $\mmu\colon M \to \GG$
a $\GG$-equivariant admissible map.
We refer to a point $q$ of $M$ such that $\mmu(q)$
lies in the center of $\GG$ as a $\mmu$-central point of $M$,
and we define a weakly $\GG$-quasi Hamiltonian structure
$(\sigma,\mmu)$ on $M$
to be {\em weakly non-degenerate\/}
when the $2$-form $\sigma_q$ on $\TT_q(M)$
is non-degenerate for every $q$-central point.
Since for an admissible $\GG$-equivariant map
$\mmu \colon M \to \GG$ the vector space 
$\ker(\Id + \Ad_{\mmu(q)}^{-1})$ is zero at every $\mmu$-central point 
$q$ of $M$,
a $\GG$-quasi Hamiltonian structure
 is necessarily weakly  non-degenerate.

Let $(M,\sigma,\mmu)$ be a weakly $\GG$-quasi Hamiltonian manifold.
Let $X$ be a point of the center of $\gg$ such that
$\exp(X)$ lies in the center of $\GG$ 
and such that
$\mmu^{-1}(\exp(X))$ is non-empty.
When $\GG$ is connected, for $X$ in the center
of $\gg$, the value
$\exp(X)$ necessarily lies in the center of $\GG$.
Further, when the center of $\GG$ is connected,
a point in the center of $\GG$ necessarily has
a pre-image under $\exp$.

Let $O\subseteq \gg$ be an open (in the classical topology) 
$\GG$-invariant neighborhood of $\gg$
in $X$ where the exponential mapping
from $\gg$ to $\GG$
is an analytic diffeomorphism onto its image.
Define the space
$\MH(M,\GG,\mmu)$
by requiring that
\begin{equation}
\begin{CD}
\MH(M,\GG,\mmu)
@>\mmu_O>> O
\\
@V{\eta}VV
@VV{\mathrm{exp}}V
\\
M
@>>\mmu> \GG
\end{CD}
\label{PB}
\end{equation}
be a pullback diagram,
cf.
\cite[(17) p.~744]{MR1370113} and
\cite[(5.2) p.~390]{MR1460627}; here we denote by $\eta$ and $\mmu_O$
the induced maps.
Depending on the situation,
the space
$\MH(M,\GG,\mmu)$
is a smooth or analytic $\GG$-manifold
and the induced map
$\eta$ from
$\MH(M,\GG,\mmu)$
to
$M$
is a 
$\GG$-equivariant
smooth or analytic injective codimension zero immersion whence
$\MH(M,\GG,\mmu)$ has the same dimension as $M$.

Maintaining notation in
\cite[Section 1]{MR1370113},
let $\rho=\exp^*(\lambda) \in 
\dR^2(\gg)$, let
$h$
be the (adjoint action invariant) integration operator
on $\dR^*(\gg)$
so that, in degrees $\geq 1$,
\begin{equation}
d h + hd = \Id,
\end{equation}
and let $\beta = h(\rho)$; then
$d \beta = \rho =\exp^*(\lambda) \in \dR^3(\gg)$.

The following theorem
reproduces a version of
\cite[Theorem 2 p.~748]{MR1370113}; it
summarizes a construction crucial in 
\cite{MR1460627}, \cite{MR1370113}, \cite{MR1670408}, \cite{MR1815112}, 
\cite {MR1277051}, \cite{MR1470732}.

\begin{thm}
\label{summarize}
The $2$-form
$\omega_{\sigma,\lambda} = \eta^*\sigma- \mmu_O^*\beta$
on 
$\MH(M,\GG,\mmu)$ is 
$\GG$-invariant and
closed, 
and
 the adjoint 
$\mmu^\sharp_{\psidot} \colon \gg \to \Form^0(\MH(M,\GG,\mmu))$
of the
composite
\begin{equation}
\mmu_{\psidot} \colon
\MH(M,\GG,\mmu)
\stackrel{\mmu_O} \longrightarrow O \subseteq \gg
\stackrel{\psidot}\longrightarrow  \gg^* 
\label{eq}
\end{equation}
is an  equivariantly closed extension 
of  $\omega_{\sigma,\lambda}$.
When the weakly $\GG$-quasi Hamiltonian structure $(\sigma,\mmu)$
on the $\GG$-manifold $M$ 
is weakly non-degenerate,
the $2$-form
$\omega_{\sigma,\lambda}$
is non-degenerate
at every $\mmu$-central point
of
$\MH(M,\GG,\mmu)$, 
and hence, up to sign, 
the
open $\GG$-subspace 
$\mathcal M(M,\GG,\mmu)$
of
$\MH(M,\GG,\mmu)$
where the
$2$-form 
$\omega_{\sigma,\lambda}$ is non-degenerate
together with the restrictions of 
$\omega_{\sigma,\lambda}$
and $\mmu_{\psidot}$ is an ordinary Hamiltonian $\GG$-manifold.
\end{thm}

\subsection{Weakly quasi Hamiltonian reduction}
\label{redsp}
Let $\GG^1$ and $\GG^2$ be Lie groups,
let
$\,\form^{\negthinspace 1}\,$ be an $\Ad$-invariant symmetric bilinear form
on the Lie algebra of  $\GG^1$ and
$\,\form^{\negthinspace 2}\,$  an $\Ad$-invariant symmetric bilinear form
on the Lie algebra of  $\GG^2$,
and let 
$\,\form^{\negthinspace\times} =\form^{\negthinspace 1}+\form^{\negthinspace 2}$
denote the resulting
 $\Ad$-invariant symmetric bilinear form
on the Lie algebra of the product group $\GG^1 \times \GG^2$.
Consider a weakly $(\GG^1 \times \GG^2)$-quasi Hamiltonian structure
$(\sigma,(\mmu^1,\mmu^2))$ on a $(\GG^1 \times \GG^2)$-manifold $M$
relative to $\,\form^{\negthinspace \times}\,$.
Let $y$ be a point of  $\GG^1$ 
and let $Z_y$ denote the centralizer of $y$ in $\GG^1$.
The following extends
\cite[Theorem 5.1]{MR1638045} to the present general situation.

\begin{prop}
\label{wqhr}
Suppose the pre-image  $\mmu^{1,-1}(y) \subseteq M$
of the point $y$ of $\GG^1$
is a smooth, analytic, or affine  algebraic (as the case may be)
submanifold of $M$
and suppose the orbit space $\mmu^{1,-1}(y)/Z_y$
is, accordingly, a  smooth, analytic, or affine algebraic manifold.
Then the restriction to
$\mmu^{1,-1}(y)$
of the $2$-form $\sigma$ descends to a
$2$-form $\sigma_\red$ on
$\mmu^{1,-1}(y)/Z_y$,
and  $\mmu^2$ 
induces a map
$\mmu^2_\red \colon \mmu^{1,-1}(y)/Z_y \to \GG^2$
in such a way that
$( \sigma_\red, \mmu^2_\red)$
is a weakly $\GG$-quasi hamiltonian structure 
on  $\mmu^{1,-1}(y)/Z_y $ relative to $\,\form^{\negthinspace 2}\,$.
When  
$\sigma$ is $(\mmu^1,\mmu^2)$-quasi non-degenerate,
i.e., $(\sigma,(\mmu^1,\mmu^2))$ is a genuine
 $(\GG^1 \times \GG^2)$-quasi Hamiltonian structure on $M$,
the $2$-form $\sigma_\red$
is
$\mmu^2_\red$-quasi non-degenerate,
that is,
$(\sigma_\red,\mmu_\red^2)$ is a genuine
 $\GG^2$-quasi Hamiltonian structure
on $\mmu^{1,-1}(y)/Z_y $.
\end{prop}

\begin{proof}
The proof of \cite[Theorem 5.1]{MR1638045} 
carries over.
\end{proof}

Under the circumstances of Proposition \ref{wqhr},
we use the notation
$M_{y,\red}^1=\mmu^{1,-1}(y)/Z_y $ and,  the roles of
$\mmu^1$  and $\mmu^2$ being interchanged, we also write
$M_{y,\red}^2=\mmu^{2,-1}(y)/Z_y $ with respect to $y \in \GG^2$, and we refer to each of these spaces as the corresponding
{\em reduced space\/}.
We refer to the passage from
$(M,\sigma,(\mmu^1,\mmu^2))$ 
to
$(M_{y,\red}^1, \sigma_\red, \mmu^2_\red)$
(to 
$(M_{y,\red}^2, \sigma_\red, \mmu^1_\red)$)
as {\em (weakly) quasi Hamiltonian reduction\/}
relative to $y$ with respect to $\mmu^1$ (to $\mmu^2$).

\begin{rema}
{\rm 
Under the circumstances of Proposition \ref{wqhr},
we can also write the reduced space as the $\GG$-orbit space
$\mmu^{1,-1}(\CcC_y)/\GG $
of the pre-image $\mmu^{1,-1}(\CcC_y)$
in $M$ of the conjugacy class $\CcC_y$ in $\GG^1$
which the point $y$ of $\GG^1$ generates. 
}
\end{rema}

Consider the special case where
 $\GG^2$ is the trivial group, and write
$\GG = \GG^1$ and $\mmu = \mmu^1$. Then
the $2$-form $\sigma_\red$ is necessarily closed.
Suppose, furthermore, that $y$ is in the center of $\GG = \GG^1$
and suppose $\GG$ compact, so that we are working over the reals,
and that $\sigma$ is $\mmu$-quasi non-degenerate. 
Then, cf.
\cite[Theorem 5.1]{MR1638045},
when $y$ is a regular value of $\mmu$,
the ordinary orbit space
$\mmu^{-1}(y)/\GG$
acquires the structure of a symplectic orbifold.

Let, furthermore, $X \in \gg$ such that $\exp(X) = y$, and
consider the resulting 
ordinary Hamiltonian $\GG$-manifold
$\left(
\mathcal M(M,\GG,\mmu),\omega_{\sigma,\lambda},\mmu_{\psidot}
\right)$
in Theorem \ref{summarize}.
It is immediate that
the map $\eta \colon\mathcal M(M,\GG,\mmu) \to M$, cf. \eqref{PB},
then determines an identification of the reduced spaces
as
symplectic orbifolds.

\begin{examp}
\label{conjwqhr}
{\rm The following extends
\cite[Example 5.1]{MR1638045}: Relative to 
 an $\Ad$-invariant symmetric bilinear form $\,\form\,$ 
on the Lie algebra of the Lie group  $\GG$,
consider the external weakly quasi Hamiltonian
 double $(\GG^ \times, \sigma^\times_{\form},(\mult,\omult))$ 
 of $(\GG,\form)$ relative to  
$\,\form\,  + \,\oform\,$, cf. Proposition \ref{momfusGs}.
Let $q$ be a point of $\GG = \oGG$.
The regularity constraints automatically hold, and
weakly quasi Hamiltonian reduction relative to $q$
with respect to each of
$\mult$ and $\omult$ yields  
the conjugacy class in $\GG$ which $q^{-1}$ generates.
More precisely,
consider the conjugacy class $\CcC_{q^{-1}} \subseteq \GG$
of $q^{-1}$ in $\GG$,
use the notation $\widetilde {\CcC}_{q^{-1}} \subseteq \oGG$ for the
very same conjugacy class as well,
and write the
inclusions as $\iota \colon \CcC_{q^{-1}} \to \GG$ 
and $\widetilde \iota \colon \widetilde \CcC_{q^{-1}} \to \oGG$;
further, let
$(\GG^\times_{q,\red}, \sigma_\red, {\omult}_\red)$
denote  the reduced space with respect to $\mult\colon \GG^\times \to \GG$
and
$(\oGG^\times_{q,\red}, \sigma_\red, \mult_\red)$
that with respect to $\omult \colon \GG^\times \to \oGG$,
with a slight abuse of the notation 
$\sigma_\red$.
The maps $\mult \colon \GG^\times \to \GG$
and $\omult \colon \GG^\times \to \oGG$
induce identifications
\begin{align}
\mult^\sharp \colon&(\oGG^\times_{q,\red}, \sigma_\red, \mult_\red) \longrightarrow
(\CcC_{q^{-1}}, \tau_{\CcC_{q^{-1}}}, \iota)
\\
\omult^\sharp \colon&(\GG^\times_{q,\red}, \sigma_\red, \omult_\red) \longrightarrow
(\widetilde \CcC_{q^{-1}}, \tau_{\widetilde \CcC_{q^{-1}}}, \widetilde\iota)
\end{align}
of weakly quasi Hamiltonian spaces.
}
\end{examp}

\subsection{Comparison with the extended moduli space formalism}
\label{comparex}

I am not aware of an explicit comparison in the literature
of
the extended moduli space approach
\cite{MR1460627}, \cite{MR1370113}, \cite {MR1277051},  
\cite{MR1670408}, 
\cite{MR1815112}, 
\cite{MR1470732}
with the quasi Hamiltonian approach \cite{MR1638045}
to the construction of moduli spaces.
The present Subsection offers such a comparison.
I am indebted to a referee for having requested such a comparison.
In \cite{MR1638045}, there is only a comparison 
of the quasi Hamiltonian approach
with the gauge theory approach
(valid for compact structure group).
To carry out the comparison, we now show how
the quasi Hamiltonian formalism 
straightforwardly
results from the approach in
\cite{MR1460627}, \cite[Section 1]{MR1370113}, \cite {MR1277051}, 
\cite{MR1470732}.
The notation 
$((\,\cdot\, , \,\cdot\,),\theta,\chi,\ovomega,\xi,\omega)$ 
in \cite{MR1638045} corresponds to
$(\,\bullet\,,\omega,\lambda,\beta,X,\tau)$ in the present paper.

\subsubsection{Forms on a product of finitely many copies of $\GG$}
\label{forms}

Let $F$ denote a finitely generated free (discrete) group. Evaluation
\begin{align*}
E &\colon \Pii^2  \times \Hom(F,\GG)
\longrightarrow \GG^2,
\end{align*}
induces a linear map
$E^*\colon\Form^2(\GG^2)
\to C^2(\Pii) \otimes \Form^2 (\Hom(F,\GG))$,
and  pairing with $2$-chains in $C_2(F)$  yields a linear map
\begin{equation}
\hinn  \colon C_2(\Pii) \otimes C^2(\Pii) \otimes  \Form^2(\Hom(F,\GG))
\longrightarrow  \Form^2 (\Hom(F,\GG))
\end{equation}
and hence the pairing
\begin{align*}
 C_2(\Pii) \otimes \Form^2(\GG^2)
&\longrightarrow  \Form^2 (\Hom(F,\GG)),
\ (c,\alpha) \mapsto \langle c, E^* \alpha\rangle .
\end{align*}

\subsubsection{Quasi Hamiltonian structure preceeding its explicit recognition}
\label{recognition}

Return to the circumstances of Section \ref{reps}
with $n=0$ and retain the notation established there.
By \cite[Lemma 2 p.~746]{MR1370113},
there is a $2$-chain $c \in C_2(\Pii)$
having boundary
\begin{equation}
\partial c = [r] \in C_1(\Pii).
\label{bd1}
\end{equation}
See, e.g., \eqref{twoc} below.
Let
$\ovomega \in \Form^1(\GG, \gg)$
denote the  right invariant Maurer-Cartan form on $\GG$.
For a differential form $\alpha$ on $\GG$, 
for $j = 1,2$, we write
as $\alpha_j$ 
 the differential form on $\GG \times \GG$
that arises from the projection $\GG \times \GG \to \GG$
to the $j$'th component.
Let
$\omega_c = \langle c, \tfrac 12  E^* (\omega_1 \ovomega_2)\rangle$, 
by construction
a $2$-form on $\Hom(F,\GG)$, and hence,
under the identification
$\Hom(F,\GG) \to \GG^{2\ell}$ which the choice of generators
$x_1,y_1,\dots,x_\ell,y_\ell$ of $\Pii$ induces, a $2$-form
on $\GG^{2\ell}$.

The following summarizes
the reasoning in
\cite{MR1370113}, 
\cite {MR1277051}, \cite{MR1470732},
see
\cite[(18) p.~ 747]{MR1370113},
\cite[Theorem 2 p.~ 748]{MR1370113},
\cite[Lemma 1 p.~246]{MR1277051},
\cite[Theorem 3 p.~247]{MR1277051}
but, of course, the terminology 
\lq weakly quasi Hamiltonian\rq\ 
was not in use when these papers were written.

\begin{prop}
\label{prop1l}
The $2$-form $\omega_c$ on
$\GG^{2 \ell}$ and the word map
$r\colon \GG^{2 \ell} \to \GG$
constitute a weakly $\GG$-quasi Hamiltonian structure on
$\GG^{2 \ell}$ relative to $\,\form\,$. \qed
\end{prop}

The following reproduces
\cite[Corollary 6.3 p.~ 393]{MR1460627},
\cite[Proposition 3.1]{MR1638045} for the case where $\GG$ is compact
and $\,\form\,$ positive;
see also \cite[Example 5.5]{MR2642360}.

\begin{prop}
\label{prop1c}
For a conjugacy class $\CcC$ in $\GG$,
 the $2$-form $\tau_\CcC$ on $\CcC$
which diagram \eqref{momintvar} with $(\CcC, \tau_\CcC)$ 
substituted for $(M,\sigma)$ characterizes
and the inclusion
$\iota \colon \CcC \subseteq \GG$ constitute a weakly $\GG$-quasi Hamiltonian
structure on $\CcC$ relative to $\,\form\,$. 
\end{prop}

\begin{proof}
It is immediate that the inclusion is a $\GG$-momentum mapping 
relative to $\,\form\,$.
A calculation shows that $\tau_\CcC$ is $\iota$-quasi closed.
\end{proof}

Now we 
return to the circumstances of Section \ref{reps}
with general $n\geq 0$ and retain the notation established there.
Choose a $2$-chain $c \in C_2(\Pii)$
having boundary
\begin{equation}
\partial c = [r] - [z_1] - \ldots - [z_n] \in C_1(\Pii),
\label{5.4}
\end{equation}
cf. \cite[(5.4)~p.~981]{MR1460627}.
As for the existence of $c$, see
the reasoning after
 \cite[(5.4)~ p.~ 981]{MR1460627}.

Recall the choice of $n$ conjugacy classes
$\{\CcC_1,  \ldots,  \CcC_n\}$ in $\GG$ made
in  Section \ref{reps}.
The choice of generators
$x_1,y_1,\dots,x_\ell,y_\ell, z_1,\ldots,z_n$ of $\Pii$ 
induces
an  identification
\begin{equation}
\Hom(F,\GG)_{\mathbf C} \longrightarrow \GG^{2\ell}\times \CcC_1 \times \ldots \times \CcC_n.
\end{equation}
Thus the restriction of the $2$-form 
$\langle c, E^* \Omega\rangle$
on $\Hom(\Pii,\GG)$
to 
$\Hom(F,\GG)_{\mathbf C}$ induces
a $2$-form
$\omega_c$
on $\GG^{2\ell}\times \CcC_1 \times \ldots \times \CcC_n $.
We denote 
the restriction to $\GG^{2\ell}\times \CcC_1 \times \ldots \times \CcC_n$
of the word map $r \colon \GG^{2\ell+n} \to \GG$
by $r \colon \GG^{2\ell}\times \CcC_1 \times \ldots \times \CcC_n  \to \GG$
as well and, for
$1 \leq j \leq n$, 
we interpret the projection
from $\GG^{2\ell}\times \CcC_1 \times \ldots \times \CcC_n$  to $\CcC_j$
as the word map $z_j \colon
\GG^{2\ell}\times \CcC_1 \times \ldots \times \CcC_n  \to \CcC_j$.
The following summarizes  
 \cite[(5.6) p.~391, (6.3.1) p.~393, Theorem 7.1 p.~396]{MR1460627}:

\begin{prop}
\label{prop12}
The $2$-form $\omega_c + z_1^* \tau _1 + \ldots +  z_n^* \tau_n$ 
on $\GG^{2\ell}\times \CcC_1 \times \ldots \times \CcC_n $
and the word map 
$
r \colon 
\GG^{2\ell}\times \CcC_1 \times \ldots \times \CcC_n 
\longrightarrow
\GG
$
constitute a weakly $\GG$-quasi Hamiltonian structure on
 $\GG^{2\ell}\times \CcC_1 \times \ldots \times \CcC_n $
relative to $\,\form\,$. \qed
\end{prop}

\subsubsection{Extended moduli space and twisted 
representation spaces}
\label{extendedm}
Return to the situation of Subsection \ref{recognition}
for the special case $n=0$.
With regard to the presentation \eqref{standpre2}
of the fundamental group $\pi$
of an orientable  closed surface, via the Schur-Hopf formula,
the relator $r$ determines
a homology class
$[r]$ in the infinite cyclic group $\Ho_2(\pi,\ZZ) \cong \Ho_2(\Sigma)$,
by construction a generator.
Consider 
an $r$-central point 
$\varphi$ of $\GG^{2\ell}\cong \Hom(\Pii,\GG)$.
The adjoint action of $\GG$ then induces a
$\pi$-module structure on $\gg$, and we write
the resulting $\pi$-module as $\gg_\varphi$.
The bilinear form
$\,\form\,$ on $\gg$ and the homology class $[r]$
determine the alternating bilinear form
\begin{equation}
\omega_{[r],\form,\varphi}\colon
\Ho^1(\pi,\gg_\varphi)
\otimes
\Ho^1(\pi,\gg_\varphi)
\stackrel{\cup}\longrightarrow
\Ho^2(\pi,\bK)
\stackrel{\cap [r]}
\longrightarrow
\bK
\end{equation}
on 
$\Ho^1(\pi,\gg_\varphi)$.
Diagram 
 \cite[(4.2) p.~749]{MR1370113}
identifies the 
cochain complex
$(\mathbf C_\varphi,\delta)$
that underlies the resulting momentum complex
$(\mathbf C_\varphi,\delta, \omega_{c,\varphi})$
with the familiar small cochain complex
computing the group cohomology $\Ho^*(\pi,\gg_\varphi)$
of $\pi$ with coefficients in $\gg_\varphi$.
By \cite[Theorem 4 p.~750]{MR1370113},
under this identification, the 2-form $[\omega_{c,\varphi}]$
on $\Ho^1(\mathbf C_\varphi,\delta)$,
cf. Proposition {classical} (4 (c)),
corresponds to
$\omega_{[r],\form,\varphi}$.
When $\,\form\,$ is non-degenerate,
the alternating bilinear form
$\omega_{[r],\form,\varphi}$ 
on 
$\Ho^1(\pi,\gg_\varphi)$ is non-degenerate by Poincar\'e duality
(in the cohomology of $\pi$).

Let $X$ be a point of the center
of $\gg$ such that $\exp(X)$ lies in the center of $\GG$
and such that $r^{-1}(\exp(X))$ is non-empty.
By Proposition \ref{summarize},
applying exponentiation to the weakly $\GG$-quasi Hamiltonian 
structure $(\omega_c,r)$
on the
$\GG$-manifold
$\GG^{2\ell}$
in Proposition \ref{prop1l}
yields the $\GG$-manifold $\MH(\GG^{2\ell},\GG,r)$
together with the $\GG$-invariant
$2$-form 
$\omega_{c,\lambda}= \eta^* \omega_c - r_O^*\beta$
on 
$\MH(\GG^{2\ell},\GG,r)$
and its equivariantly closed extension
$r_{{\psidot}}\colon \MH(\GG^{2\ell},\GG,r) \to \gg^*$.

Suppose $\,\form\,$ non-degenerate.
Then $\omega_{c}$ is non-degenerate
at every $r$-central point $\varphi$
of $\Hom(F,\GG) \cong\GG^{2\ell}$
and $\omega_{c,\lambda}$ is non-degenerate
at every $r$-central point $\widehat \varphi$
of $\MH(\GG^{2\ell},\GG,r)$.
The open subspace 
$\mathcal M(\GG^{2\ell},\GG,r)$
of
$\MH(\GG^{2\ell},\GG,r)$
where the
$2$-form 
$\omega_{c,\lambda}$ is non-degenerate, together with the restrictions
of
$\omega_{c,\lambda}$
and
$\mmu_{c,\lambda}$ to $\mathcal M(\GG^{2\ell},\GG,r)$,
cf. Theorem \ref{summarize},
is the extended moduli space
in \cite[Section 5 p.~752/53]{MR1370113},
by construction an ordinary Hamiltonian $\GG$-manifold, written there
as
\begin{equation}
(\mathcal M(\mathcal P,\GG),\omega_{c,\mathcal P},\mu).
\label{exten2}
\end{equation}
In particular, 
$\omega_{c,\mathcal P}$ is an ordinary symplectic structure.

In the same vein,
with regard to the presentation \eqref{standpre2}
of the fundamental group $\pi$
of a compact surface with $n$  boundary circles
and to the choice of $n$ conjugacy classes
$\{\CcC_1,  \ldots,  \CcC_n\}$ in $\GG$,
cf.  Section \ref{reps},
applying the same kind of reasoning to the non-degenerate weakly $\GG$-quasi
Hamiltonian structure $(\omega_c + z_1^* \tau_1 + \ldots z_n^*\tau_n,r)$
on the $\GG$-manifold
$\GG^{2\ell} \times \CcC_1 \times \ldots \times \CcC_n$
in Proposition \ref{prop12},
we arrive at
the ordinary  Hamiltonian $\GG$-manifold
\begin{equation}
\left (\mathcal M(\GG^{2\ell}\times \CcC_1 \times \ldots \times \CcC_n ,\GG,r),
\eta^*(\omega_c + z_1^* \tau_1 + \ldots z_n^*) - r_O^*\beta, r_{\psidot}\right).
\end{equation}
This is the extended moduli space 
in \cite[Theorem 8.12 p.~402]{MR1460627}, written there as
\begin{equation}
(\mathcal M(\mathcal P,\GG)_{\mathbf C},\omega_{c,\mathcal P,\mathbf C},\mu).
\label{exten3}
\end{equation}
In particular, 
$\omega_{c,\mathcal P,\mathbf C}$ is an ordinary symplectic structure.

\subsubsection{Comparison in the torus case via the
internally fused double}

Apply fusion to
the external weakly quasi Hamiltonian double
$(\GG^\times,\sigma^\times_{\form}, (\mult,\omult))$
of $(\GG,\form)$, cf.  Subsection \ref{predouble}.
This  yields, with respect to diagonalwise conjugation,
the weakly $\GG$-quasi Hamiltonian structure 
\begin{equation}
(\sigma_1,\mmu_1)
 = \left( \sigma^\times_{\form}  -\tfrac 12 (\mult,\omult)^*(\omega_1 \form \ovomega_2),
\mult\cdot\omult \colon \GG \times \GG \to \GG
\right)
\end{equation}
on $\GG \times \GG$
relative to $\,\form\,$.
The pieces of structure $\sigma_1$ and $\mmu_1$ yield 
the {\em internally fused double\/}
of $\GG$
in the realm of weakly quasi Hamiltonian spaces,
see {\rm (\cite[Example 6.1]{MR1638045})}
for the case where $\GG$ is compact and $\,\form\,$ positive.

Consider the standard presentation
$\mathcal P = \langle x,y; r\rangle$
($r = [x,y]$)
of the
fundamental group $\pi$ of a (real) torus.
The word map $r\colon \GG \times \GG \to \GG$
which $r$ induces coincides with $\mmu_1$ and
\begin{align*}
\omega_{[x|y]}&=\tfrac 12 \omega_1 \form \ovomega_2,
\\
 \omega_{[x^{-1}|y^{-1}]} &=\tfrac 12 \ovomega_1 \form \omega_2,
\\
\omega_{[xy|x^{-1}y^{-1}]}&=
\tfrac 12(\mult,\omult)^*( \omega_1 \form \ovomega_2),
\\
\sigma^\times_{\form}&= 
-\tfrac 12 (\omega_1 \form \ovomega_2 + \ovomega_1 \form \omega_2),
\\
\sigma_1 &=
\sigma^\times_{\form} -\tfrac 12(\mult,\omult)^*( \omega_1 \form \ovomega_2= 
-\omega_{[x|y]} -\omega_{[x^{-1}|y^{-1}]} -\omega_{[xy|x^{-1}y^{-1}]} .
\end{align*}
The $2$-chain 
\begin{equation}
c= -[x|y]-[x^{-1}|y^{-1}] -[xy|x^{-1}y^{-1}] + [x|x^{-1}] + [y|y^{-1}]
\in C_2(F)
\label{twoc}
\end{equation}
has boundary
$\partial(c) = [xyx^{-1} y^{-1}]$,
i.e., satisfies \eqref{bd1} in the case at hand,
and $\omega_{[x|x^{-1}]} = 0 = \omega_{[y|y^{-1}]}$.
Consequently the $2$-forms
$
\omega_c
$ 
and
$\sigma_1$ on $\GG \times \GG$
coincide.
This choice of $c$ renders the $2$-chain $c\in C_2(\Pii)$ in
\cite[Lemma 2 p.~746]{MR1370113}
explicit, cf. Subsection \ref{recognition}.
Thus the 
internally fused double structure, i.e., 
weakly $\GG$-quasi Hamiltonian structure
$(\sigma_1,\mmu_1)$ on $\GG \times \GG$,
coincides with the
weakly $\GG$-quasi Hamiltonian structure
$(\omega_c,r)$ on $\GG \times \GG$
given in Proposition \ref{prop1l}, for the special case $\ell = 1$.
Hence:

\begin{concl}
\label{concl1}
Suppose
the $2$-form  $\form\,$ on $\gg$ is non-degenerate.
Then applying
 exponentiation, cf. Subsection {\rm \ref{exponentiation}}, 
 to 
 the weakly non-de\-ge\-nerate
weakly $\GG$-quasi Hamiltonian
$\GG$-manifold
$\left(\GG \times \GG,\sigma_1,\mmu_1\right)$
yields
the extended moduli space 
{\rm \eqref{exten2}}
in
\cite[Section 5 p.~752]{MR1370113} 
for the genus $1$ case
relative to a suitably chosen $2$-chain $c \in C_2(F)$.
\end{concl}

\subsubsection{Comparison in the general case}
Consider the presentation
\eqref{standpre2}
of the fundamental group $\pi$
of an orientable  closed surface of genus $\ell \geq 1$ so that,
in particular, $n=0$.
Fusing $\ell$ copies of $(\GG \times \GG, \sigma_1,\mmu_1)$
yields the weakly $\GG$-quasi Hamiltonian structure
$(\sigma_\ell,\mmu_\ell)$ on $\GG^{2 \ell}$, cf.
\cite[Section 9.3]{MR1638045}. 
As in the previous subsection,
for a suitable choice of
the $2$-chain $c\in C_2(\Pii)$ 
with 
$\partial(c)=[r]$,
the weakly $\GG$-quasi Hamiltonian structure
$(\sigma_\ell,\mmu_\ell)$ on $\GG^{2 \ell}$
coincides with the
the weakly $\GG$-quasi Hamiltonian structure
$(\omega_c,r)$
given in Proposition \ref{prop1l},
both weakly non-degenerate 
(even non-degenerate but this is not important here)
when so is 
the $2$-form  $\,\form\,$ on $\gg$.
Hence:

\begin{concl}
\label{concl2}
Suppose
the $2$-form  $\,\form\,$ on $\gg$ non-degenerate.
Then applying
 exponentiation in the sense of Subsection {\rm \ref{exponentiation}} 
 to 
 the weakly non-de\-ge\-nerate
weakly $\GG$-quasi Hamiltonian
$\GG$-manifold
$\left(\GG^{2\ell},\sigma_\ell,\mmu_\ell\right)$
yields
the extended moduli space 
{\rm \eqref{exten2}}
in
\cite[Section 5 p.~752]{MR1370113} 
for the genus $\ell$ case
relative to a suitably chosen $2$-chain $c \in C_2(F)$.
\end{concl}

In the same vein,
under the circumstances of Proposition \ref{prop12},
with regard to the presentation
\eqref{standpre2} and the conjugacy classes $\CcC_1$, ... , $\CcC_n$ in $\GG$,
including in the fusion process
the conjugacy classes as well,
for a suitable choice of
the $2$-chain $c\in C_2(\Pii)$ 
with 
$\partial(c)=[r]$,
the resulting weakly $\GG$-quasi Hamiltonian structure
$(\sigma_{\ell,n},\mmu_{\ell,n})$ on $\GG^{2 \ell} \times \CcC_1 \times \ldots \times \CcC_n$
coincides with
the weakly $\GG$-quasi Hamiltonian structure
$(\omega_c,r)$
given in Proposition \ref{prop12},
both non-degenerate when so is 
the $2$-form  $\,\form\,$ on $\gg$.
Hence:

\begin{concl}
\label{concl3}
Suppose
the $2$-form  $\,\form\,$ on $\gg$ is non-degenerate.
Then applying
 exponentiation in the sense of Subsection {\rm \ref{exponentiation}} 
 to 
 the weakly non-de\-ge\-nerate
weakly $\GG$-quasi Hamiltonian
$\GG$-manifold
$\left(\GG^{2\ell}\times \CcC_1 \times \ldots \times \CcC_n,\sigma_{\ell,n},\mmu_{\ell,n}\right)$
yields
the extended moduli space 
{\rm \eqref{exten3}}
in
 \cite[Theorem 8.12 p.~402]{MR1460627}
for the genus $\ell$ case with $n$ boundary circles
relative to a suitably chosen $2$-chain $c \in C_2(F)$.
\end{concl}

\begin{rema}
\label{geneq}
{\rm
Under the circumstances of
\eqref{concl1}, \eqref{concl2},
\eqref{concl3},
when we carry out 
the construction  
in 
Proposition {\rm \ref{prop1l}} or
Proposition {\rm \ref{prop12}}
of
the requisite 
weakly $\GG$-quasi Hamiltonian structure
with a general $2$-chain $\widetilde c \in C_2(F)$
subject to \eqref{bd1},
viewed as a $2$-chain 
in 
$C_2(\pi)$,
this $2$-chain is homologous to
a $2$-chain of the kind $c \in C_2(F)$
\eqref{concl1}, \eqref{concl2},
\eqref{concl3}, 
viewed as a $2$-chain 
in 
$C_2(\pi)$.
By \cite[(16) p.~745]{MR1370113},
the restrictions 
of the $2$-forms
$\omega_{c,\mathcal P}$
and
$\omega_{\widetilde c,\mathcal P}$
or, as the case may be,
$\omega_{c,\mathcal P,\mathbf C}$ and
$\omega_{\widetilde c,\mathcal P,\mathbf C}$,
to the preimage
$\mmu^{-1}(X)$ for suitable $X \in \gg$,
coincide. Hence
the structures on the reduced level
coincide.
In this sense,
the weakly quasi Hamiltonian approach is equivalent
to the extended moduli space approach, 
whatever choice of the $2$-chain $c\in C_2(F)$
subject to \eqref{bd1} or \eqref{5.4}.

}
\end{rema}

\begin{rema}
{\rm
In the quasi Hamiltonian picture,
the relationship with
the cohomology of the fundamental group
of the underlying surface
with or without peripheral structure, as the case may be,
cf. Subsection \ref{extendedm}
for the case where the underlying surface is closed
and  \cite[Section 3 p.~385 ff]{MR1460627} for the general case,
 is only visible via an observation of the kind spelled out in
Theorem \ref{visible} below.
}
\end{rema}

\begin{rema}
\label{qhreduction}
Symplectically
 reducing
an extended moduli space of the kind {\rm \eqref{exten2}}
or {\eqref{exten3}}, 
the quotient space being suitably defined when $\GG$ is not compact,
e.g., as a categorical quotient,
yields 
a stratified symplectic structure in the sense of
\cite{MR1127479} (a Poisson structure which on each stratum restricts to
a symplectic Poisson structure)
 on a
twisted representation space
of the kind $\mathrm{Rep}_X(\Gamma,\GG)$
\cite[Section 6 p.~754]{MR1370113}
and $\mathrm{Rep}(\pi,\GG)_{\mathbf C}$
 \cite[Theorem 9.1 p.~403]{MR1460627} 
respectively.
This includes not necessarily non-singular 
moduli spaces of semistable vector bundles
over a complex curve,  with or without parabolic structure.
At present, there is no machinery available
that explains such a stratified symplectic structure
in general directly in terms of quasi Hamiltonian reduction.
Quasi Hamiltonian reduction relative to a compact
group as developed in
\cite[Section 5]{MR1638045}, cf. Subsection {\rm \ref{redsp}},
yields a globally defined  reduced space 
only in the regular case. 
\end{rema}
 
\begin{rema}
{\rm
Suppose $\,\form\,$ non-degenerate.
Then the weakly quasi Hamiltonian structures in
Propositions \ref{prop1l},
\ref{prop1c}, and \ref{prop12}
are non-degenerate, not just weakly non-degenerate, i.e.,
quasi Hamiltonian.
The argument in \cite{MR1638045}
for the case where $\GG$ is compact and $\,\form\,$ positive
carry over to the general case.
For an alternate argument, see
Proposition \ref{alternate} below.
}
\end{rema}

\subsubsection{Alternate quasi Hamiltonian approach to moduli spaces}
\label{alternateq}
This approach starts from a $\GG$-manifold
of the kind $\GG^{2\ell} \times  \GG^n$ rather than one of the
kind
 $\GG^{2\ell} \times \CcC_1 \times \ldots  \times \CcC_1$
for conjugacy classes $\CcC_1, \ldots,  \CcC_n$ in $\GG$; cf. 
Conclusion \ref{concl3} for the notation.

Return to the situation of Subsection \ref{fusq}
and maintain the notation $\,\form\,$
for an $\Ad$-invariant
symmetric bilinear form  on the Lie algebra $\gg$ of the Lie  group $\GG$.
Further, let $H$ be a Lie group with an $\Ad$-invariant
symmetric bilinear form  on its Lie algebra and, with an abuse of notation,
we denote this bilinear form 
by $\,\form\,$ as well.
Endow the Lie algebra $\gg \oplus \hh$
of $\GG \times H$ with the corresponding
 $\Ad$-invariant
symmetric bilinear form  
$\form + \form$
on $\gg \oplus \hh$ and the Lie algebra
$\gg^\times \oplus \hh$
of $\GG^\times \times H$ with the corresponding
 $\Ad$-invariant
symmetric bilinear form  
$\form^{\negthinspace \times} + \form$
on $\gg^\times \oplus \hh$.
The operation of fusion, cf. Proposition \ref{fusmoms},
is available more generally for a 
$(\GG^\times \times H)$-manifold
with a weakly 
$(\GG^\times \times H)$-quasi Hamiltonian
structure
\begin{equation}
(\sigma^\times, ((\mmu_1,\mmu_2, \mmu_3)\colon M \to \GG^ \times \times H))
\label{wtimesH}
\end{equation}
on $M$ relative to
$\form^{\negthinspace \times} + \form$ and leads to the
 weakly 
$(\GG \times H)$-quasi Hamiltonian
structure
\begin{equation}
(\sigma_\fus, ((\mmu_1\mmu_2, \mmu_3)\colon M \to \GG \times H))
\end{equation}
on $M$ relative to
$\form+ \form$, a genuine 
$(\GG \times H)$-quasi Hamiltonian
structure
when so is \eqref{wtimesH}.
 \cite[Theorem 6.1]{MR1638045} establishes this fact for $\GG$
and $H$ compact with positive $\Ad$-invariant symmetric bilinear forms
on their Lie algebras, and the reasoning carries over
to the general case.

For $k \geq 1$, endow the Lie algebra
$\gg^{\oplus k}$ of $\GG^{\times k}$
with the 
$\Ad$-invariant
symmetric bilinear form  
which arises as the sum of the forms on the summands.
Fusing $\ell\geq 1$ copies of the internally fused double 
$(\GG \times \GG,\sigma_1,\Phi_1)$
with $n \geq 0$ copies of the
externally fused double 
$(\GG \times \GG,\sigma_1,(\mult, \omult))$,
for each such copy the operation of fusion being carried out 
with respect to the second copy $\oGG$ of $\GG$ in
$\GG \times \GG = \GG \times \oGG$,
yields a weakly $\GG^{n+1}$-quasi Hamiltonian structure 
$(\sigma^{\ell,n},\Phi^{\ell,n})$ on
the product $\GG^{2(\ell +n)}$ of
$2(\ell +n)$ copies of $\GG$
relative to the corresponding 
$\Ad$-invariant
symmetric bilinear form  on $\gg^{\oplus (n+1)}$
which arises as the sum of the forms on the summands.
In view of
Example \ref{conjwqhr},
with respect to the copy of $\GG^n$ in $\GG^{n+1}= \GG \times\GG^n$,
the Lie algebra of $\GG^n$ being endowed with the sum of the
$\Ad$-invariant
symmetric bilinear forms on the summands,
the reduction procedure in Proposition \ref{wqhr}
relative to $n$ suitable conjugacy classes 
$\widetilde \CcC_1,  \ldots ,\widetilde \CcC_n$
in $\GG$
yields a weakly $\GG$-quasi Hamiltonian manifold
of the kind
$\left(\GG^{2\ell}\times \CcC_1 \times \ldots \times \CcC_n,
\sigma_{\ell,n},\mmu_{\ell,n}\right)$
in Conclusion \ref{concl3};
more precisely, for $1 \leq j \leq n$, we must take $\widetilde \CcC_j$
to be the conjugacy class of the point $q^{-1}$
for a point $q$ of $\CcC_j$.
There is no regularity constraint here.

On the other hand,
relative to $n+1$ suitable conjugacy classes 
$\widetilde\CcC_0, \widetilde \CcC_1,  \ldots ,\widetilde \CcC_n$
in $\GG$,
the weakly quasi Hamiltonian  reduction procedure in Proposition \ref{wqhr}
applies only under certain regularity assumptions
and then leads, for 
$\widetilde\CcC_0 =\{e\}$,
to a moduli space of the kind
$\Rep(\pi,\GG)_{\mathbf C}$ and,
for n=0 and $\widetilde\CcC_0 =\{z\}$, the point $z$ being 
in the center of $\GG$, to a 
 moduli space of the kind
$\Rep_X(\Gamma,\GG)$, cf. Remark \ref{qhreduction},
but this procedure does not recover
the full moduli space in the non-regular case.

Relative to a connected
complex algebraic group $\GG$ and a
non-degenerate $\Ad$-invariant
symmetric bilinear form $\,\form\,$ on its Lie algebra,
\cite{MR3126570} uses the approach we are presently discussing
for the construction of 
\lq\lq wild character varieties\rq\rq\
as
algebraic Poisson varieties.
In \cite[Theorem 2.3]{MR3126570},
what then corresponds to the space 
$\GG^{2(\ell +n)}$ that underlies the algebraic 
$\GG^{n+1}$-quasi Hamiltonian manifold
$(\GG^{2(\ell +n)},\sigma^{\ell,n},\Phi^{\ell,n})$
is written there as $\Hom(\Pi, \GG)$,
with the notation $g$ and $m$ 
playing the role of the present $\ell$ and $n+1$.
Such a  variety arises as follows:

Take $\Sigma$ to be a complex curve with $m >0$ punctures.
Let $C_\GG(Q)\subseteq \GG$ denote the stabilizer of a \lq\lq type\rq\rq\  $Q$
\cite[Section 7 p.~31]{MR3126570}
(definition not reproduced here), a connected complex reductive group.
Consider $m$ types  $Q_1, \ldots, Q_m$,
accordingly, let $\Hom_{\mathbb S}(\Pi,\GG) \subseteq \Hom(\Pi,\GG)$ 
denote the space of Stokes representations
as defined just before \cite[Theorem 8.2 p.~42]{MR3126570},
and let $\mathbf H = C_\GG(Q_1) \times \ldots \times C_\GG(Q_m) \subseteq \GG^m$,
a connected reductive subgroup of $\GG^m$.
By 
that theorem,
$\Hom_{\mathbb S}(\Pi,\GG)$ acquires 
a canonical $\mathbf H$-quasi Hamiltonian structure
$(\sigma,\mmu)$.
By Proposition \ref{4.6} and Theorem \ref{unpois},
the affine algebraic
quotient
$\mathbf M_B(\Sigma) =\Hom_{\mathbb S}(\Pi,\GG)// \mathbf H$
acquires the structure of an algebraic Poisson variety.
This is
\cite[Corollary 8.3 p.~43]{MR3126570}.

While the corresponding analytic Poisson variety
associated with such an 
algebraic Poisson variety can as well be constructed
from an associated extended moduli space,
for the construction
as an algebraic Poisson variety the quasi Hamiltonian approach is essential
since an extended moduli space is a merely analytic object.

Choose  $m$ conjugacy classes
$\CcC_1\subseteq C_{\GG(Q_1)}, \ldots,\CcC_m\subseteq C_{\GG(Q_m)}$;
the 
affine algebraic
quotient
$\mmu^{-1}(\CcC_1 \times \ldots \times \CcC_m)// \mathbf H$
canonically embeds into
$\Hom_{\mathbb S}(\Pi,\GG)// \mathbf H$, as a symplectic leaf
when both are non-singular (affine) varieties, and hence
the Poisson structure
thereupon descends to one on 
$\mmu^{-1}(\CcC_1 \times \ldots \times \CcC_m)// \mathbf H$.
In the presence of singularities,
quasi Hamiltonian reduction is not available here and in particular
does not lead to a Poisson structure on
the algebraic quotient 
$\mmu^{-1}(\CcC_1 \times \ldots \times \CcC_m)// \mathbf H$, however,
while, in terms of  the corresponding extended moduli space,
standard techniques show that the Poisson structure on
$\Hom_{\mathbb S}(\Pi,\GG)// \mathbf H$
restricts to an analytic  stratified symplectic Poisson structure
on 
the quotient 
$\mmu^{-1}(\CcC_1 \times \ldots \times \CcC_m)// \mathbf H$.
When this quotient is non-singular,
\cite[Theorem 1.1]{MR3126570} 
implies that this Poisson structure arises from an algebraic one.

The quasi Poisson technology which we develop in the rest of the paper
yields, in particular, a more
direct construction of such an algebraic  Poisson variety,
also covers the case where the underlying variety
is non-singular,
and 
includes (not necessarily non-singular) algebraic Poisson varieties
that do not necessarily arise
from a quasi Hamiltonian structure
relative to a
non-degenerate $\Ad$-invariant
symmetric bilinear form on the  Lie algebra of the target group
written here as $\GG$.
We come back to this situation in Subsection \ref{stokesd} below.

\begin{thm}
\label{visible}
Let $\Sigma$ be a compact, connected, and oriented (real) topological surface
of genus $\ell$
with $n+1$ boundary circles $(n\geq 0)$,
suppose $\GG$ compact and connected and the requisite
 $\Ad$-invariant symmetric bilinear form
on its Lie algebra positive,
let
$\prin \colon P \to \Sigma$ be a principal $\GG$-bundle on $\Sigma$,
necessarily trivial (for topological reasons), and
let $\mathcal \GG^{(n+1)}$ denote the group of gauge transformations
of $\prin$ that are the identity  at each boundary circle.
Relative to a suitable Sobolev topology,
even in the Fr\'echet topology,
the assignment to a flat connection of suitable holonomies yields a 
diffeomorphism
from the space $\mathrm{Flat}_\xi/\mathcal \GG^{(n+1)}$
of 
$\mathcal \GG^{(n+1)}$-orbits of flat connections on $\xi$
to  $\GG^{2(\ell +n)}$
as $\GG^{n+1}$-quasi Hamiltonian spaces.
\end{thm} 

Since we do not use this theorem
we do not prove it  here nor do we make the 
$\GG^{n+1}$-quasi hamiltonian
structure of $\mathrm{Flat}_\xi/\mathcal \GG^{(n+1)}$ precise.
Suffice it to note that 
the momentum mapping to $\GG^{n+1}$ arises from the monodromies
with respect to the boundary circles.
For a suitable Sobolev topology,
the theorem is precisely
 \cite[Theorem 9.3]{MR1638045}.
The techniques in \cite{MR3836789}
show it is valid
in the Fr\'echet topology.
It is, perhaps, illuminating to recall
that the
$\mathcal \GG^{(n+1)}$-action on the space of connections on $\xi$
is free and that the assignment to a gauge transformation
on $\xi$
of its values on the boundary circles
(perhaps better: at  corresponding punctures)
determines a surjection from the group of gauge transformations
to $\GG^{n+1}$ having $\mathcal \GG^{(n+1)}$
as its kernel. 
Indeed, it might be more appropriate to argue in terms of a punctured
surface and to play it off against its oriented real blow up
(which substitutes an oriented  boundary circle for each puncture)
but we spare the reader and ourselves these added troubles here.

The reader should be warned that in the proof
of \cite[Theorem 2.3]{MR3126570}, the reference
to
 \cite[Theorem 9.1]{MR1638045} is misleading and only heuristically 
appropriate, since the gauge theoretic description
\cite[Theorem 9.1]{MR1638045},
cf. Theorem \ref{visible} above,  is not available for non-compact 
structure group, see Subsection \ref{8.2} below.
In the proof
of \cite[Theorem 2.3]{MR3126570}, the reference to the
corresponding fusion product suffices to validate the claim.

\section{Quasi Poisson structures} 
\label{quasipoiss}
Let $P$ be a bivector
on a manifold $M$.
With a slight abuse of notation,
we denote by $P$ the bidifferential operator 
on $M$ which the bivector $P$ induces,
and we use the notation $P^\sharp$ for the adjoint
\begin{equation}
P^\sharp \colon \TT^* M \longrightarrow \TT M,\ 
P^\sharp(df)(h) = P(df,dh),\ f,h \colon M \to \bK,
\label{Psharp}
\end{equation}
a morphism of vector bundles on $M$,
and the notation  $\TT^* M^P$ for the kernel of 
$P^\sharp$,
a distribution over $M$, not necessarily the total space of a vector bundle.

\subsection{Quasi Jacobi identity}

Define a bracket $\pbra$ on $\mathcat A[M]$ by setting
\begin{equation}
\{a,b\} =\langle P, d a \wedge d b\rangle,
\ a,b \in \mathcat A[M].
\label{pb2}
\end{equation}
This bracket is skew and a derivation in each variable.

Let $\gg$ be a Lie algebra and 
$\gg \to \Vect(M)$ an infinitesimal action
of $\gg$ on $M$,
i.e., a morphism of Lie algebras.
Suppose that
$P$ is invariant under $\gg$.
Then the bracket \eqref{pb2}
induces a bracket
$\pbra \colon \mathcat A[M]^\gg \otimes \mathcat A[M]^\gg \to \mathcat A[M]^\gg$
on the subalgebra $\mathcat A[M]^\gg$ of $\gg$-invariants.

\begin{prop}
\label{prop1}
Let $\phi \in \LAL^{\mrc,3}[\gg]$,
let $\phi_M$ be the image of $\phi$ in
$\Vect(M)$, and
suppose that $[P,P] = \phi_M$.
Then the  restriction of the bracket $\pbra$
to the subalgebra $\mathcat A[M]^\gg$ of $\gg$-invariants
satisfies the Jacobi identity.
\end{prop}

\begin{proof}
For $f_1,f_2,f_3 \in \mathcat A[M]$,
\begin{equation}
\{\{f_1,f_2\},f_3\} +\{\{f_2,f_3\},f_1\} +\{\{f_3,f_1\},f_2\} 
=2\phi_M(d f_1, d f_2, d f_3).
\end{equation}
Suppose $f_1, f_2, f_3\in \mathcat A[M]$ invariant under $\gg$,
that is,
for $X_1, X_2, X_3 \in \gg$,
\begin{equation}
X_1(f_1)=d f_1(X_1) =0,  
\
X_2(f_2)= d f_2(X_2)=0,
\ 
X_3(f_3) =d f_3(X_3) = 0. 
\end{equation}
The $3$-vector
$\phi$ is a sum of terms of the kind
$X_1 \wedge X_2 \wedge X_3$ with $X_1,X_2,X_3 \in \gg$.
Evaluating
$X_1 \wedge X_2 \wedge X_3$ at
$(d f_1, d f_2, d f_3)$
gives zero.
\end{proof}

\begin{rema}
\label{interp1}
{\rm 
This proposition  offers an interpretation of the claim
\lq\lq Since $\phi_M$ vanishes on invariant forms, the space 
$C^\infty(M,\RR)^\GG$ of $\GG$-invariant functions is a Poisson
algebra under $\pbra$\rq\rq\ 
just before 
\cite[Theorem 6.1]{MR1880957}.
(I do not see why, under the circumstances of that theorem,
$\phi_M$ vanishes on general invariant forms,
not of the kind $df_1 \wedge df_2 \wedge df_3$.)
}
\end{rema}

\subsection{Symmetric $2$-tensor and totally antisymmetric $3$-tensor}
\label{symtot}
In this Subsection we work over a general ground ring
$\bR$ that is an algebra over the rationals.
Consider an $\bR$-Lie algebra $\gg$
whose underlying $\bR$-module 
 is a finitely generated and projective
and suppose the corresponding Lie group $\GG$ is well defined.
We do not make this precise.
A typical example is the group of gauge transformations of a principal
bundle and the Lie algebra of infinitesimal gauge transformations
in the Fr\'echet topology;
as a module over the functons, this Lie algebra is 
finitely generated and projective.
See Example \ref{examp3} below.

Let $\oomega$ be an 
$\Ad$-invariant symmetric $2$-tensor in $\gg \otimes \gg$, that is,
a
member of the $\GG$-invariants
$\Sigm^{\mrc,2}[\gg]^\GG$ of the symmetric cosquare 
$\Sigm^{\mrc,2}[\gg] 
= (\gg \otimes \gg)^{C_2}
\subseteq \gg \otimes \gg $ of $\gg$
(invariants under the twist action of the group $C_2$ with $2$ elements).
For book keeping purposes, write
the first copy of $\gg$ in  $\gg \oplus \gg$ as 
$\gg^1$ and the second one as $\gg^2$ if need be.
Relative to
 the  identification
\begin{align}
\wedge \colon \gg^1 \otimes \gg^2  &\longrightarrow  \gg^1 \boxwedge \gg^2, \
u \otimes v \longmapsto u \wedge v,
\ u,v \in \gg,
\label{identi1}
\end{align}
cf. Subsection \ref{gracoh}, let
\begin{equation}
\chi_{\oomega} = \tfrac 12 \wedge (\oomega) \in 
\gg^1 \boxwedge \gg^2 \subseteq\LAL^{\mrc,2}[\gg^1 \oplus \gg^2]
\label{chiH}
\end{equation}
and, with respect to the map
\begin{equation}
\xymatrixcolsep{5pc}
\xymatrix{
\gg \otimes \gg \otimes \gg \otimes \gg
\ar[r]^{\Id \otimes \bra \otimes \Id}
&
\gg \otimes \gg \otimes \gg,
}
\label{compo2}
\end{equation}
let
\begin{equation}
\phi_\oomega = \tfrac 12
(\Id \otimes \bra \otimes \Id)(\oomega \otimes \oomega)
\in \gg \otimes \gg \otimes \gg.
\label{phiH}
\end{equation}
We say the symmetric $2$-tensor 
$\oomega$ in $\gg \otimes \gg$ is 
$\GG$-{\em quasi Poisson convenient\/}
when $\phi_\oomega$
is totally antisymmetric, i.e.,
lies in
$\LAL^{\mrc,3}[\gg] \subseteq \gg \otimes \gg \otimes \gg$.
The factor $\tfrac 12$ is a matter of convenience, so that there
are no coefficients in \eqref{calc11} below.

When $\oomega$ is $\GG$-quasi Poisson convenient,
we  refer to $\phi_\oomega$ as the {\em Cartan element\/}
which $\oomega$ and $\bra$ determine;
the Cartan element
yields a non-trivial class $[\phi_{\oomega}]$
in $\Ho_3(\gg)$ (= $\Ho_3(\gg, \bR)$).
This observation is classical
when $\gg$ is semisimple over a field
and $\oomega$ arises from the trace form.

Recall the canonical identification  
$\LAL[\gg] \to \Lc[\gg]$ of $\bR$-Hopf algebras,
cf. Subsection \ref{gracoh}.
We prove the following  proposition after Lemma \ref{calc} below.
\begin{prop}
\label{cartan}
Suppose $\oomega$ is $\GG$-quasi Poisson convenient.
Then, relative to the diagonal map
$\Delta \colon \LAL[\gg] \to  \LAL[\gg^1]\otimes \LAL[\gg^2]$
of $\LAL[\gg]$
and the Gerstenhaber bracket
$\bra$ on $\LAL[\gg^1 \oplus \gg^2]$,
\begin{equation}
[\chi_\oomega,\chi_\oomega] = \Delta(\phi_\oomega) - \phi_\oomega^1 - \phi_\oomega^2 .
\label{calc11}
\end{equation}
\end{prop}

Evaluation induces a canonical morphism $\ev\colon
\gg \otimes \gg \otimes \gg^* \to \gg$
of $\GG$-modules. Hence, under the canonical identification
of the $\Ad$-invariant symmetric $2$-tensor $\oomega \in \gg \otimes\gg$
with a $\GG$-equivariant morphism $\bR \to \gg \otimes \gg$,
the adjoint $\bR$-module morphism
\begin{equation}
\xymatrix
{
\oomega^\sharp\colon \gg^* \ar[r]^{\oomega \otimes \Id \phantom{aa}} &\gg \otimes \gg \otimes \gg^* 
\ar[r]^{\phantom{aaaa}\ev}&
\gg
}
\end{equation}
is a morphism of $\GG$-modules.
Hence the image   $\oomega^\sharp(\gg^*) \subseteq \gg$
is a $\GG$-submodule of $\gg$, therefore a 
 $\gg$-submodule relative to the adjoint representation
of $\gg$ on itself,
and thence a $\GG$-invariant Lie ideal $\gg_\oomega$ in $\gg$.

Let $\odot$ be  the $\Ad^*$-invariant symmetric bilinear form
on $\gg^*$ which $\oomega$ determines via evaluation, not necessarily
non-degenerate.

\begin{prop}
\label{arise}
Suppose that the $\bR$-module that underlies $\gg_\oomega$
is finitely generated and projective so that, in particular,
the canonical morphisms 
$\gg_\oomega^{\otimes n} \to \gg^{\otimes n}$
of $\bR$-modules
are injections for $n \geq 0$.

\begin{enumerate}

\item 
The symmetric bilinear form
$\odot$
on $\gg^*$
induces a non-degenerate
$\Ad$-invariant symmetric bilinear form
$\,\form\,$
on $\gg_\oomega$,
and the adjoint $\psidot\colon \gg_\oomega \to \gg_\oomega^*$
thereof is a $\GG$-equivariant isomorphism.

\item The $2$-tensor $\oomega$ lies in
$\gg_\oomega \otimes \gg_\oomega \subseteq \gg \otimes \gg$.

\item
Under the $\GG$-equivariant isomorphism 
\begin{equation}
(\psidot)^{\otimes 3} \colon 
\gg_\oomega \otimes \gg_\oomega \otimes \gg_\oomega \longrightarrow  
 (\gg_\oomega^*) \otimes (\gg_\oomega^*) \otimes (\gg_\oomega^*) \stackrel{\cong}\longrightarrow
\Hom(\gg_\oomega \otimes \gg_\oomega \otimes \gg_\oomega ,\bR),
\end{equation} 
the $3$-tensor
which the image of $\oomega \otimes \oomega$
under \eqref{compo2}
(with $\gg_\oomega$ substituted for $\gg$)
goes to 
the negative of the $3$-tensor
which the triple product associated with
$\bra$ and $\,\form\,$ induces.

\item
Consequently
the image of $\oomega \otimes \oomega$ under {\rm \eqref{compo2}}
lies in $\LAL^{\mrc,3}[\gg_\oomega] 
\subseteq \gg_\oomega^{\otimes 3}\subseteq \gg^{\otimes 3}$,
that is,
 the $2$-tensor $\oomega$ is $\GG$-quasi Poisson convenient.
\end{enumerate}

\end{prop}

\begin{proof}
Since $\gg_\oomega$ is $\bR$-projective,
the canonical $\bR$-module morphism 
$\gg_\oomega \otimes \gg_\oomega \to \gg \otimes \gg$ being an injection,
the $2$-form $\odot$ on $\gg^*$ factors through an
$\Ad$-invariant symmetric bilinear form
$\,\form\,$
on $\gg_\oomega$
such that the composite
\begin{equation}
\gg^* \otimes \gg^* \stackrel{\oomega^\sharp \otimes \oomega^\sharp}
\longrightarrow \gg_\oomega \otimes \gg_\oomega \stackrel{\form}
\longrightarrow \bR
\end{equation}
coincides with $\odot$ and $\,\form\,$ is non-degenerate.

Since the $\bR$-module that underlies $\gg_\oomega$ 
is finitely generated and projective,
the symmetric bilinear form $\,\form\,$ on $\gg_\oomega$ determines
a $2$-tensor 
$\oomega^* \in \gg_\oomega^* \otimes\gg_\oomega^*$.
Let $(e_j, \eeta^k)_{1 \leq j,k \leq r}$
be a coordinate system for $\gg_\oomega$ and
write $\oomega^*$ as
$\oomega^* = \eeta_{j,k} \eeta^j \otimes \eeta^k$
 for suitable members 
 $\eeta_{j,k}$ of $\bR$.
Accordingly
\begin{align*}
e_j \form e_k &= \eeta_{j,k}, \quad 1 \leq j,k \leq r,
\end{align*}
and the adjoint 
$\psidot=\oomega^{*\sharp} \colon \gg_\oomega \to \gg_\oomega^*$
of $\,\form\,$ reads
\begin{align}
\psidot(e_j)=\oomega^{*\sharp} (e_j)&=\eeta_{j,k} \eeta^k .
\label{reads}
\end{align}
Since $\,\form\,$ is non-degenerate, 
the adjoint $\oomega^{*\sharp}$ is an isomorphism.
Hence there are members
$\eeta^{j,k}$ ($ 1 \leq j,k \leq r$) of
$\bR$ that define a symmetric $2$-tensor
$\widetilde \oomega= \eeta^{j,k} e_j \otimes e_k$
over $\gg_\oomega$
such that
\begin{equation}
\widetilde \oomega^\sharp \colon  \gg_\oomega^* \longrightarrow \gg_\oomega,
\quad
\widetilde \oomega^\sharp (\eeta^j)=\eeta^{j,k} e_k
\end{equation}
yields the inverse
of the adjoint $\oomega^{*\sharp}$,
and $\eeta^{j,k} \eeta_{k,s} = \delta^j_s$.
Viewed as a member of $\gg \otimes \gg$,
the $2$-tensor $\widetilde \oomega$ coincides with $\oomega$.

For the rest of the proof, 
to simplify the exposition,
we may assume $\gg_\oomega = \gg$.
Thus $\,\form\,$ is
an $\Ad$-invariant symmetric bilinear form
on  $\gg$. 
This form and the Lie bracket
induce the triple product
\begin{equation}
\gg \otimes \gg \otimes \gg \longrightarrow \bR,\ (x,y,z) \mapsto
[x,y]\form z,\ x,y,z \in \gg,
\label{alternate11}
\end{equation}
an alternating $\bR$-valued  $3$-form
on $\gg$.

Write the Lie bracket of $\gg$ as a $3$-tensor
$\bra^\flat \in \gg^* \otimes \gg^* \otimes \gg $.
Under the canonical isomorphism
$\LALc[\gg^*] \to \Alt^*(\gg,\bK)$,
the alternating
$3$-form \eqref{alternate11} corresponds to a member of
$\LAL^{\mrc,3}[\gg^*]$,
that is, to a totally antisymmetric $3$-tensor over $\gg^*$,
 and
the image of $\bra^\flat \in \gg^* \otimes \gg^* \otimes \gg $
under
\begin{equation}
\xymatrixcolsep{3.5pc}
\xymatrix{
\gg^* \otimes \gg^* \otimes \gg \ar[r]^{\Id \otimes \Id \otimes {\psidot}}
&\gg^* \otimes \gg^* \otimes \gg^* 
}
\end{equation}
yields that very same  member of
$\LAL^{\mrc,3}[\gg^*]$.

Under 
$\psidot\colon\gg \to \gg^*$,
the image of $\oomega \otimes \oomega$ under 
\eqref{compo2}
goes to a  member of
$\gg^* \otimes \gg^* \otimes \gg^*$
and hence defines an $\bR$-valued $3$-linear
form on $\gg$.
Up to sign, this form coincides with the familiar $3$-form
arising from the triple product.

Indeed, in terms of the chosen coordinate system, 
define the \lq\lq structure constants\rq\rq\ 
$\eeta^k_{u,v} \in \bR$ by
\begin{equation}
[e_u,e_v]= \eeta^k_{u,v} e_k , 
\end{equation}
and let
$\eeta^{j,k,s} =\eeta^{j,u}  \eeta^k_{u,v}  \eeta^{v,s}$.
In terms of this notation, the image
of $\oomega \otimes \oomega$ under \eqref{compo2}
is the member
$\eeta^{j,k,s}e_j  \otimes e_k \otimes e_s $ of
$\gg \otimes \gg\otimes \gg $.

On the other hand, 
$\bra^\flat =\eeta^k_{u,v} \eeta^u \otimes \eeta^v \otimes e_k$ whence,
with the notation
$\eeta_{u,v,s}=\eeta^k_{u,v} \eeta_{k,s}$ ($1 \leq u,v,s \leq r$),
the right-hand side of
\begin{align*}
(\Id \otimes \Id \otimes \psidot)(\bra^\flat) &=
\eeta_{u,v,s} \eeta^u \otimes \eeta^v \otimes \eeta^s
\end{align*}
recovers
the triple product on $\gg$.

In view of \eqref{reads},
\begin{equation*}
(\psidot \otimes \psidot \otimes \psidot)
(e_j \otimes e_k \otimes e_s) = 
\eeta_{j,u} \eeta_{k,v} \eeta_{s,w}\eeta^u \otimes \eeta^v \otimes \eeta^w.
\end{equation*}
Since $ \eeta_{r,w}\eeta^{w,v} = \delta_r^v$,
\begin{align}
\eeta_{j,a} \eeta_{k,b} \eeta_{s,c}\eeta^{j,k,s} &
=\delta^u_a \delta^w_b \delta^v_c \eeta_{u,v,w} =\eeta_{a,c,b}
\label{cal101}
\end{align}
whence
\begin{equation*}
(\psidot \otimes \psidot \otimes \psidot)
(\eeta^{j,k,s}e_j \otimes e_k \otimes e_s) = 
\eeta_{u,w,v}\eeta^u \otimes \eeta^v \otimes \eeta^w=
-\eeta_{u,v,w}\eeta^u \otimes \eeta^v \otimes \eeta^w. \qedhere
\end{equation*}
\end{proof}

\begin{cor}
\label{every}
For an ordinary finite dimensional Lie 
group $\GG$ with Lie
algebra $\gg$ 
over a field $\bK$, 
every $\Ad$-invariant
symmetric $2$-tensor in $\gg \otimes \gg$
is
$\GG$-quasi Poisson convenient,
and every such
$2$-tensor on $\gg$ arises
from an $\Ad$-invariant  non-degenerate symmetric 
bilinear form on a $\GG$-ideal $\hh$ in $\gg$
as the corresponding
 $\Ad$-invariant  symmetric 
$2$-tensor
 on $\hh$ and hence on $\gg$. \qed
\end{cor}

\begin{lem}
\label{calc}
As before, for book keeping purposes, write the first copy of $\gg$ in
$\gg \oplus \gg$ as $\gg^1$ and the second one as $\gg^2$,
let
$(e_j)$ be a family of members of $\gg$
that span a Lie subalgebra,
and let $\eeta^k_{u,v}\in \bR$ be the corresponding structure constants
so that
$[e_u,e_v]= \eeta^k_{u,v} e_k$.
Further, let
\begin{equation}
 \eeta^{j,k}  e_j^1 \otimes e_k^2\in \gg^1 \otimes \gg^2
\end{equation}
be a symmetric $2$-tensor in $\gg \otimes \gg$, suppose that
$\eeta^{j,k,s}= \eeta^{j,u} \eeta^k_{u,v} \eeta^{v,s}$
is totally antisymmetric in $j,k,s$,
and let
\begin{align}
\chi&=\tfrac {\eeta^{j,k}}2  e_j^1 \wedge e_k^2 
\in \gg^1 \boxwedge \gg^2 \subseteq \LAL^2[\gg^1 \oplus \gg^2]
\label{chi}
\\
\phi&= \tfrac 1 {12}\eeta^{j,k,s} e_j \wedge e_k \wedge e_s
\in \gg \otimes \gg \otimes \gg .
\label{twelvephi4}
\end{align}
Then, 
relative to the diagonal map
$\Delta \colon \LAL[\gg] \to  \LAL[\gg^1]\otimes \LAL[\gg^2]$
of $\LAL[\gg]$
and the Gerstenhaber bracket
$\bra$ on $\LAL[\gg^1 \oplus \gg^2]$,
\begin{equation}
[\chi,\chi] = \Delta(\phi) - \phi^1 - \phi^2 .
\label{calc1}
\end{equation}
\end{lem}

\begin{proof}
From
\begin{align*}
[e^1_j \wedge e^2_j, e^1_k \wedge e^2_k]&=
\begin{cases}
\phantom{+}
[e^1_j, e^1_k] \wedge e^2_j \wedge e^2_k
\\
-
[e^1_j, e^2_k] \wedge e^2_j \wedge e^1_k 
\\
-[e^2_j, e^1_k] \wedge e^1_j  \wedge e^2_k
\\
+[e^2_j,  e^2_k]  \wedge
e^1_j \wedge e^1_k
\end{cases}
=
\eeta^s_{j,k} \left(e^1_s \wedge e^2_j \wedge e^2_k
+e^2_s  \wedge e^1_j \wedge e^1_k \right)
\end{align*}
we deduce
\begin{equation}
4[\chi,\chi] 
= 
\eeta^{j,k,s}(e^1_j \wedge e^2_k \wedge e^2_s +e^2_j \wedge e^1_k \wedge e^1_s).
\label{ident22}
\end{equation}
In the same vein, 
\begin{align*}
\Delta(e_j \wedge e_k \wedge e_s)&=
\begin{cases}
\phantom{+}(e_j \wedge e_k \wedge e_s) \otimes 1
\\
+
(e_j \wedge e_k)  \otimes e_s+
e_s \otimes (e_j \wedge e_k) 
\\
+
e_j \otimes ( e_k \wedge  e_s)+
 (e_k \wedge  e_s) \otimes e_j
\\
-(e_j \wedge e_s) \otimes e_k
-e_k \otimes (e_j \wedge e_s)
\\
+
1 \otimes (e_j \wedge e_k \wedge  e_s)
\end{cases}
\\
&=
\begin{cases}
\phantom{+}
e^1_j \wedge e^1_k \wedge e^1_s
\\
+
e^1_j \wedge e^1_k  \wedge e^2_s+
e^1_s \wedge  e^2_j \wedge e^2_k 
\\
+
e^1_j \wedge e^2_k \wedge  e^2_s+
e^1_k \wedge  e^1_s \wedge e^2_j
\\
-e^1_j \wedge e^1_s \wedge e^2_k
-e^1_k \wedge e^2_j \wedge e^2_s
\\
+
e^2_j \wedge e^2_k \wedge  e^2_s
\end{cases}
=
\begin{cases}
\phantom{+}
e^1_j \wedge e^1_k \wedge e^1_s
\\
+
e^1_j \wedge e^1_k  \wedge e^2_s
+
e^1_j \wedge e^2_k \wedge  e^2_s
\\
+
e^1_k \wedge  e^1_s \wedge e^2_j+e^1_k \wedge e^2_s \wedge e^2_j
\\
+e^1_s \wedge e^1_j \wedge e^2_k
+
e^1_s \wedge  e^2_j \wedge e^2_k \quad
\\
+
e^2_j \wedge e^2_k \wedge  e^2_s .
\end{cases}
\end{align*}
Consequently
\begin{align}
\Delta(\eeta^{j,k,s} e_j \wedge e_k \wedge  e_s)
&=
\begin{cases}
\phantom{+}\eeta^{j,k,s}
e^1_k \wedge  e^1_k \wedge e^1_s 
+
\eeta^{j,k,s}
e^2_j \wedge  e^2_k \wedge e^2_s 
\\
+3\eeta^{j,k,s}\left(e^1_j \wedge e^1_k  \wedge e^2_s
+
e^1_j \wedge e^2_k \wedge  e^2_s\right).
\end{cases}
\label{ident33}
\end{align}
\eqref{ident22} and \eqref{ident33} together imply 
\eqref{calc1}.
\end{proof}

\begin{proof}[Proof of Proposition {\rm \ref{cartan}}]
In terms of the notation of Lemma \ref{calc},
\begin{equation*}
\phi_\oomega = \tfrac 12 \eeta^{j,k,s} e_j \otimes e_k \otimes e_s =
\tfrac 1{12} \eeta^{j,k,s} e_j \wedge e_k \wedge e_s. \qedhere
\end{equation*}
\end{proof}

\begin{examp}
\label{fund}
{\rm
Over a field,
\cite{MR823849} exhibits, on Lie algebras that
are not necessarily semisimple or reductive, 
examples of 
 $\Ad$-invariant non-degenerate symmetric
bilinear forms that do not arise
as trace forms.

Let $\oomega$ be a non-trivial 
$\Ad$-invariant symmetric $2$-tensor in $\gg \otimes \gg$
such that the adjoint
$\varotheta \colon \gg^* \to \gg$
of the 
$\GG$-invariant symmetric bilinear form $\odot$  on $\gg^*$
which $\oomega$ induces
factors as
$\gg^* \to \hh^* \to \hh$
through an isomorphism $\hh^* \to \hh$
of $\GG$-modules.
Then $\oomega$ is $\GG$-quasi Poisson convenient.
For example, this happens when
$\gg$ is the direct sum
 $\gg = \hh \oplus \qq $
of two $\GG$-ideals with
$\hh$ semisimple when we take $\oomega$ to be the
symmetric $2$-tensor in  $\hh \otimes \hh$ and hence in $\gg \otimes \gg$
which an $\Ad$-invariant  non-degenerate
symmetric bilinear form on $\hh$ induces.
}

\end{examp}

\subsection{Quasi Poisson structure}
\label{qps}
From  now on, $\GG$ is an ordinary
(finite-dimensional) Lie group
and $\gg$ its Lie algebra.
Let $\oomega \in \gg \otimes \gg$
be an $\Ad$-invariant symmetric $2$-tensor,
by Corollary \ref{every}
necessarily $\GG$-quasi Poisson convenient.

For a $\GG$-manifold $M$, 
the 
$\Ad$-invariant
totally antisymmetric $3$-vector $\phi_\oomega$ (relative to the 
$\Ad$-invariant symmetric $2$-tensor 
$\oomega$ in $\gg \otimes \gg$), cf. Proposition \ref{cartan}, induces,
via
the infinitesimal action
$\gg \to \Vect(M)$ of the Lie algebra
$\gg$ on $M$, a $\GG$-invariant totally antisymmetric $3$-tensor
$\phi_{\oomega,M}$ on $M$.
We define a $\GG$-{\em quasi Poisson\/} structure
on a $\GG$-manifold $M$ {\em relative to\/} $\oomega$
to be a $\GG$-invariant bivector $P$ on $M$ such that
\begin{equation}
[P,P] = \phi_{\oomega,M} .
\end{equation}
Occasionally we simplify the notation somewhat and write
$\phi_M$ rather than $\phi_{\oomega,M}$.
A  $\GG$-{\em quasi Poisson\/} manifold
is a $\GG$-manifold $M$ 
together with a 
$\GG$-quasi Poisson structure $P$. 
When $\GG$ is compact and $\oomega$  arises from 
a  non-degenerate  positive definite
$\Ad$-invariant symmetric bilinear form on $\gg$,
the present definition recovers  \cite[Definition 2.1 p.~5]{MR1880957}.

\begin{examp}
\label{examp2}
{\rm
Let the first copy of $\GG \times \GG$ act on $\GG$
by left translation and the second one by right translation.
This turns $\GG$ into a $(\GG \times \GG)$-manifold.
Let $\phi^L_{\GG,\oomega},\phi^R_{\GG,\oomega}\in \LAL^{\mrc,3}[\GG]$
denote the respective image of $\phi_{\oomega}\in \LAL^{\mrc,3}[\gg]$
in $\LAL^{\mrc,3}[\GG]$ under  left translation
and right translation.
(N.B. $\phi^L_{\GG,\oomega}$ is the image of
$\phi_{\oomega}$ under the infinitesimal right translation action of $\gg$ on  $\GG$
and 
$-\phi^R_{\GG,\oomega}$ that of
$\phi$ under the infinitesimal left translation action.) 
Then $\phi_{\GG,\oomega} = \phi^L_{\GG,\oomega} - \phi^R_{\GG,\oomega}$.
Since $\phi$ is $\Ad$-invariant,
the $3$-vectors  $\phi^L_{\GG,\oomega}$ and $\phi^R_{\GG,\oomega}$
on $\GG$ coincide.
Hence the zero structure is a $(\GG \times \GG)$-quasi Poisson structure
on $\GG$ relative to the symmetric $2$-tensor $\oomega$ in $\gg \otimes \gg$.

This kind of reasoning applies to any  $(\GG \times \GG)$-manifold
$M$ such that the sum 
$\phi^1_M + \phi^2_M$
of the respective  $3$-vectors 
$\phi^1_M \in \LAL^{\mrc,3}[M]$ 
and $\phi^2_M\in \LAL^{\mrc,3}[M]$
relative to the action of the first and second copy of $\GG$
is zero.

}
\end{examp}

\begin{examp}
\label{examp1}
{\rm
Since $\oomega$ is symmetric,
the composite
\begin{equation}
\xymatrixcolsep{5pc}
\xymatrix{
P_{\GG,\oomega}\colon \GG \ar[r]^{\oomega \phantom{aaaaaa}} & 
\GG \times (\gg \otimes \gg) \ar[r]^{\tfrac 12 (L + R) \otimes (L-R)}
&\TT \GG \otimes_\GG \TT \GG
}
\label{PG}
\end{equation}
or, equivalently,
the composite
\begin{equation}
\xymatrixcolsep{5pc}
\xymatrix{
P_{\GG,\oomega}\colon \GG \ar[r]^{\oomega \phantom{aaaaaa}} & 
\GG \times (\gg \otimes \gg) \ar[r]^{-\tfrac 12 (L \wedge  R)}
&\TT \GG \otimes_\GG \TT \GG
}
\label{PG2}
\end{equation}
characterizes a skew-symmetric $2$-tensor in $\gg \otimes \gg$.
Let
 $\phi_{\GG,\oomega} \in \LAL^{\mrc,3}[\GG]$
denote the image 
of $\phi_\oomega \in \LAL^{\mrc,3}[\gg]$
relative to the  conjugation action of $\GG$ on itself.
\begin{prop}
With respect to conjugation,
the bivector field
$P_{\GG,\oomega}$ on $\GG$ is a $\GG$-quasi Poisson structure
on $\GG$ relative to the symmetric $2$-tensor $\oomega$
in $\gg \otimes \gg$, that is
\begin{equation}
[P_{\GG,\oomega},P_{\GG,\oomega}] = \phi_{\GG,\oomega} \in \LAL^{\mrc,3}[\GG].
\label{qPG}
\end{equation}
\end{prop}
\begin{proof}

The action of $\GG \times \GG$  on $\GG$
in Example \ref{examp2}
has fundamental vector field map
\begin{equation}
L^2 - R^1 \colon\GG \times (\gg^1 \oplus \gg^2) \longrightarrow \TT \GG.
\label{fundLR}
\end{equation}
Let $\Delta(\phi_\oomega)_\GG\in \LAL^{\mrc,3}[\GG]$
denote the  image of 
$\Delta(\phi_\oomega)\in \LAL^{\mrc,3} (\gg \oplus \gg)$
under the map
\begin{equation}
\GG \times \LAL^{\mrc,3} (\gg \oplus \gg) \longrightarrow  
\LAL^{\mrc,3}[\TT \GG]
\end{equation}
which \eqref{fundLR} induces.
Since combining the  $(\GG \times \GG)$-action
on $\GG$ with the diagonal $\GG \to \GG \times \GG$
yields the conjugation action of $\GG$ on itself, the value
 $\Delta(\phi_\oomega)_\GG \in \LAL^{\mrc,3}[\GG]$
coincides with  $\phi_{\GG,\oomega}$. 
The $3$-vectors
 $\phi^L_{\GG,\oomega}$ and $-\phi^R_{\GG,\oomega}$ being identical,
identity \eqref{calc1} implies \eqref{qPG}.
\end{proof}
In terms of a basis
$e_1,\ldots,e_d$ of $\gg$, with
$\oomega = \eeta^{j,k} e_j \otimes  e_k$ for $\eeta^{j,k} \in \bK$,
\begin{align}
P_{\GG,\oomega}&= \tfrac 12 \eeta^{j,k} e_j^R \wedge e_k^L. 
\label{tB1}
\end{align}
}
\end{examp}
Henceforth we simplify the notation
somewhat and write $P_\GG$ rather than 
$P_{\GG,\oomega}$ etc. when
the $2$-tensor $\oomega$ is clear from the context.

\begin{examp}
\label{examp3}
{\rm
Let $\prin\colon \Prin \to M$ be a principal $\GG$-bundle 
and
 $\Conn$ the affine space of connections on
$\prin$.
The tangent space 
$\TT_A\AA_\xi$
at a connection $A$
is the space 
$\AA^1(M, \ad(\xi))$
of $\ad(\xi)$-valued $1$-forms on $M$.
The group $\mathcat G_\xi$ of gauge transformations of $\xi$ is the group
of $\GG$-equivariant diffeomorphisms $\Prin \to \Prin$
over the identity of $M$.
Identify $\mathcat G_\xi$ with $\Map_\GG(\Prin,\GG)$ and its Lie algebra
$\gg_\xi$ with $\Map_\GG(\Prin,\gg)$.
Since $\oomega$ is $\Ad$-invariant,
(with a slight abuse of the notation
$\oomega$,)
the constant map $\oomega \colon \Prin \to \gg \otimes \gg$
is $\GG$-equivariant, indeed,
factors through the constant map $\oomega \colon \Prin/\GG \to \gg \otimes \gg$,
and $\oomega$, viewed as a member of
$\Map_\GG(\Prin,\gg \otimes \gg)$,
is $\Ad$-invariant
with respect to the adjoint action of
$\Map_\GG(\Prin,\GG)$ on $\Map_\GG(\Prin,\gg \otimes \gg)$.
Moreover, (with a slight abuse of the notation
$\phi_{\oomega}$,)
the constant map
$\phi_{\oomega}\colon \Prin \to \LAL^{\mrc,3}[\gg]$
is $\GG$-equivariant, indeed,
factors likewise through the constant map 
$\phi_{\oomega} \colon \Prin/\GG \to \LAL^{\mrc,3}[\gg]$, and we view
$\phi_{\oomega}\colon \Prin \to \LAL^{\mrc,3}[\gg]$ 
as a totally antisymmetric  $3$-tensor
over
$\Map_\GG(\Prin, \gg)$.

The operations 
\begin{equation}
L,R \colon \Map_\GG(\Prin,\GG) \times\Map_\GG(\Prin,\gg) \longrightarrow
\Map_\GG(\Prin,\TT \GG) 
\end{equation}
make sense:
For $u \in  \Map_\GG(\Prin,\GG)$ and  $v \in\Map_\GG(\Prin,\gg)$, define
$L_u(v) \in\Map_\GG(\Prin,\TT \GG) $
by
\begin{equation}
L_u(v)(q)= L_u(q) v(q) \in \TT_{u(q)} \GG.
\end{equation} 
Hence the tensor
\begin{equation}
P_{\Map_\GG(\Prin,\GG)}= \tfrac 12 (R \wedge L)(\oomega)\colon  \Map_\GG(\Prin, \GG) \to \Map_\GG(\Prin, \TT^2\GG)
\end{equation}
is available and, indeed,
yields a $\Map_\GG(\Prin, \GG)$-quasi Poisson structure
on  $\Map_\GG(\Prin, \GG)$ relative to the $3$-tensor
$\phi_{\oomega}\colon \Prin \to \LAL^{\mrc,3}[\gg]$ over
$\Map_\GG(\Prin, \gg)$.

Let $q$ be a point of $M$ and consider the evaluation map
$\ev_q\colon \Map_\GG(\Prin,\GG) \to \GG$. 
In the smooth setting,
in the Fr\'echet topology,
the evaluation map $\ev_q$ is a smooth epimorphism 
of Lie groups. By construction, $\ev_q$ 
is compatible with the quasi Poisson structures.

In the algebraic setting,
when $\GG$ is an algebraic group defined over $\bK$
with coordinate Hopf algebra $\bK [\GG]$
 and $\prin \colon \Prin \to M$ an
algebraic principal $\GG$-bundle 
over an affine variety $M$ so that the coordinate ring $\bK[M]$ makes sense,
we view
\begin{equation}
\Map_\GG(\Prin,\GG) \cong \Hom_{\Alg}(\bK [\GG], \bK[\Prin])^\GG 
\end{equation}
as the group of $(\bK[M])$-points of $\GG$
and $\Map_\GG(\Prin,\gg)$ as the Lie algebra of $(\bK[M])$-points of $\gg$, 
twisted via $\prin$.
When $\prin$ is trivial,
this comes down to
 the ordinary group of $(\bK[M])$-points of $\GG$.
In the general case, the underlying 
$(\bK[M])$-module
of 
$\Map_\GG(\Prin,\gg)$ 
is finitely generated and projective.
In the smooth setting, suitably interpreted,
this fact holds as well.
}
\end{examp}

\subsection{Momentum  mapping}
\label{momentum}

Let  $P$ be a $\GG$-invariant bidifferential operator 
(equivalently bivector field or skew-symmetric $2$-tensor)
on $M$.
Recall from the introduction that we refer to
an admissible   $\GG$-equivariant 
map $\Phi \colon M \to \GG$ 
(with respect  to the conjugation action of $\GG$ on itself)
as a $\GG$-{\em momentum mapping for\/}
$P$ 
{\em relative to\/} $\oomega$
when it renders  diagram \eqref{PMPhi}
 commutative.
The following is immediate.

\begin{prop}
\label{momG}
The identity of $\GG$ is a $\GG$-momentum mapping for $P_\GG$,
see {\rm \eqref{PG}} and {\rm \eqref{PG2}},
 relative
to $\oomega$. 
Furthermore, the restriction $P_{\CcC}$ of the bivector field $P_\GG$
to a conjugacy class ${\CcC}$ in $\GG$ is tangent to ${\CcC}$ whence, with respect to conjugation,
 $P_{\CcC}$ is a
$\GG$-quasi Poisson structure on ${\CcC}$, and  
the  inclusion $\iota \colon {\CcC} \to \GG$
is a $\GG$-momentum mapping for $P_{\CcC}$ relative to $\oomega$.
\qed
\end{prop}

\begin{cor}
\label{cor1}
The diagram
\begin{equation}
\begin{gathered}
\xymatrixcolsep{3pc}
\xymatrixcolsep{4pc}
\xymatrix{
\TT^* \GG\ar[d]_{2P_\GG^\sharp}
\ar[r]^{L^* + R^*} 
& \GG \times \gg^* \ar[d]^{\Id \times \psi^\oomega} 
\\
\TT \GG
&\ar[l]^{L-R} \GG \times \gg 
}
\end{gathered}
\label{CD03}
\end{equation}
is commutative,
and the diagram being commutative characterizes $P$.
\end{cor}

The following proposition characterizes the momentum property in other ways.
This will enable us
to reconcile our momentum  property with a corresponding  one in the 
literature.

\begin{prop}
\label{equival1}
For  a $\GG$-invariant bivector $P$
on a $\GG$-manifold $M$ relative to $\oomega$
and a $\GG$-equivariant map
$\Phi \colon M  \to \GG$,
the following are equivalent:
\begin{enumerate}
\item The $\GG$-equivariant map
$\Phi \colon M  \to \GG$ is a $\GG$-momentum mapping for $P$
relative to $\oomega$.

\item
The operators $L_\Phi^*$ and $R_\Phi^*$ dual to
the operators in {\rm \eqref{display}}
render the diagram
\begin{equation}
\begin{gathered}
\xymatrixcolsep{4pc}
\xymatrix{
\TT^* M\ar[d]_{2P^\sharp}  &
\ar[l]_{(d \Phi)_M^*} \TT_\Phi^*\GG  \ar[r]^{L_\Phi^* + R_\Phi^*}&
M \times \gg^*
\ar[d]^{\Id \times \psi^\oomega}
\\
\TT M 
&&\ar[ll]^{\fund_M} M \times \gg  
}
\end{gathered}
\label{CD01}
\end{equation}
commutative.

\item
The operators $L_\Phi$ and $R_\Phi$ 
in {\rm \eqref{display}}
render the diagram
\begin{equation}
\begin{gathered}
\xymatrixcolsep{4pc}
\xymatrix{
\TT M
\ar[r]^{d \Phi_M}& \TT_\Phi\GG & \ar[l]_{L_\Phi + R_\Phi}
M \times \gg
\\
\TT^* M \ar[u]^{-2P^\sharp}  
\ar[rr]_{\fund^*_M}&& M \times \gg^*  
\ar[u]_{\Id \times \psi^\oomega}
}
\end{gathered}
\label{CD01d}
\end{equation}
commutative.

\item
For every $\bK$-valued admissible function $f$ on $\GG$,
in terms of the basis $e_1,\ldots, e_d$ of $\gg$
such that $\oomega = \eeta^{j,k} e_j \otimes e_k$,
\begin{equation}
2 P^\sharp(df \circ d \Phi)
=\eeta^{j,k}\left( (e_j^L + e_j^R)(f)\circ \Phi  \right)e_{k,M}.
\label{momapG44}
\end{equation}
\item
For every $\bK$-valued admissible function $f$ on $\GG$,
the vector field
$(P^\sharp_\GG (df) \circ \Phi)_M$
 on $M$
which the composite
$(P^\sharp_\GG (df)) \circ \Phi \colon M \to \gg$
induces
through the infinitesimal $\gg$-action on $M$
satisfies the identity
\begin{equation}
P^\sharp(d (f  \circ \Phi)) = (P^\sharp_\GG (df) \circ \Phi)_M.
\label{momapP1}
\end{equation}
\end{enumerate}
\end{prop}

\begin{proof}
This is straightforward.
We only note that, 
since $P^\sharp$ is skew,
its dual $P^{\sharp,*} \colon \TT^* M \to \TT M$
coincides with $-P^\sharp$.
\end{proof}

\begin{rema}
\label{reconcile}
{\rm When
 $\GG$ is a compact Lie group (so that $\bK = \RR$)
 and $\oomega$  arises from 
a  non-degenerate $\Ad$-invariant positive definite
 symmetric bilinear form on $\gg$,
 identity \eqref{momapG44} 
 recovers  (6) in
\cite[Definition 2.2 p.~6]{MR1880957}.
Thus, in this case,
our definition
of a momentum mapping is equivalent to that
in \cite[Definition 2.2 p.~6]{MR1880957}.
For a general Lie group $\GG$ and 
non-degenerate $2$-form on its Lie algebra $\gg$,
our definition  is also equivalent to the definition in
\cite[\S 3.5.1 p.~17]{MR2103001}
and that in
\cite[\S 5.4]{MR2642360}.
}
\end{rema}

\begin{rema}
\label{rema1}
{\rm
For $\alpha \in \gg^*$, let
$\alpha^{\bra}$ denote the member of $ \gg^ * \otimes \gg^*$
which the composite 
$\gg \otimes \gg \stackrel \bra \to \gg \stackrel \alpha \to \bK$ 
characterizes. 
On  
the affine $\GG$-manifold $M_\gg^*$ that underlies $\gg^*$,
the map
\begin{equation}
P_{\bra}\colon M_{\gg^*} 
\longrightarrow
M_{\gg^*} \times (\gg^* \otimes \gg^*),\ P_{\bra}(\alpha) = (\alpha, \alpha^{\bra}),
\end{equation}
recovers the Lie-Poisson tensor.
Let $\mathrm I \in \gg^* \otimes \gg$
be the fundamental tensor, that is, the tensor
which under the canonical isomorphism
$\gg^* \otimes \gg \to \Hom(\gg,\gg)$
goes to the identity. The diagram
\begin{equation}
\begin{gathered}
\xymatrix
{
M_{\gg^*} \ar[r]^{P_{\bra}\phantom{aaaaaa}}
\ar@{=}[d]
&
M_{\gg^*} \times (\gg^* \otimes \gg^*) 
\\
M_{\gg^*} \ar[r]^{\mathrm I\phantom{aaaaaa}} & M_{\gg^*} \times ( \gg^* \otimes \gg) \ar[u]_{\Id \otimes_{M_{\gg^*}} \fund_{M_{\gg^*}}}
}
\end{gathered}
\end{equation}
is commutative.

Consider an ordinary $\GG$-invariant Poisson tensor $P$ on a
$\GG$-manifold $M$. A   $\GG$-equivariant map
$\Phi \colon M \to \gg^*$
is a momentum mapping for $P$
when it satisfies one of the two equivalent conditions below:

\begin{enumerate}
\item
Every $X \in \gg$
satisfies the identity
\begin{equation}
P^{\sharp}(d (X \circ \Phi)) = X_M.
\end{equation}

\item
The diagram
\begin{equation}
\begin{gathered}
\xymatrixcolsep{4pc}
\xymatrix
{
M \ar[r]^P
\ar@{=}[d]
&
\TT^2 M \ar[r]^{(d\Phi)_M \otimes_M \Id \phantom{aaaaaa}} &(\TT_\Phi \gg^*) \otimes_M \TT M
\\
M \ar[r]^{\mathrm I\phantom{aaaaaa}} & M \times ( \gg^* \otimes \gg) 
\ar[r]_{(d\Phi)_M \otimes_M \Id } 
& (\TT_\Phi\gg^*) \otimes \gg\ar[u]^{\Id \otimes \fund_M}
}
\end{gathered}
\label{CD0Pois}
\end{equation}
is commutative.
\end{enumerate}
}
\end{rema}

\subsection{Quasi Poisson reduction}
\label{qpr}

\begin{thm}
\label{qprt}
Let $(M,P)$
be a $\GG$-quasi Poisson manifold 
(smooth, analytic, algebraic)
relative to the symmetric $\Ad$-invariant
$2$-tensor $\oomega \in \gg \otimes \gg$
and let $\Phi \colon M \to \GG$
be a (smooth, analytic, algebraic) $\GG$-momentum mapping 
for $P$ relative to $\oomega$.
\begin{enumerate}
\item The bracket
{\rm \eqref{pb2}}, viz.
$\{f,h\} = \langle P, df \wedge d h\rangle$, for
$f,h \in \mathcat A[M]$,
yields a Poisson bracket $\pbra$ on
the algebra $\mathcat A[M]^\GG$
of $\GG$-invariant admissible functions on $M$.

\item 
Let $\CcC$ be a conjugacy class in $\GG$
in the image
of $\Phi$.
Then the ideal $I_\CcC$ of  admissible functions
in  $\mathcat A[M]^\GG$
(ideal of $\GG$-invariant admissible functions in $\mathcat A[M]$)
that vanish on
$\Phi^{-1}(\CcC) \subseteq M$
is a Poisson ideal in 
$\mathcat A[M]^\GG$.
Consequently the data determine a Poisson bracket on the quotient algebra
$\left(\mathcat A[M]^\GG\right)/I_\CcC$.
\begin{enumerate}
\item For $\GG$ compact, this yields a Poisson
 algebra of continuous functions
on the orbit space $\Phi^{-1}(\CcC)/\GG$.

\item For $M$ a complex manifold,  $\GG$ complex reductive,
and $\Phi$ holomorphic,
 this yields a Poisson algebra of analytic functions  
on an analytic quotient
of  $\Phi^{-1}(\CcC)$ by $\GG$.

\item
For $M$ a real analytic manifold 
subject to suitable additional hypotheses and $\GG$ real reductive,
 this yields a Poisson  algebra 
on a  quotient
of $\Phi^{-1}(\CcC)$  by $\GG$
of the kind explored in
{\rm \cite{MR364272, MR423398, MR1087217}}.

\item
In the algebraic case,
over an algebraically closed field $\bK$,
for $M$ a non-singular affine variety, $\GG$ reductive,
and $\Phi$ algebraic,
this turns the
affine  categorical quotient
of the kind  $\Phi^{-1}(\CcC)//\GG$
into an affine  Poisson variety
in the sense that
the data determine a Poisson structure on
the affine coordinate ring
$\left(\mathcat A[M]^\GG\right)/I_\CcC$
of $\Phi^{-1}(\CcC)//\GG$.
\end{enumerate}

\end{enumerate}

\end{thm}

\begin{proof}
Claim (1) is Proposition \ref{prop1} above.

To establish (2),
let $f$  be a $\GG$-invariant admissible function
on $M$, and
let  $X_f= \{f,\,\cdot \,\}$, the
{\em quasi Hamiltonian vector field\/}
associated with $f$.
Since $\Phi \colon M \to \GG$ is
a $\GG$-momentum mapping for $P$ relative to $\oomega$,
by \eqref{momapP1},
for $h \colon \GG \to \bK$,
\begin{equation}
X_f(h \circ \Phi)=
P^\sharp(d f) (h \circ \Phi)= 
-
(P^\sharp_\GG (dh) \circ \Phi)_M(f)
\label{show1}
\end{equation}
and, since
$(P^\sharp_\GG (dh) \circ \Phi)_M$ factors through
the fundamental vector field map
$M \times \gg \to \TT M$
and since $f$ is $\GG$-invariant,
the right-hand side of \eqref{show1} vanishes.
Hence, for admissible
$h \colon \GG \to \bK$, the function
$X_f(h \circ \Phi)$ vanishes.

We now suppose $\bK = \RR$ or $\bK = \CC$
and place ourselves
in the classical (smooth or analytic) setting.
The function
$h \circ \Phi$
is then constant along the integral curves of
$X_f$ 
for any admissible function
$h \colon \GG \to \bK$.

Let $p$ be a point of $\GG$
and  $q$  a point of $M$ with $\Phi(q) = p$.
Consider
the integral curve
$t \mapsto \varphi^f_q(t)$ of $X_f$ in $M$
having
$ \varphi^f_q(0) =q$
($t$ in a neighborhood of $0 \in \bK$).
For any admissible function
$h \colon \GG \to \bK$,
since the function
$X_f(h \circ \Phi)$ vanishes,
the $\bK$-valued function
$ h\circ \Phi \circ \varphi^f_q$  
is constant and has
constant value $h (\Phi(q))= h (p) \in \bK$.
Since the admissible functions separate points,
the curve
$\Phi \circ\varphi^f_q$
in $\GG$ is constant and
has constant value $p \in \GG$.
Consequently
the integral curve
$ \varphi^f_q$ of $X_f$ 
lies in
$\Phi^{-1} \Phi(q)= \Phi^{-1}(p)$
whence, for
an admissible  function $F \colon M \to \bK$ that vanishes
on the level subspace $\Phi^{-1}(p)$ of $M$,
the $\bK$-valued function
$ F \circ \varphi^f_q$  
is constant.
Differentiating with respect to the variable $t$ and evaluating at
$t=0$
we find
\begin{equation*}
\{f, F\}(q) =  (X_f F)(q) = 0 .
\end{equation*}

The pre-image $\Phi^{-1}(\CcC)$ is the union of the pre-images
$\Phi^{-1}(p)$ as $p$ ranges over $\CcC$.
Hence a $\GG$-invariant admissible function $F$ 
on $M$ that vanishes at the point $q$ of $M$  
with $\Phi(q)=p \in \CcC$  vanishes 
on  $\Phi^{-1}(\CcC)$.
Consequently, for a $\GG$-invariant  admissible  function 
$F \colon M \to \bK$ that vanishes
on $\Phi^{-1}(\CcC)$,
\begin{equation}
\{f, F\}(q) =  (X_f F)(q) = 0
\label{show}
\end{equation}
whenever $\Phi(q) \in \CcC$.

In the purely algebraic setting we must be more circumspect:
Thus consider a
non-singular affine variety over $\bK$
(not necessarily $\RR$ or $\CC$) 
and let
$\mathcat A= \mathcat A[M]$
be its coordinate ring. Then
the $\mathcat A$-module $\Vect(M)$
of derivations $\Der(\mathcat A)$
is finitely generated and projective as an $\mathcat A$-module.
Suppose $\GG$ is an algebraic group defined over $\bK$
and that the $\GG$-momentum mapping $\Phi \colon M \to \GG$
relative to $\oomega$
is a morphism of $\bK$-varieties.
As above, let $f$ and $F$ be members of
$\mathcat A^\GG$.
The vector field $X_f$ is a member of 
 $\Der(\mathcat A)$.
Extending scalars,
view $M$ as an affine variety over  $\CC$.
The canonical extensions of the functions $f$ and $F$ 
are then admissible
relative to the classical topology
as is the extension of the momentum mapping $\Phi$,
and  we can argue in terms of integral curves as before.
\end{proof}

\begin{rema}
{\rm
The argument in the above proof shows that $X_f$ 
is tangent to the level subspaces of $\Phi$
at any point where being tangent makes sense,
in particular, when we are in the regular case. 

}
\end{rema}

\begin{rema}
{\rm
By a theorem in \cite{MR618321}, 
the radical of an ideal of polynomials
closed under Poisson bracket
is also closed under Poisson bracket.
Perhaps one can use this fact
 to concoct a purely algebraic proof of
Theorem \ref{qprt}
in the algebraic setting.
}
\end{rema}

\begin{rema}
{\rm
In the regular case, for a compact group $\GG$
and an $\Ad$-invariant  non-de\-ge\-ne\-rate symmetric bilinear form on $\gg$,
a similar notion of quasi Poisson reduction is in
\cite[Theorem 6.1 p.~16]{MR1880957} .
}
\end{rema}

\section{Quasi Poisson fusion}
\label{fusion}

In this section we extend the operation of fusion in
\cite[Section 5]{MR1880957}
(for $\GG$ compact and $\oomega$ arising from a positive
$2$-form on $\gg$)
and that in \cite[Section 5]{MR2642360}
(for general  $\GG$  and $\oomega$ arising from a 
non-degenerate
$2$-form on $\gg$)
to our more general setting.

\subsection{Preparations}
\label{fusionprep}

We extend the notation in Subsection \ref{prodtg}:
Thus $\GG^1$ and $\GG^2$ are  Lie groups, and
$\GG^\times =\GG^1 \times \GG^2$ and
$\gg^\times = \gg^1 \oplus \gg^2$.
Let $\oomega^1 \in \gg^1 \otimes \gg^1$,
$\oomega^2 \in \gg^2 \otimes \gg^2$,
be symmetric $\Ad$-invariant
$2$-tensors, and let
\begin{equation}
\oomega^\times = \oomega^1+\oomega^2\in 
\gg^1 \otimes \gg^1 \oplus \gg^2 \otimes \gg^2
\subseteq
(\gg^1 \oplus \gg^2)\otimes (\gg^1 \oplus \gg^2) .
\label{oomegatime}
\end{equation}
In view of \eqref{view1} - \eqref{view4}, the diagrams
\begin{equation}
\begin{gathered}
\xymatrixcolsep{5pc}
\xymatrix{
\GG^\times
\ar[r]^{\oomega^\times\phantom{aaaaa} }
\ar[dr]_{\Id \times \oomega^1 \phantom{aa}}
&{\GG^\times} \times ({\gg^\times} \otimes {\gg^\times}) 
\ar[d]^{\Id \times \pr_{\gg^1 \otimes \gg^1}}
\ar[r]^{\phantom{aaa}(L^\times +R^\times)\otimes \Id }
&
(\TT {\GG^\times}) \otimes \gg^\times 
\ar[d]^{\pr \times  \pr_{\gg^1}}
\\
&
{\GG^\times} \times ({\gg^1} \otimes {\gg^1}) 
\ar[r]_{\phantom{}((L^1 +R^1)\times \Id)\otimes \Id\phantom{aa} }
&
((\TT {\GG^1})\times \GG^2) \otimes \gg^1
\\
\GG^\times 
\ar[r]^{\oomega^\times\phantom{aaaaa} }
\ar[dr]_{\Id \times \oomega^2 \phantom{aa}}
&{\GG^\times} \times ({\gg^\times} \otimes {\gg^\times}) 
\ar[d]^{\Id \times \pr_{\gg^2 \otimes \gg^2}}
\ar[r]^{\phantom{aaa}(L^\times +R^\times)\otimes \Id}
&
\TT {\GG^\times} \otimes \gg^\times 
\ar[d]^{\pr \otimes \pr_{\gg^2}}
\\
&
{\GG^\times} \times ({\gg^2} \otimes {\gg^2}) 
\ar[r]_{(\Id \times (L^2 +R^2))\otimes \Id \phantom{aa}}
&
(\GG^1 \times (\TT {\GG^2})) \otimes \gg^2
}
\end{gathered}
\label{CDs}
\end{equation}
are commutative.

Let $\GG^1 = \GG = \GG^2$ and let $\mult \colon \GG^1 \times \GG^2 \to \GG$
denote the group multiplication. The diagrams
\begin{gather}
\xymatrixcolsep{4pc}
\begin{gathered}
\xymatrix{
\TT_{\Phi^1} \GG^1\ar[r]^{\left(\Id, \tau_{\Phi^1}\right)}\ar[d]_{\tau_{\Phi^1}}
&(\TT_{\Phi^1} \GG^1) \times M \ar[r]^{\can\times\Phi^2}
\ar[d]^{\tau_{\Phi^1} \times\Id\phantom{}} 
&\TT \GG^1 \times \GG^2\ar[r]^{R^2} \ar[d]_{\tau_{\GG^1} \times \Id}
& \TT \GG \ar[d]^{\tau_\GG}
\\
M\ar[r]_{\diag}&M \times M  \ar[r]_{\Phi^1 \times \Phi^2} 
&\GG^1 \times \GG^2 \ar[r]_{\mult}
& \GG
}
\end{gathered}
\label{display25}
\\
\xymatrixcolsep{4pc}
\begin{gathered}
\xymatrix{
\TT_{\Phi^2} \GG^2\ar[r]^{\left(\tau_{\Phi^2},  \Id\right)}\ar[d]_{\tau_{\Phi^2}}
&M \times (\TT_{\Phi^2} \GG^2)  \ar[r]^{\Phi^1 \times \can}
\ar[d]^{\Id \times \tau_{\Phi^2}\phantom{}} 
&\GG^1 \times (\TT \GG^2)\ar[r]^{L^1} \ar[d]_{\Id\times \tau_{\GG^2}}
& \TT \GG \ar[d]^{\tau_\GG}
\\
M\ar[r]_{\diag}&M \times M  \ar[r]_{\Phi^1 \times \Phi^2} 
&\GG^1 \times \GG^2 \ar[r]_{\mult}
& \GG
}
\end{gathered}
\label{display26}
\end{gather}
are commutative. Exploiting this commutativity,
we now extend the operations of right and left translation in
{\rm \eqref{display}} as follows:

\begin{prop}
Right and left translation induce isomorphims
\begin{equation}
R_{\Phi^2}\colon \TT_{\Phi^1}\GG^1 \longrightarrow \TT_{\Phi^1\Phi^2} \GG,
\ 
L_{\Phi^1}\colon \TT_{\Phi^2}\GG^2 \longrightarrow \TT_{\Phi^1\Phi^2} \GG
\end{equation}
of vector bundles on $M$
as displayed in 
\begin{gather}
\xymatrixcolsep{4pc}
\begin{gathered}
\xymatrix{
\TT_{\Phi^1} \GG^1
\ar[r]^{\left(\Id, \tau_{\Phi^1}\right)}
\ar@{.>}[dr]|-{R_{\Phi^2}}
\ar@/_/[ddr]_{\tau_{\Phi^1}}
&(\TT_{\Phi^1} \GG^1) \times M \ar[r]^{\can\times\Phi^2}
&\TT \GG^1 \times \GG^2\ar[dr]^{R^2} 
& 
\\
&\TT_{\Phi^1 \Phi^2} \GG \ar[rr]^\can\ar[d]^{\tau_{\Phi^1\Phi^2}}  
& 
& \TT \GG \ar[d]^{\tau_\GG}
\\
&M \ar[rr]_{\Phi^1 \Phi^2} 
&
& \GG
}
\end{gathered}
\label{display27}
\\
\xymatrixcolsep{4pc}
\begin{gathered}
\xymatrix{
\TT_{\Phi^2} \GG^2
\ar[r]^{\left(\tau_{\Phi^2}, \Id\right) \phantom{aaa}}
\ar@{.>}[dr]|-{L_{\Phi^1}}
\ar@/_/[ddr]_{\tau_{\Phi^2}}
&M \times (\TT_{\Phi^2} \GG^2) \ar[r]^{\Phi^1 \times \can}
&\GG^1 \times (\TT \GG^2)\ar[dr]^{L^1} 
& 
\\
&\TT_{\Phi^1 \Phi^2} \GG \ar[rr]^\can\ar[d]^{\tau_{\Phi^1\Phi^2}}  
& 
& \TT \GG \ar[d]^{\tau_\GG}
\\
&M \ar[rr]_{\Phi^1 \Phi^2} 
&
& \GG 
}
\end{gathered}
\label{display28}
\end{gather}
in such a way that
\begin{align}
L_{\Phi^1} \circ L_{\Phi^2} =L_{\Phi^1\Phi^2} &
\colon M \times \gg \longrightarrow \TT_{\Phi^1\Phi^2} \GG
\\
R_{\Phi^2} \circ R_{\Phi^1} =R_{\Phi^1\Phi^2} &
\colon M \times \gg \longrightarrow \TT_{\Phi^1\Phi^2} \GG . \qed
\end{align}

\end{prop}

Let $(\TT_{\Phi^1} \GG^1) \times_M (\TT_{\Phi^2} \GG^2)$
denote the fiber product
of $\TT_{\Phi^1} \GG^1$ and $\TT_{\Phi^2} \GG^2$
over $M$ and $(\TT_{\Phi^1} \GG^1) \oplus_M (\TT_{\Phi^2} \GG^2)$
the total space of the Whitney sum of the
two vector bundles on $M$ under discussion.
In view of completely formal properties 
of the pullback construction, the canonical map
\begin{equation}
(\TT_{\Phi^1} \GG^1) \times_M (\TT_{\Phi^2} \GG^2) 
\longrightarrow
\TT_{\Phi^1,\Phi^2} \GG^\times
\end{equation}
is a vector bundle  isomorphism over $M$.
Moreover, the obvious vector bundle injections of
$\TT_{\Phi^1} \GG^1$ and $\TT_{\Phi^2} \GG^2$
into  $(\TT_{\Phi^1} \GG^1) \times_M (\TT_{\Phi^2} \GG^2)$
induce an isomorphism
\begin{equation}
(\TT_{\Phi^1} \GG^1) \oplus_M (\TT_{\Phi^2} \GG^2) 
\longrightarrow
(\TT_{\Phi^1} \GG^1) \times_M (\TT_{\Phi^2} \GG^2) 
\label{rew1}
\end{equation}
of vector bundles on $M$, and
the derivative
\begin{equation}
(\TT \GG^1) \times \GG^2 \oplus_{\GG^\times} \GG^1 \times (\TT \GG^2) \longrightarrow \TT \GG
\end{equation}
of the multiplication map of $\GG$
induces the morphism
\begin{equation}
R_{\Phi^2} +L_{\Phi^1}\colon  (\TT_{\Phi^1} \GG^1) \oplus_M (\TT_{\Phi^2} \GG^2) 
\longrightarrow
\TT_{\Phi^1\Phi^2} \GG
\label{morvecb}
\end{equation}
of vector bundles on $M$.

\subsection{Fusion}
Consider the product group $\GG \times \GG$
with componentwise conjugation action on itself.
This is the situation in Example \ref{examp1},
with $\GG \times \GG$ substituted for $\GG$.
As before, for
book keeping purposes, write
the first copy of $\GG$ 
in $\GG \times \GG$
as $\GG^1$ and the second one as $\GG^2$ and, accordingly,
the first copy of $\gg$ 
in $\gg \oplus \gg$
as $\gg^1$ and the second one as $\gg^2$, and
let $\oomega^1 \in \gg^1 \otimes \gg^1$
and $\oomega^2 \in \gg^2 \otimes \gg^2$
denote the corresponding copy of $\oomega \in \gg \otimes \gg$.
Let 
\begin{equation}
\oomega^\times = \oomega^1+\oomega^2\in 
\gg^1 \otimes \gg^1 \oplus \gg^2 \otimes \gg^2
\subseteq
(\gg^1 \oplus \gg^2)\otimes (\gg^1 \oplus \gg^2) .
\label{oomegatimes}
\end{equation}
This yields 
$P_{\GG^1 \times \GG^2}= P_{\GG^1} +P_{\GG^2} 
(= - \chi_{\oomega,\GG^1} -\chi_{\oomega,\GG^2}) 
\in \LAL^{\mrc,2}[\GG^1 \times \GG^2]$.
Moreover,
the $(\GG^1 \times \GG^2)$-module
 $\LAL^{\mrc,3}[\gg^1 \oplus \gg^2]$ decomposes canonically as
\begin{equation}
\LAL^{\mrc,3}[\gg^1 \oplus \gg^2] 
= \LAL^{\mrc,3}[\gg^1]\oplus
\gg^1 \boxwedge \gg^2 \oplus \gg^2 \boxwedge \gg^1
\oplus \LAL^{\mrc,3}[\gg^2],
\label{canonically2}
\end{equation}
so that
\begin{align*}
\phi_{\gg^1 \oplus \gg^2} &=  \phi^1 + \phi^2 \in 
\LAL^{\mrc,3}[\gg^1]\oplus \LAL^{\mrc,3}[\gg^2],
\\
\phi_{\GG^1 \times \GG^2} &=  \phi_\GG^1 + \phi_\GG^2
\in \LAL^{\mrc,3}[\GG^1 \times \GG^2],
\\
[P_{\GG^1 \times \GG^2},P_{\GG^1 \times \GG^2}]&= 
[P_{\GG^1},P_{\GG^1}]  +[P_{\GG^2},P_{\GG^2}]=    \phi_\GG^1 + \phi_\GG^2 .
\end{align*}

Consider a $(\GG^1 \times \GG^2)$-quasi Poisson structure
$P$
on a $(\GG^1 \times \GG^2)$-manifold
$M$ relative to $\oomega^\times$.
Let $\chi_{\oomega,M} \in \LAL^{\mrc,2}[M]$
be the image of 
$\chi_{\oomega} \in \gg^1 \boxwedge \gg^2  \subseteq\LAL^{\mrc,2}[ \gg^1 \oplus \gg^2]$
under the resulting  infinitesimal $(\gg^1 \oplus \gg^2)$-action
$\gg^1 \oplus \gg^2 \to \Vect(M)$
of $\gg^1 \oplus \gg^2$ on $M$.
In terms of the restrictions
$\fund_M^1\colon M \times \gg^1 \to \TT M$ and
$\fund_M^2\colon M \times \gg^2 \to \TT M$ 
of the infinitesimal action
of $\gg^1 \oplus \gg^2$ on $M$,
\begin{equation}
2 \chi_{\oomega,M} = (\fund_M^1 \wedge \fund_M^2)(\oomega)
= (\fund_M^1 \otimes \fund_M^2 -\fund_M^2 \otimes \fund_M^1)(\oomega).
\end{equation}

\begin{thm}
\label{fusmom}
Let $M$ be a $\GG^\times$-manifold and
 $P$   a $\GG^\times $-invariant skew symmetric  bivector
on
$M$. Further, let
$(\Phi^1,\Phi^2)\colon M \to \GG^1 \times \GG^2 = \GG^\times$
be an admissible $\GG^\times$ equivariant map, and let
\begin{equation}
P_\fus = P-\chi_{\oomega,M} .
\label{Pfus}
\end{equation}
\begin{enumerate}
\item
When $(\Phi^1,\Phi^2)$ is a 
$\GG^\times$-momentum mapping for $P$ relative to $\oomega^\times$, 
with respect to  the diagonal $\GG$-action on $M$, the product 
 $\Phi^1\Phi^2\colon M \to \GG$ is a $\GG$-momentum mapping
for $P_\fus$ relative to $\oomega$ .
\item
When $P$ is a $\GG^\times$-quasi Poisson structure
relative to $\oomega^\times$,
with respect to
the
diagonal $\GG$-action on $M$, the bivector
$P_\fus$
is a $\GG$-quasi Poisson structure on $M$ relative to $\oomega$.
\end{enumerate}
\end{thm}

\begin{proof}
By \eqref{calc1},
\begin{align*}
[P-\chi_{\oomega,M},P-\chi_{\oomega,M}]&= [P,P] - [\chi_{\oomega,M},P]-[P,\chi_{\oomega,M}] +[\chi_{\oomega,M},\chi_{\oomega,M}]
\\
&= \phi_M^1 + \phi_M^2 + \Delta (\phi)_M - \phi_M^1 - \phi_M^2 .
\end{align*}
Since $P$ is invariant under
$\GG^1 \times \GG^2$,
the terms $[P,\chi_{\oomega,M}]$  and $[\chi_{\oomega,M},P]$ vanish
and
$\Delta (\phi)_M$ coincides with the image 
$\phi^\diag_M$ 
of $\phi \in \LAL^{\mrc,3}[\gg]$
under the diagonal action of $\GG$ on $M$.
This proves (2).

To establish (1), we note first that,
since $(\Phi^1,\Phi^2)\colon M \to \GG^1 \times \GG^2$
is a $(\GG^1 \times \GG^2)$-momentum for the
$(\GG^1 \times \GG^2)$-quasi Poisson structure
$P$ on $M$,  the diagram
\begin{equation}
\begin{gathered}
\xymatrixcolsep{4.5pc}
\xymatrix{
M \ar@{=}[r] 
\ar[d]|-{2 P}
&M
\ar[d]|-{\oomega^\times}
\\
\TT^2 M \ar[d]|-{(d( {\Phi^1,\Phi^2}))_M \otimes_M \Id}
& M \times (\gg^\times \otimes \gg^\times)
\ar[d]|-{\left(L_{(\Phi^1,\Phi^2)}+R_{(\Phi^1,\Phi^2)}\right)\otimes_M \Id}
\\
(\TT_{(\Phi^1,\Phi^2)} \GG) \otimes_M \TT M 
&
(\TT_{(\Phi^1,\Phi^2)}\GG) \otimes \gg^\times
\ar[l]^{\phantom{aa}\Id \otimes_M \fund^\times_M}
}
\end{gathered}
\label{PMPhimultimesvar}
\end{equation}
commutes.
According to the decomposition \eqref{decomp1}, the right-hand vertical  
row of 
\eqref{PMPhimultimesvar} is the sum of the four respective constituents.
However, the canonical map
\begin{equation}
(\gg^1 \otimes \gg^1) \oplus 
(\gg^1 \otimes \gg^2) \oplus
(\gg^2 \otimes \gg^1) \oplus 
(\gg^2\otimes \gg^2) \longrightarrow
(\gg^1 \oplus \gg^2)\otimes(\gg^1 \oplus \gg^2)
\end{equation}
is an isomorphism, and the symmetric tensor $\oomega^\times$ lies in (the image of)
$(\gg^1 \otimes \gg^1) \oplus(\gg^2\otimes \gg^2)$. Hence at most the first 
two  constituents in \eqref{constit1} yield non-trivial contributions to the 
right-hand  vertical row of  \eqref{PMPhimultimesvar}.

Exploiting the diagrams \eqref{CDs} being commutative,
we rewrite  \eqref{PMPhimultimesvar}
in terms of \eqref{rew1}
as the outermost  diagram
of
\begin{equation}
\begin{gathered}
{
\xymatrix{
M\ar@{=}[r] 
\ar[d]|-{2P}
&M\ar[d]|-{\oomega}
\\
\TT^2 M \ar@/_5.5pc/[ddd]|-{\left((d\Phi^1)_M +(d\Phi^2)_M\right) \otimes_M \Id}
 \ar[d]|-{(d( {\Phi^1\Phi^2}))_M \otimes_M \Id}
&
 M \times (\gg \otimes \gg)
\ar@{.>}[d]
\ar@/^7.5pc/[ddd]|-{\left(\left(L_{\Phi^1}+R_{\Phi^1}\right)\otimes \iota^1,
\left(L_{\Phi^2}+R_{\Phi^2}\right)\otimes \iota^2
\right)}
\\
(\TT_{\Phi^1\Phi^2} \GG) \otimes_M \TT M 
&
(\TT_{\Phi^1\Phi^2}\GG) \otimes (\gg^1 \oplus \gg^2)
\ar[l]|-{\phantom{aa}\Id \otimes_M \fund^\times_M}
\\
&
\\
(\TT_{\Phi^1} \GG^1 \oplus_M \TT_{\Phi^2} \GG^2) \otimes_M \TT M 
\ar[uu]|-{\left(R_{\Phi^2} + L_{\Phi^1}\right)\otimes _M \Id}
&
(\TT_{\Phi^1} \GG^1 \oplus_M \TT_{\Phi^2} \GG^2)\otimes(\gg^1 \oplus \gg^2)  .
\ar@/^2pc/[l]^{\phantom{aa}\Id \otimes_M \fund^\times_M \phantom{aaa}}
\ar[uu]|-{\left(R_{\Phi^2} + L_{\Phi^1}\right)\otimes\Id}
}
}
\end{gathered}
\label{Pbig}
\end{equation}
This diagram is commutative,
but inserting 
\begin{equation}
\left(L_{\Phi^1\Phi^2}+R_{\Phi^1\Phi^2}\right)\otimes_M \diag
\colon
M \times (\gg \otimes \gg)
\longrightarrow (\TT_{\Phi^1 \Phi^2}\GG) \otimes (\gg^1 \oplus \gg^2)
\end{equation}
for the dotted arrow
does not render it commutative.
Let
\begin{align}
A^{2,1}& = R_{\Phi^2}\circ (L_{\Phi^1}+R_{\Phi^1}) \colon M \times \gg 
\longrightarrow
\TT_{\Phi^1\Phi^2}\GG,
\\
A^{1,2}& = L_{\Phi^1}\circ (L_{\Phi^2}+R_{\Phi^2}) \colon M \times \gg 
\longrightarrow
\TT_{\Phi^1\Phi^2}\GG.
\end{align}
The commutativity of the above diagram tells us that
the  diagram
\begin{equation}
\begin{gathered}
\xymatrixcolsep{6pc}
\xymatrix{
M \ar[r]^{2 P}\ar@{=}[d]  
& 
\TT^2 M\ar[r]^{(d( {\Phi^1\Phi^2}))_M \otimes_M \Id\phantom{aaaa}}
& 
(\TT_{\Phi^1\Phi^2} \GG)\otimes_M \TT M 
\\
M \ar[r]_{\oomega\phantom {aaaa}}&M \times (\gg\otimes \gg) 
\ar[r]_{
\left(
A^{2,1}\otimes \iota^1, A^{1,2}\otimes \iota^2
\right) \phantom{aaaa}
}
&
(\TT_{\Phi^1 \Phi^2} \GG) \otimes (\gg^1 \oplus \gg^2)
\ar[u]_{\Id \otimes_M \fund^\times_M} .
}
\end{gathered}
\label{PMPhimuvar}
\end{equation}
is commutative.
Let
\begin{align}
B^{2,1}&=    L_{\Phi^1}\circ L_{\Phi^2} - R_{\Phi^2}\circ L_{\Phi^1}\colon M \times \gg 
\longrightarrow
\TT_{\Phi^1\Phi^2}\GG,
\\
B^{1,2}&=  L_{\Phi^1}\circ R_{\Phi^2} -R_{\Phi^2}\circ R_{\Phi^1} 
\colon M \times \gg 
\longrightarrow
\TT_{\Phi^1\Phi^2}\GG.
\end{align}
Then
\begin{align}
A^{2,1}+B^{2,1}&= 
 L_{\Phi^1}\circ L_{\Phi^2}+R_{\Phi^2} \circ R_{\Phi^1}
= A^{1,2}- B^{1,2} \colon M \times\gg
\longrightarrow
\TT_{\Phi^1\Phi^2}\GG
\label{view111}
\\
&= 
 L_{\Phi^1 \Phi^2}+ R_{\Phi^1\Phi^2}\colon M \times\gg
\longrightarrow
\TT_{\Phi^1\Phi^2}\GG .
\label{view121}
\end{align}

The same kind of reasoning which leads to the commutative diagram
\eqref{PMPhimuvar} 
shows that the diagrams
\begin{align}
\begin{gathered}
\xymatrixcolsep{6pc}
\xymatrix{
\ar@{=}[d] 
M \times (\gg\otimes \gg) 
\ar[r]^{\fund_M^1 \otimes \fund_M^2}
& 
\TT^2 M\ar[r]^{(d( {\Phi^1\Phi^2}))_M \otimes_M \Id\phantom{aaaaaaaa}}
& 
(\TT_{\Phi^1\Phi^2} \GG)\otimes_M \TT M 
\\
M \times (\gg\otimes \gg) 
\ar[rr]_{B^{1,2}\otimes \iota^2 \phantom{aaaa}
}
&&
(\TT_{\Phi^1 \Phi^2} \GG) \otimes  \gg^2
\ar[u]_{\Id \otimes_M \fund_M^2} 
}
\end{gathered}
\label{PMPhimua1}
\\
\begin{gathered}
\xymatrixcolsep{6pc}
\xymatrix{
\ar@{=}[d] 
M \times (\gg\otimes \gg) 
\ar[r]^{\fund_M^2 \otimes \fund_M^1}
& 
\TT^2 M\ar[r]^{(d( {\Phi^1\Phi^2}))_M \otimes_M \Id\phantom{aaaaaaaa}}
& 
(\TT_{\Phi^1\Phi^2} \GG)\otimes_M \TT M 
\\
M \times (\gg\otimes \gg) 
\ar[rr]_{B^{2,1}\otimes \iota^1 \phantom{aaaa}
}
&&
(\TT_{\Phi^1 \Phi^2} \GG) \otimes  \gg^1
\ar[u]_{\Id \otimes_M \fund_M^1} 
}
\end{gathered}
\label{PMPhimua2}
\end{align}
are commutative.

Since $\chi_{\oomega,M} = (\fund_M^1 \wedge \fund_M^2)(\oomega) = 
 (\fund_M^1 \otimes \fund_M^2)(\oomega) - (\fund_M^2 \otimes \fund_M^1)(\oomega)
$, in view of \eqref{view111} and \eqref{view121}, 
since $\iota^1 + \iota^2 = \diag \colon \gg \to \gg^1 \oplus \gg^2$,
we conclude that
the diagram
\begin{equation}
\begin{gathered}
\xymatrixcolsep{6pc}
\xymatrix{
M \ar[r]^{2 P- 2 \chi_{\oomega,M}}\ar@{=}[d]  
& 
\TT^2 M\ar[r]^{(d( {\Phi^1\Phi^2}))_M \otimes_M \Id\phantom{aaaa}}
& 
(\TT_{\Phi^1\Phi^2} \GG)\otimes_M \TT M 
\\
M \ar[r]_{\oomega\phantom {aaaa}}&M \times (\gg\otimes \gg) 
\ar[r]_{
\left(
\left(L_{\Phi^1 \Phi^2}+ R_{\Phi^1\Phi^2}\right) 
\otimes \diag
\right) 
\phantom{aaaa}
}
&
(\TT_{\Phi^1 \Phi^2} \GG) \otimes (\gg^1 \oplus \gg^2)
\ar[u]_{\Id \otimes_M (\fund_M^\times)} 
}
\end{gathered}
\label{PMP}
\end{equation}
is commutative.
This shows that 
the product $\Phi^1 \Phi^2 \colon M \to \GG$
is a $\GG$-momentum mapping for $P_\fus$ relative to $\oomega$.
\end{proof}

\subsection{The group $\GG$ revisited} 
\label{revisited}

Under the situation of Example \ref{examp2},
applying the operation of fusion 
to the zero $(\GG \times \GG)$-quasi Poisson structure on $\GG$ 
relative to $\oomega^\times$
recovers,
with respect  to the conjugation action of $\GG$ on itself,
the $\GG$-quasi Poisson structure $P_\GG$ 
relative to $\oomega$
in Example \ref{examp1}.

\subsection{Double}
\label{doublep}

We extend the notion of double in
\cite{MR1880957}
(for $\GG$ compact and $\oomega$ arising from a positive
$2$-form on $\gg$)
and that in \cite{MR2642360}
(for general  $\GG$  and $\oomega$ arising from a 
non-degenerate
$2$-form on $\gg$)
to our more general setting.

Consider the product group $\GG^\times =\GG \times \GG$
and maintain the notation in Subsection \ref{predouble}. 
We denote
by $\ovoomega \in \ogg \otimes \ogg$
the corresponding $\Ad$-invariant symmetric $2$-tensor.
The $\Ad$-invariant symmetric $2$-tensor \eqref{oomegatimes} then reads
\begin{equation}
\oomega + \ovoomega \in \gg \otimes \gg \oplus 
\ogg \otimes \ogg \subseteq 
(\gg \oplus \ogg)\otimes (\gg \oplus \ogg).
\end{equation}
Here is the precise analog of Proposition \ref{momfusGs}.

\begin{thm}
\label{momfusG}
With respect to the action {\rm \eqref{act3}}
of the product group $ \GG\times \oGG$ on 
$\GG^\times=\GG^1\times \GG^2$,
the bivector field
\begin{equation}
\xymatrixcolsep{4pc}
\xymatrix{
P_{\oomega}^\times 
\colon
\GG^\times
\ar[r(0.6)]^{\phantom{aaaaaaa}\oomega}
& \GG^\times \times (\gg \otimes \gg)
\ar[r(1)]^
{\phantom{aaaaaaa} \tfrac 12 
\left(L^1 \wedge R^2 + R^1 \wedge L^2\right)}
&
\thickspace\thickspace\thickspace\thickspace\thickspace
\thickspace\thickspace\thickspace\thickspace\thickspace
\thickspace\thickspace\thickspace\thickspace\thickspace\thickspace
\TT^2 \GG^\times
}
\label{Ptimes}
\end{equation}
on $\GG^\times$
is a  $(\GG \times \oGG)$-quasi Poisson structure 
relative to $\oomega + \ovoomega$,
 and $(\mult,\omult)$, cf. \eqref{momfusGGoldd},
is a $(\GG \times \oGG)$-momentum mapping for 
$P_\oomega^\times$ relative to $\oomega + \ovoomega$,
the action of $ \GG \times \oGG$ on itself
being by conjugation.
\end{thm}

The  Hamiltonian $(\GG \times \oGG)$-quasi
Poisson manifold  $(\GG^1 \times \GG^2,P_{\oomega}^\times,(\mult,\omult))$ 
relative to 
$\oomega+ \ovoomega$
is the (external) {\em Hamiltonian   quasi 
Poisson double\/}
of $(\GG,\oomega)$.

The inversion map $\inv\colon \GG^\times \to \GG^\times$
induces a morphism 
$(d\, \inv)^\sharp \colon \TT_{\omult} \oGG \to \TT_{\mult} \GG$
of vector bundles on $\GG^\times$.
Let 
\begin{align}
\Theta^\times &=(d\,\inv)^\sharp \otimes_{\GG^\times} \Id
\colon \TT_{\omult} \oGG \otimes_{\GG^\times} \TT \GG^\times
\longrightarrow \TT_{\mult} \GG \otimes_{\GG^\times} \TT \GG^\times
\\
\Theta_\gg &=(d\,\inv)^\sharp \otimes_{\GG^\times} \Id
\colon \TT_{\omult} \oGG \otimes \ogg
\longrightarrow \TT_{\mult} \GG \otimes \gg .
\end{align}
The following is the precise analog of Proposition \ref{techn},
and the proof is essentially the same.

\begin{prop}
\label{technn}
The diagram
\begin{equation}
\begin{gathered}
\scalefont{0.7}
\xymatrix{
\GG^\times 
\ar[dddd]_{\oomega}
\ar[rr]^{\oomega\phantom{aaaaaa}}
\ar@{=}[dr]
&&
\GG^\times \times (\gg\otimes \gg)
\ar[rr]^{\phantom{aaa}L^1 \wedge R^2 + R^1\wedge L^2\phantom{}}
&&
\TT^2 \GG^\times
\ar[dd]|-{\left(d\, {\mult}\right)_{\GG^\times} \otimes _{\GG^\times} \Id}
\ar@{=}[dl]
\\
&\GG^\times 
\ar[dd]_{\ovoomega}
\ar[r]^{\oomega\phantom{aaaaaa}}
&
\GG^\times \times (\gg\otimes \gg)
\ar[r]^{\phantom{aaa}L^1 \wedge R^2 + R^1\wedge L^2\phantom{}}
&
\TT^2 \GG^\times
\ar[d]|-{\left(d\, {\omult}\right)_{\GG^\times} \otimes _{\GG^\times} \Id}
&
\\
&& &
\left(\TT_{\omult}\oGG\right) \otimes_{\GG^\times}\left(\TT \GG^\times\right)
\ar[r]^{\Theta^\times}
&
\left(\TT_{\mult}\GG\right) \otimes_{\GG^\times}\left(\TT \GG^\times\right)
\\
&\GG^\times \times \left( \ogg \otimes \ogg\right)
\ar[rr]_{\left(L_{\omult} + R_{\omult}\right)\otimes \Id\phantom
{aaaa}}
&&
\left(\TT_{\omult}\oGG\right)
\otimes \ogg
\ar[u]|-{\Id \otimes_{\GG^\times} \ofund_{\GG^\times} }
\ar[dr]_{\Theta_\gg}
&
\\
\GG^\times \times \left( \gg \otimes \gg\right)
\ar[rrrr]_{\left(L_{\mult} + R_{\mult}\right)\otimes \Id\phantom
{aaaa}}
\ar@{=}[ur]
&&&&
\left(\TT_{\mult}\GG\right)
\otimes \gg
\ar[uu]|-{\Id \otimes_{\GG^\times} \fund_{\GG^\times} }
}
\end{gathered}
\label{unveil03}
\end{equation}
is commutative. \qed
\end{prop}

\begin{proof}[Proof of Theorem {\rm \ref{momfusG}}]
The reader will readily verify directly that
$P_{\oomega}^\times$ is a $(\GG \times \oGG)$-quasi Poisson
structure on $\GG^\times$
relative to $\oomega + \ovoomega$.
An alternate reasoning for the latter  goes as follows:
Consider the 
$(\GG^\times\times \GG^\times)$-action
\begin{equation}
\begin{aligned}
\GG^\times\times \GG^\times 
\times \GG^\times
&\longrightarrow
 \GG^\times
\\
((x_1,x_2),(y_1,y_2),(q_1,q_2)) &\longmapsto 
(x_1,x_2)(q_1,q_2)(y_1,y_2)^{-1}
=
(x_1 q_1 y_1^{-1}, y_2 q_2 x_2^{-1})
\end{aligned}
\label{act1}
\end{equation}
on $\GG^\times$.
With respect
 to this action,
the zero bivector 
 is  a 
 $(\GG^\times\times \GG^\times)$-quasi Poisson 
structure on $\GG^\times$ relative to
$\oomega^{\times \times}$, cf. \eqref{oomegatimes},
similarly as in Example \ref{examp2}.
The composite of  \eqref{act1} with the product 
$\Delta \times \Delta \colon \GG \times \GG \to \GG^\times \times \GG^\times$
of the diagonal maps
yields \eqref{act3}.
Fusing with respect to the first diagonal yields,
with respect to the resulting $(\GG \times \GG^\times)$-action on
$\GG^ \times$,
the 
$(\GG \times \GG^\times)$-quasi Poisson structure
$\tfrac 12 (L^1 \wedge R^2)(\oomega)$
on $\GG^\times$ relative to
$\oomega^{1,\times}$.
Fusing thereafter with respect to the second diagonal yields,
with respect to the resulting $(\GG \times \GG)$-action on
$\GG^ \times$, by construction precisely \eqref{act3},
the 
$(\GG \times \GG)$-quasi Poisson structure
$P^\times_{\oomega}=\tfrac 12 (L^1 \wedge R^2 + R^1 \wedge L^2)(\oomega)$
on $\GG^\times$ relative to
$\oomega^{\times}$.

The outermost diagram
of \eqref{unveil03} 
being commutative says that the map
$\mult \colon \GG^\times \to \GG$
is a $\GG$-momentum mapping for
$P^\times_\oomega$ relative to $\oomega$
and the innermost diagram
of \eqref{unveil03} 
being commutative says that
$\omult \colon \GG^\times \to \oGG$
is a $\oGG$-momentum mapping for
$P^\times_\oomega$ relative to $\widetilde \oomega$.
Consequently
$(\mult,\omult)$, cf.
\eqref{momfusGGoldd}, 
is a $(\GG \times \oGG)$-momentum mapping for $P^\times_\oomega$
relative to $\oomega +\widetilde \oomega$.
\end{proof}

\subsection{Internally fused double}
\label{internallyfused}
In view of  Theorem \ref{momfusG},
substitute $(\GG \times \GG, P^\times_\oomega, \mult, \omult)$ 
 for $(M,P,\Phi^1,\Phi^2) $ in
Theorem \ref{fusmom};
by that Theorem,
with respect to pairwise conjugation
of $\GG$ on $\GG \times \GG$,
the bivector $P_1=P_\fus =P^\times_\oomega - \chi_{\oomega,\GG \times \GG}$
yields a
$\GG$-quasi Poisson structure on $\GG \times \GG$ relative to $\oomega$,
 and
\begin{equation}
\xymatrixcolsep{4pc}
\xymatrix{
\Phi_1 \colon \GG \times \GG \ar[r]^{\phantom{a} (\mult,\omult)}
 &\GG \times \GG \ar[r]^{\mult} &\GG
}
\label{mom11}
\end{equation}
is a $\GG$-momentum mapping for
$P_1$
relative to $\oomega$.
The pieces of structure $P_1$ and $\Phi_1$ yield 
the {\em internally fused double\/}
$(\GG \times \GG, P_1,\Phi_1)$
of $\GG$
in the realm of  quasi Poisson  structures.

\section{Momentum duality and non-degeneracy}
\label{comparison}
\subsection{General case}
\label{gencase}
Let $M$ be  a $\GG$-manifold $M$ and $\Phi \colon M \to \GG$ an admissible  
$\GG$-equivariant map. We denote by  $\rho_\Phi \colon \TT M \to \TT M$
the morphism
\begin{equation}
\xymatrixcolsep{4pc}
\xymatrix
{
\rho_\Phi\colon \TT M  \ar[r]^{\phantom{aa}(d\Phi)_M} &
\TT_\Phi \GG
\ar[r]^{L_\Phi^{-1} - R_\Phi^{-1} }
&
M \times \gg \ar[r]^{\phantom{aa}\fund_M}
&\TT M 
}
\label{3rhowrite}
\end{equation}
of vector bundles on $M$. 
The dual thereof reads
\begin{equation}
\xymatrixcolsep{4.5pc}
\xymatrix
{
\rho^*_\Phi\colon \TT^* M \ar[r]^{\mathrm{fund^*_M}}  &M \times \gg^*
\ar[r]^{\phantom{aa}L^{*,-1}_\Phi - R^{*,-1}_\Phi}
&\TT^*_\Phi \GG \ar[r]^{(d \Phi)_M^*}
&\TT^* M .
}
\label{3rhod}
\end{equation}
The morphisms
\begin{align}
(\Id - \tfrac 14 \rho_\Phi) + \fund_M &\colon \TT M \oplus_M M \times \gg
\longrightarrow \TT M
\label{surj3}
\\
(\Id - \tfrac 14 \rho^*_\Phi) + (d\Phi)_M^* &\colon \TT^* M \oplus_M \TT^*_\Phi \GG
\longrightarrow \TT^* M
\label{surj4}
\end{align}
of vector bundles on $M$
are manifestly epimorphisms.

Let $\sigma$ be a  $\GG$-invariant $2$-form on $M$
and $P$  a 
$\GG$-invariant skew symmetric $2$-tensor on $M$.
Consider the adjoints $P^\sharp \colon \TT^* M \to \TT M$
of $P$ 
 and $\sigma^\fflat \colon \TT M \to \TT^* M$ of $\sigma$,
cf. Subsection \ref{wqH} and
\eqref{Psharp}.
We say that
$\sigma$ and $P$ are $\Phi$-{\em dual to each other\/} 
or {\em dual to each other
via the admissible $\GG$-equivariant map\/} 
$\Phi\colon M \to \GG$ 
when 
the adjoints $P^\sharp$
 and $\sigma^\fflat$
satisfy the identity
\begin{equation}
P^\sharp \circ \sigma^\fflat=
\Id_{\TT M} - \tfrac 14 \rho_\Phi \colon   \TT M \longrightarrow \TT M.
\label{31}
\end{equation}

Since, for a skew-symmetric $2$-tensor  $P$ and a $2$-form $\sigma$,
the dual $P^{\sharp,*}\colon \TT^* M \to \TT M$ 
of $P^\sharp$ coincides with $-P^\sharp$
and the dual
$\sigma^{\flat,*} \colon \TT M \to \TT^* M$
of $\sigma^\fflat$
coincides with $- \sigma^\fflat$,
the following is is immediate.

\begin{prop}
\label{associated1}
The admissible $\GG$-equivariant map $\Phi \colon M \to \GG$ being given,
the  $\GG$-invariant $2$-form $\sigma$ 
and the  skew-symmetric  
$\GG$-invariant $2$-tensor $P$ on $M$
are $\mmu$-dual to each other if and only if
the adjoints $P^\sharp$
 and $\sigma^\fflat$
satisfy the identity 
\begin{gather}
 \sigma^\fflat \circ P^\sharp=
\Id_{\TT^* M} - \tfrac 14 \rho^*_\Phi \colon   \TT^* M \longrightarrow \TT^* M.
\label{31d}
\end{gather}
 \end{prop}

\begin{rema}
{\rm
For a $2$-tensor $P$ and a $1$-form $\alpha$ on $M$,
the composite
\begin{equation*}
\langle P,\alpha\rangle \colon 
P \stackrel P\longrightarrow
\TT M \otimes_M \TT M
\stackrel {\alpha \otimes \Id}
\longrightarrow \bK \otimes \TT M
\cong \TT M
\end{equation*}
characterizes  a vector field on $M$.
The $2$-form  $\sigma$ and the alternating $2$-tensor $P$ are 
$\Phi$-dual to each other if and only if, for any $1$-form
$\alpha \colon \TT M \to \bK$,
the diagram
\begin{equation*}
\begin{gathered}
\xymatrix{
\TT M \ar[r]^{\Id - \tfrac 14 \rho_\Phi}
\ar[d]_{\langle P,\alpha\rangle \otimes_M \Id}
& \TT M
\ar[d]^\alpha
\\
\TT M \otimes _M \TT M
\ar[r]_{\phantom{aaaaaa}\sigma}
&
\bK
}
\end{gathered}
\end{equation*}
is commutative; indeed, the composite
$\sigma \circ (\langle P,\alpha\rangle \otimes_M \Id)$
recovers $(\sigma^\fflat \circ P^\sharp)(\alpha)$.
This characterization of  $\Phi$-duality
does not
explicitly  refer to the cotangent bundle of $M$
and does, perhaps, still work in infinite dimensions.
}
\end{rema}

Define a $\GG$-invariant $2$-tensor $P$ on $M$ to be
$\GG$-{\em quasi non-degenerate\/}
when the morphism 
\begin{equation}
P^\sharp+ \fund_M\colon
\TT ^* M \oplus_M M \times \gg \longrightarrow \TT M
\label{surj21}
\end{equation}
of vector bundles on $M$ is an epimorphism or, equivalently, when
the morphism
\begin{equation}
(P^\sharp, \fund_M^*)\colon
\TT^* M \longrightarrow 
\TT  M \oplus_M M \times \gg^* 
\label{inj22}
\end{equation}
of vector bundles on $M$ is a monomorphism.

\begin{prop}
\label{compar32}
Let  $\sigma$ be a $\GG$-invariant $2$-form on $M$ and $P$ a 
$\GG$-invariant skew-symmetric 
$2$-tensor on $M$. Suppose that $\sigma$ and $P$ are 
dual to each other via the $\GG$-equivariant 
 admissible map $\Phi \colon M \to \GG$.
\begin{enumerate}
\item
The morphisms
{\rm \eqref{surj11}} and {\rm \eqref{surj21}}
of vector bundles on $M$ are epimorphisms, that is, 
 the $2$-form $\sigma$ on $M$ is $\Phi$-non-degenerate,
and the $2$-tensor $P$ is $\GG$-quasi non-degenerate.
\item
For a $\GG$-invariant admissible function $f \colon M \to \bK$, the 
vector field $X_f = P^\sharp(df)$ is $\GG$-invariant and
satisfies the identities
\begin{align}
(\sigma^\fflat \circ P^\sharp)(df) &= df
\label{idqp1}
\\
 (P^\sharp \circ \sigma^\fflat) (X_f)&=X_f.
\label{idqp2}
\end{align}

\item
For two $\GG$-invariant admissible functions $f,h \colon M \to \bK$,
\begin{equation}
\sigma (X_f, X_h)  =P(dh,df).
\label{idqp3}
\end{equation}
\end{enumerate}
\end{prop}

\begin{proof}
Since \eqref{surj3} is an epimorphism of vector bundles, so is \eqref{surj21} 
and, since \eqref{surj4} is an epimorphism of vector bundles, so is, likewise,
\eqref{surj11}.

Let $f$ be a $\GG$-invariant $\bK$-valued admissible function on $M$. Since 
$f$ is $\GG$-invariant,  $\fund_M^*(df)$ is zero. Hence  \eqref{31d} implies 
\eqref{idqp1} and \eqref{idqp2}. Moreover, for two $\GG$-invariant admissible 
functions $f,h \colon M \to \bK$, \eqref{31d} implies \eqref{idqp3}.
\end{proof}

\begin{rema}
{\rm
For a $\GG$-invariant admissible $\bK$-valued function
on $M$, identity \eqref{idqp1}, rewritten as 
$\sigma(X_f,\,\cdot\,) = df$,
recovers the classical definition of an ordinary
Hamiltonian vector field.
In classical mechanics,
an expression of the kind 
$\sigma (X_f, X_h)$ characterizes
the Poisson bracket of two functions $f$ and $h$, 
cf., e.g., \cite{MR515141}.
In the present paper,
for consistency with 
\cite{MR1880957},
we define the Poisson bracket of two functions $f$ and $h$
by an expression of the kind $P(df,dh)$, however,
cf. \eqref{pb2}.
Also, with this definition,
in the classical case,
the canonical map from 
a Poisson  algebra of smooth functions
on a smooth manifold
to the smooth vector fields thereof
which sends a function to its Hamiltonian vector field
is compatible with the Lie brackets, that is,
there is no sign coming in.
}
\end{rema}

Let $\sigma$ be a $\GG$-invariant $2$-form on $M$ and
$P$ a $\GG$-invariant skew-symmetric $2$-tensor on $M$, and let
$\Phi \colon M \to \GG$ be a $\GG$-equivariant admissible map.
When $\Phi$ is a $\GG$-momentum mapping for
$P$ relative to some $\Ad$-invariant symmetric $2$-tensor
in $\gg \otimes \gg$ and for $\sigma$ 
 relative to some $\Ad$-invariant symmetric $2$-form
on $\gg$,
we say $P$ and $\sigma$ are
 {\em $\Phi$-momentum dual to each other
via\/} $\Phi$.

\begin{prop}
\label{unique}
Suppose that $P$ and $\sigma$ are
  $\Phi$-momentum dual to each other
  via  an admissible  $\GG$-equivariant map $\Phi\colon M \to \GG$
relative to some
$\Ad$-invariant symmetric $2$-tensor $\oomega$
in $\gg \otimes \gg$ and 
 to some $\Ad$-invariant symmetric $2$-form
$\form$ on $\gg$.
Then the diagrams
\begin{equation}
\begin{gathered}
\xymatrixcolsep{4pc}
\xymatrix{
&\TT^*_\Phi \GG \oplus_M \TT M 
\ar@{}[d]|-{+}
\ar[dr]|-{\Id - \tfrac {\rho_\Phi}4}
\ar[r]^{\phantom{aaa}(d\Phi)^*_M + \sigma^\fflat}
\ar[dl]|-{\tfrac 12 (L_\Phi^* + R_\Phi^*)}
& \TT^* M \ar@{.>}[d]^{P^\sharp}
\\
M \times \gg^*
\ar[r]_{\Id \times \psi^\oomega} 
&M \times \gg \ar[r]_{\fund_M}& \TT M
}
\end{gathered}
\label{render1}
\end{equation}
and
\begin{equation}
\begin{gathered}
\xymatrixcolsep{4pc}
\xymatrix{
&M \times \gg \oplus_M \TT^* M \ar[dr]|-{\Id - \tfrac {\rho^*_\Phi}4}
\ar[r]^{\phantom{aaaaa}\fund_M + P^\sharp}
\ar[dl]|-{\Id \times \psidot}
\ar@{}[d]|-{+}
& \TT M \ar@{.>}[d]^{\sigma^\fflat}
\\
M \times \gg^* \ar[r]_{\tfrac 12 (L_\Phi^{*,-1} +R_\Phi^{*,-1})}& \TT_\Phi^* \GG \ar[r]_{(d\Phi)_M^*}& \TT^* M
}
\end{gathered}
\label{render2}
\end{equation}
are commutative. Hence $\sigma$ determines $P$ uniquely,
and $P$ determines $\sigma$ uniquely.
\end{prop}

\begin{proof}
It is immmediate that the two diagrams are commutative.
By Proposition \ref{compar32}, the top rows of the two diagrams
are epimorphisms of vector bundles on $M$.
Hence $\sigma$ and $P$ determine each other uniquely.
\end{proof}

Under the circumstances of Proposition \ref{unique} we say that
$\sigma$ is the $\Phi$-{\em momentum dual of\/}
$P$ and that
$P$ 
is the $\Phi$-{\em momentum dual of\/} $\sigma$.

As before, let 
$\,\form\,$ designate 
an $\Ad$-invariant symmetric bilinear form
on the Lie algebra $\gg$ of $\GG$ and
$\oomega$
an $\Ad$-invariant symmetric $2$-tensor
in $\gg \otimes \gg$.

\begin{prop}
\label{6.3}
Consider a $(\GG \times \GG)$-manifold $M$ together with 
  a $(\GG\times \GG)$-quasi Poisson structure $P^\times$ 
 relative to $\oomega^\times$
and  a weakly 
 $(\GG\times \GG)$-quasi closed $2$-form $\sigma^\times$
relative to $\,\form\,$,
and let
$(\Phi^1,\Phi^2)\colon M \to \GG \times \GG$ be a 
$(\GG \times \GG)$-momentum mapping
for $P^\times$ relative to $\oomega^\times$
as well as one for $\sigma^\times$ relative to $\,\form\,$.
Suppose that
$\sigma^\times$ and $P^\times$
are dual to each other via
$(\Phi^1,\Phi^2)$.
Then, with respect to the diagonal $\GG$-action on $M$,
the weakly  $\Phi$-quasi closed $2$-form
\begin{equation*}
\sigma=\sigma_\fus = \sigma^\times - (\Phi^1,\Phi^2)^*(\tfrac 12 \omega_1 \form \ovomega_2)
\end{equation*}
relative to $\,\form\,$,
cf. Propopsition {\rm \ref{fusmoms}}, and
the  $\GG$-quasi Poisson structure
$P = P_\fus$ on $M$ relative to $\oomega$,
cf. {\rm \eqref{Pfus}},
are dual to each other
via the 
 $\GG$-momentum mapping
\begin{equation*}
\Phi =\Phi_\fus= \Phi^1 \Phi^2 \colon M \to \GG
\end{equation*}
relative to $\oomega$ and $\,\form\,$
resulting from fusion,
cf. 
Proposition {\rm \ref{momfusGs}}
and Theorem {\rm \ref{momfusG}}.
\end{prop}

\begin{proof}
Left to the reader.
\end{proof}

As before, consider a $\GG$-invariant  skew-symmetric  
$2$-tensor $P$ on $M$.
Suppose the symmetric $2$-tensor $\oomega$ in $\gg \otimes \gg$
is non-degenerate, and let
$\Phi \colon M \to \GG$ be a $\GG$-momentum mapping for
$P$ relative to $\oomega$.
Let $\psi_\Phi^\oomega \colon \TT _\Phi^* \GG \to  \TT _\Phi \GG$
denote the isomorphism of vector bundles on $M$
which $\oomega$ induces.
Recall that $\TT^* M^{P}$ denotes the kernel of the adjoint
 $P^\sharp\colon \TT^* M \to \TT M$ of $P$, a morphism of vector 
bundles on $M$.
From the commutative diagrams  \eqref{CD01} and \eqref{CD01d}, 
we concoct the commutative diagram
\begin{equation}
\begin{gathered}
\xymatrix{
&\ker(L_\Phi^{*}+ R_\Phi^{*})
\ar[r]
\ar@{>->}[d]
&\TT^* M^{P}\ar@{>->}[d]\ar[rr] 
&
&\ker(L_\Phi^{*,-1}+ R_\Phi^{*,-1})\ar@{>->}[d] 
\\
M \times \gg^*
\ar[d]|-{-\Id \times \psi^\oomega}
&\ar[l]_{L_\Phi^* + R_\Phi^*}\TT^* _\Phi\GG 
\ar[r]^{(d \Phi)^*_M} 
&
\TT^* M \ar[d]^{- P^\sharp}
 \ar[rr]^{\fund_M^*} 
&
&
 M \times \gg^* 
\ar[d]|-{\tfrac 12 (L_\Phi^{*,-1}+ R_\Phi^{*,-1})} 
\\
M \times \gg
\ar[rr]_{\fund_M}
&&\TT M \ar[r]_{(d \Phi)_M}
&
\TT _\Phi \GG
\ar[r]_{\psi_\Phi^{\oomega,-1}}
&
\TT^* _\Phi \GG ,
}
\end{gathered}
\label{PMukKK}
\end{equation}
the lower right-hand rectangle being a variant of \eqref{CD01d}.
The following is exactly analogous to
Proposition \ref{tech}.
 
\begin{prop}
\label{techP}
The symmetric $2$-tensor $\oomega$ in $\gg \otimes \gg$ being non-degenerate,
the restriction of
$
2 L_\Phi^* \colon \TT _\Phi^* \GG \longrightarrow
M \times \gg^* 
$
to $\ker(L_\Phi^{*}+ R_\Phi^{*})$
yields the
upper row of {\em \eqref{PMukKK}}.
Hence
\begin{equation*}
(d \Phi)^*_M|\colon \ker(L_\Phi^{*}+ R_\Phi^{*})\longrightarrow \TT^* M^P
\end{equation*}
is a monomorphism of distributions on $M$
and
\begin{equation*}
\TT^* M^P \longrightarrow \fund_M^*|\colon \ker(L_\Phi^{*,-1}+ R_\Phi^{*,-1})
\end{equation*}
is an epimorphism of distributions on $M$. \qed
\end{prop}

Here is the exact analogue of Proposition \ref{nondegqh}.
\begin{prop}
\label{nondegqhP}
The symmetric $2$-tensor $\oomega$ in $\gg \otimes \gg$ being non-degenerate,
the following are equivalent.
\begin{enumerate}
\item
The morphism 
$(d \Phi)^*_M|\colon \ker(L_\Phi^{*}+ R_\Phi^{*})\to \TT^* M^P$
is an epimorphism of distributions on $M$.
\item
The morphism 
$(d \Phi)^*_M|\colon \ker(L_\Phi^{*}+ R_\Phi^{*})\to \TT^* M^P$
is an isomorphism of distributions on $M$
\item
The intersection
$
(\TT ^* M)^P \cap \ker(\fund_M^*)
$
is trivial.
\item
The $2$-tensor $P$ is $\GG$-quasi non-degenerate.
\item
The morphism 
$\fund_M^*|\colon \TT^* M^P \to\ker(L_\Phi^{*,-1}+ R_\Phi^{*,-1})$
is a monomorphism of distributions on $M$.
\item
The morphism 
$\fund_M^*|\colon \TT^* M^P \to\ker(L_\Phi^{*,-1}+ R_\Phi^{*,-1})$
is an isomorphism of distributions on $M$. \qed
\end{enumerate}
\end{prop}
 
\begin{thm}
\label{existence}
Let $M$ be a $\GG$-manifold and
$\Phi \colon M \to \GG$ a $\GG$-equivariant admissible map.
Suppose
the symmetric $\Ad$-invariant $2$-tensor 
 $\oomega$ in $\gg \otimes \gg$ arises from 
a non-degenerate  $\Ad$-invariant symmetric bilinear form $\,\form\,$ on $\gg$
as the image of
the corresponding $2$-tensor in $\gg^* \otimes \gg^*$
under the inverse of the  adjoint $\psidot \colon \gg \to \gg^*$
of $\,\form\,$.
Then the commutative diagrams
{\rm \eqref{render1}} and {\rm \eqref{render2}}
of morphisms of vector bundles on $M$
establish a bijective correspondence
between $\Phi$-non-degenerate
 $\GG$-invariant $2$-forms 
$\sigma$ on $M$ 
having $\Phi$ as $\GG$-momentum mapping relative to
$\,\form\,$
and $\GG$-quasi non-degenerate  
$\GG$-invariant skew-symmetric $2$-tensors $P$ on $M$ having $\Phi$ as 
$\GG$-momentum mapping relative to $\oomega$.
That is to say: 
\begin{enumerate}
\item
Given the $2$-form $\sigma$ on $M$,
the upper row of \eqref{render1}
being an epimorphism of vector bundles on $M$,
the identity
\begin{equation}
P^\sharp \circ \left((d \Phi)_M^* + \sigma^\fflat \right) =
\tfrac 12\fund_M \circ(\Id \times\psi^\oomega)\circ(L_\Phi^* + R_\Phi^*)
+(\Id - \tfrac 14 \rho_\Phi)
\end{equation}
characterizes the adjoint $P^\sharp$ of the skew-symmetric $2$-tensor $P$ on $M$.
\item
Given the skew-symmetric $2$-tensor $P$ on $M$, the 
upper row of \eqref{render2} being an epimorphism of vector bundles on $M$,
the identity 
\begin{equation}
\sigma^\fflat \circ(\fund_M + P^\sharp) =
\tfrac 12 (d \Phi)^*_M \circ (L_\Phi^{*,-1} + R_\Phi^{*,-1})
\circ(\Id \times \psi^{\form}) + (\Id - \tfrac 14 \rho_\Phi^*)
\end{equation}
characterizes the adjoint $\sigma^\fflat$ of the $2$-form  $\sigma$ on $M$.
\end{enumerate}
Under this correspondence, the $2$-form $\sigma$ is $\Phi$-quasi closed
and hence $(\sigma,\Phi)$
a $\GG$-quasi Hamiltonian structure relative to $\,\form\,$ on $M$ 
if and only if
the $2$-tensor $P$ is $\GG$-quasi Poisson and hence
$(P,\Phi)$ a Hamiltonian $\GG$-quasi Poisson structure relative to $\oomega$
on $M$.
\end{thm}

\begin{cor}
\label{momdalg}
Under the circumstances of Theorem {\rm \ref{existence}},
suppose $\GG$ is an algebraic group,
$M$ an affine algebraic variety,
and $\Phi$ a morphism of affine varieties.
When $(P,\Phi)$ is a $\GG$-quasi Poisson structure
and $(\sigma,\Phi)$ a weakly $\GG$-quasi Hamiltonian structure
(and hence, by Proposition {\rm \ref{compar32} (1)}, 
a $\GG$-quasi Hamiltonian structure, i.e.,
$\sigma$ is $\Phi$-non-degenerate),
the $2$-form
$\sigma$ 
is algebraic if and only if its 
$\Phi$-momentum dual $P$ is algebraic. \qed
\end{cor}

Our proof of Theorem \ref{existence} in Subsection \ref{diracst} below
relies on explicit descriptions of bijective  correspondences
between quasi structures and twisted Dirac structures
in the literature.
We cannot simply quote the theorems, however,
since only the proofs render the correspondences explicit.
Theorems \ref{dvsq1} and \ref{dvsq2}
offer explicit descriptions
of these correspondences.

\begin{rema}
\label{lurking}
{\rm
For the special case where
$\GG$ is compact and $\,\form\,$ positive definite,
the claim of Theorem \ref{existence}
is lurking behind
{\rm \cite[Theorem 10.3 p.~24]{MR1880957}} and the proof thereof.
}
\end{rema}

\subsection{Dirac structures and the proof of Theorem \ref{existence}}
\label{diracst}

For intelligibility,
we recollect a bare minimum.
The reader can find more details in
\cite{MR2642360, MR2103001, MR2068969,
MR998124, MR2023853}.

\subsubsection{Linear Dirac structures}

Let $V$ and $W$ be a $\bK$-vector spaces.
Write  a linear endomorphism $\chi$ of $V \oplus W$ as
\begin{equation}
\chi = \left[ 
\begin{matrix}
\chi_{1,1} & \chi_{1,2}\\ \chi_{2,1} & \chi_{2,2} 
\end{matrix}
\right]\colon V \oplus W
\to V \oplus W,
\ 
\chi\left[ 
\begin{matrix}
X\\ \alpha
\end{matrix}
\right]
=
\left[ 
\begin{matrix}
\chi_{1,1}(X) + \chi_{1,2}(\alpha) \\ \chi_{2,1}(X) +\chi_{2,2}(\alpha) 
\end{matrix}
\right],\ X \in V, \alpha \in W.
\label{write}
\end{equation}

The direct sum $V \oplus V^*$
carries the standard non-degenerate split bilinear form
\begin{equation}
\hinn \colon (V \oplus V^*) \otimes (V \oplus V^*) \to \bK,
\ 
\langle (v,\alpha),(w,\beta) \rangle = \alpha(w) + \beta(v),
\ 
v,w \in V,\ \alpha, \beta \in V^*.
\label{hinnn}
\end{equation}
For an endomorphism $p \colon V \oplus V^* \to V \oplus V^*$, let
 $p^t \colon V \oplus V^* \to V \oplus V^*$ 
denote the adjoint of $p$ relative to $\hinn$.
A {\em linear Dirac structure\/} on 
$V$  is a linear subspace $E$ of $ V \oplus V^*$ 
that is Lagrangian relative to $\hinn$.
A {\em Dirac space\/} is a vector space 
 together with a linear Dirac structure.

\subsubsection{Lagrangian splittings}
There is a bijective correspondence between 
projection operators
\begin{align}
\pp&
 \colon V \oplus V^* 
\twoheadrightarrow E
\rightarrowtail
V \oplus V^*
\label{ppphiV}
\\
\pp^t& 
 \colon V \oplus V^* 
\twoheadrightarrow F
\rightarrowtail
V \oplus V^*
\label{qqpphiV}
\end{align}
enjoying the property
$\pp + \pp^t =\Id$ and Lagrangian splittings 
$V \oplus V^*
=
E \oplus F
$
of $V \oplus V^*$
relative to 
$\hinn$.

Consider a projection operator
$
\pp= 
\left[ 
\begin{matrix}
\pp_{1,1} & \pp_{1,2}\\ \pp_{2,1} & \pp_{2,2} 
\end{matrix}
\right] \colon V \oplus V^* 
\to
V \oplus V^*
$
such that
$\pp + \pp^t =\Id$.
Then ${\pp_{1,2}\colon V^* \to V}$ is the adjoint of 
a skew-symmetric $2$-tensor $P \in \LAL^{\mrc,2}[V]$
and
 ${\pp_{2,1}\colon V \to V^*}$ is the adjoint of 
an alternating $2$-form $\sigma \colon V \otimes V \to \bK$.
It is common to say that the
 projection 
$\pp$ 
{\em defines\/} the
 skew-symmetric $2$-tensor $P$
and the alternating $2$-form $\sigma$.

The graph
\begin{equation*} 
\Gr_\sigma
= \{(v,\sigma^\fflat(v)); v \in V \} 
\subseteq  V \oplus V^* 
\end{equation*}
of an alternating $2$-form $\sigma$
on $V$ is a Lagrangian subspace of $V \oplus V^*$,
and the projection 
\begin{equation}
{
\pp_\sigma= 
\left[ 
\begin{matrix}
\Id & 0\\ \sigma^\fflat & 0 
\end{matrix}
\right] \colon V \oplus V^* 
\twoheadrightarrow \Gr_\sigma \rightarrowtail
V \oplus V^*
}
\label{projsigma}
\end{equation}
to $\Gr_\sigma$ 
determines the Lagrangian splitting  $\Gr_\sigma \oplus V^*$ of $V \oplus V^*$
and defines $\sigma$.
The graph
\begin{equation*} 
\Gr_P
= \{(P^\sharp(\alpha),\alpha); \alpha \in V^* \} 
\subseteq  V \oplus V^* 
\end{equation*}
of a skew-symmetric $2$-tensor $P$
in $V \otimes V$ is, likewise, a Lagrangian subspace of $V \oplus V^*$,
and the projection 
\begin{equation}{
\pp_P= 
\left[ 
\begin{matrix}
0 & P^\sharp \\ 0 & \Id 
\end{matrix}
\right] \colon V \oplus V^* 
\twoheadrightarrow \Gr_P \rightarrowtail
V \oplus V^*
}
\label{projP}
\end{equation}
to $\Gr_P$
determines the Lagrangian splitting  $V \oplus \Gr_P$ of $V \oplus V^*$
and defines $P$.
For later reference, we spell out the following,
whose proof is immediate.

\begin{prop}
\label{latref1}
The graph $\Gr_\sigma \subseteq V \oplus V^*$ of an alternating 
$2$-form $\sigma$ on $V$ is transverse to $V\subseteq V \oplus V^*$
if and only if $\sigma$ is non-degenerate, 
the projection 
\begin{equation}
{
\pp= 
\left[ 
\begin{matrix}
0 & \sigma^{\fflat,-1} \\ 0 & \Id 
\end{matrix}
\right] \colon V \oplus V^* 
\twoheadrightarrow \Gr_\sigma \rightarrowtail
V \oplus V^* 
}
\label{projPsigma}
\end{equation}
relative to the
Lagrangian splitting 
$\Gr_\sigma \oplus V$ of $V \oplus V^*$
then defines the 
skew-symmetric  $2$-tensor $P$ in 
$V\otimes V$ dual to $\sigma$ via the adjoints
as $P^\sharp = \sigma^{\fflat,-1}$, and
$\Gr_\sigma$ coincides with $\Gr_P$.

The graph $\Gr_P \subseteq V \oplus V^*$ of a skew-symmetric 
$2$-tensor $P$ in $V$ is, likewise, 
transverse to $V^*\subseteq V \oplus V^*$
if and only if $P$ is non-degenerate, 
the projection 
\begin{equation}
{
\pp= 
\left[ 
\begin{matrix}
\Id & 0 \\ P^{\sharp,-1}  & 0 
\end{matrix}
\right] \colon V \oplus V^* 
\twoheadrightarrow \Gr_P \rightarrowtail
V \oplus V^* 
}
\label{projPP}
\end{equation}
relative to the
Lagrangian splitting 
$\Gr_P \oplus V^*$ of $V \oplus V^*$
then defines  the
alternating   $2$-form $\sigma$ on 
$V$ dual to $P$
via the adjoints as  $ \sigma^{\fflat}=P^{\sharp,-1} $, and
$\Gr_P$ coincides with $\Gr_\sigma$. \qed
\end{prop}

\subsubsection{Forward and backward images}
\label{fbi}

Consider two vector spaces $V_1$ and $V_2$, and let
$\sigma$ be an alternating  $2$-form on $V_1$
and $\Phi \colon V_1 \to V_2$  a linear map.
For a Lagrangian subspace  $E\subseteq V_1 \oplus V_1^*$, define
the $(\Phi,\sigma)$-{\em forward image\/} 
$E_{\Phi,\sigma}\subseteq V_2 \oplus V_2^*$
of $E$, necessarily a Lagrangian subspace,
by
\begin{equation}
E_{\Phi,\sigma} = \{ (\Phi(v),\alpha);  
(v,\alpha \circ \Phi -\sigma^\flat(v))  \in E\} \subseteq V_2 \oplus V_2^*.
\end{equation}
When $\sigma$ is zero, we write
$E_\Phi$ and refer to it as the $\Phi$-{\em forward image\/}
of $E$.
For a vector space $V$
and an alternating  $2$-form $\sigma$ on $V$,
we write the
 $(\Id,\sigma)$-forward image 
of a Lagrangian subspace $E$ of $V \oplus V^*$ as $E_\sigma$ and
refer to it as the $\sigma$-{\em forward image\/}
of $E$. In particular, the $\sigma$-forward image
 \begin{equation}
V_\sigma =
\{ (v,\alpha); 
(v,\alpha -\sigma^\flat(v))  \in V
\}
\subseteq V \oplus V^*
\end{equation}
of $V$  coincides with 
the graph
$\Gr_\sigma \subseteq V \oplus V^*$
of $\sigma$ in $V \oplus V^*$.
The $\Phi$-{\em forward image\/} 
$\Gr_{\sigma;\Phi}\subseteq V_2 \oplus V_2^*$
of the graph $\Gr_\sigma \subseteq V_1 \oplus V_1^*$ 
of an alternating $2$-form $\sigma$ on $V_1$
reads
\begin{equation}
\Gr_{\sigma;\Phi} 
=
\{ (\Phi(v),\alpha); (v, \alpha \circ \Phi) \in E_\sigma
\}
= \{ (\Phi(v),\alpha);  \alpha \circ \Phi= 
\sigma(v,\,\cdot\,)\}
\subseteq V_2 \oplus V_2^*
\end{equation}
and coincides with the 
 $(\Phi,\sigma)$-{\em forward image\/} 
\begin{equation}
V_{1,\Phi,\sigma} = \{ (\Phi(v),\alpha);  
(v,\alpha \circ \Phi -\sigma^\flat(v))
 \in V_1
\} \subseteq V_2 \oplus V_2^*
\end{equation}
of $V_1$.

For a Lagrangian subspace $E \subseteq V_2\oplus V_2^*$
and an alternating $2$-form $\sigma$ on $V_1$,
define
 the $(\Phi,\sigma)$-{\em backward image\/}
$E^{\Phi,\sigma}\subseteq V_1 \oplus V_1^*$
of $E$, necessarily a Lagrangian subspace, 
by
\begin{equation}
E^{\Phi,\sigma} = \{ (v, \alpha);\   
\Phi^*(\beta) = \alpha + \sigma^\flat(v)
\text{\ for\ some\ } \beta \in V_2^* \text{\ such\ that\ }
(\Phi(v),\beta) \in E \}
\end{equation}
or, equivalently,
\begin{equation}
E^{\Phi,\sigma} = \{ (v, \Phi^*(\beta)-\sigma^\flat(v)); \ 
(\Phi(v),\beta) \in E \} .
\end{equation}
When $\sigma$ is zero, we write $E^\Phi$ rather than $E^{\Phi,\sigma}$
and refer to it as the 
 $\Phi$-{\em backward image\/}
of $E$.

\subsubsection{Morphisms}
Consider two Dirac spaces  $(V_1,E_1)$ and $(V_2,E_2)$. A linear map 
$\Phi \colon V_1 \to V_2$ is a {\em Dirac morphism\/}  
 $\Phi \colon (V_1,E_1) \to (V_2,E_2)$
when $E_2 = E_{1;\Phi}$, that is
\begin{equation*}
E_2 = \{ (\Phi(v),\alpha);  
(v,\alpha \circ \Phi)  \in E_1\} \subseteq V_2 \oplus V_2^*.
\end{equation*}
More generally, a {\em Dirac morphism\/}  
 $(\Phi,\sigma) \colon (V_1,E_1) \to (V_2,E_2)$
consists of a linear map $\Phi \colon V_1 \to V_2$
and an alternating  $2$-form $\sigma$ on $V_1$
such that $E_2 = E_{1;\Phi,\sigma}$, that is,
\begin{equation*}
 E_2= \{ (\Phi(v),\alpha);  
(v,\alpha \circ \Phi -\sigma^\flat(v))  \in E_1\} \subseteq V_2 \oplus V_2^*.
\end{equation*}
Relative to a  Dirac morphism  
 $\Phi \colon (V_1,E_1) \to (V_2,E_2)$,
recall the notation 
\begin{align}
\ker(\Phi,\sigma)&= \{(v,-\sigma^\flat(v));\  \Phi(v) = 0\} \subseteq 
V_1 \oplus V_1^* .
\end{align}
A   Dirac morphism  
 $(\Phi,\sigma) \colon (V_1,E_1) \to (V_2,E_2)$
 is {\em strong\/}
when
\begin{equation}
\ker(\Phi,\sigma) \cap E_1 =\{0\}, 
\end{equation}
that is, $(v,-\sigma^\flat(v)) \in E_1$ and $\Phi(v)=0$
implies $v=0$.

\begin{prop}
\label{cruc}
Let  $(\Phi,\sigma)\colon (V_1, E_1) \to (V_2, E_2)$
 be a Dirac morphism
and $F$ a Lagrangian subspace of $V_2\oplus V_2^*$
that is transverse to $E_2$. 
The following are equivalent:
\begin{enumerate}
\item
The Dirac morphism $(\Phi,\sigma)$ is strong. 
\item
The Dirac morphism
$\Phi\colon (V_1, E_{1;\sigma}) \to (V_2, E_2)$ is strong. 
\item
The $(\Phi,\sigma)$-backward image
$F^{\Phi,\sigma}$ of $F$ is transverse to $E_1$, and hence
$E_1 \oplus F^{\Phi,\sigma}$ is a Lagrangian splitting of
$V_1 \oplus V_1^*$.
\item
The $\Phi$-backward image
$F^\Phi$ of $F$ is transverse to 
the $\sigma$-forward image
$E_{1;\sigma}$ of $E_1$, and hence
$E_{1,\sigma} \oplus F^\Phi$ is a Lagrangian splitting of
$V_1 \oplus V_1^*$.
\end{enumerate}
\end{prop}

\begin{proof}
For $v \in V_1$,
the vector $(v, -\sigma^\flat(v))$ belongs to $E_1$ if and only if
$(v,0)$ belongs to $E_{1;\sigma}$.
This implies that (1) and (2) are equivalent.
For a proof that
(1) and (3) are equivalent,
see, e.g.,  \cite[Prop.1.15]{MR2642360}.
Accordingly, (2) and (4) are equivalent.
\end{proof}

\begin{rema}
{\rm
Consider the Dirac morphism $(0,\sigma) \colon (V,E) \to (0,0)$.
The backward image $F^{0,\sigma}$ of the zero Lagrangian subspace
$F$ of $0$ coincides with  
\begin{equation*}
\ker(0,\sigma)=
F^{0,\sigma}=\{(v,-\sigma^\flat(v)); v \in V\}\subseteq V \oplus V^*.
\end{equation*}
The Dirac morphism $(0,\sigma)$
is strong if and only if 
$\ker(0,\sigma)$ is transverse to $V$, that is,
$\sigma$ non-degenerate.
}
\end{rema}

\subsubsection{Dirac structures over manifolds}
\label{dsom}
Let $M$ be a manifold and $\zeta \colon T \to M$ a
vector bundle.
The non-degenerate split bilinear form 
\eqref{hinnn} extends in an obvious manner to a
non-degenerate split bilinear form 
\begin{equation}
\hinn
 \colon (T \oplus_M T^*) \otimes_M(T \oplus_M T^*) \longrightarrow \bK.
\label{hinnTn}
\end{equation}
An {\em almost Dirac structure\/} on $\zeta\colon  T \to M$ 
is a subbundle $E\to M $ of $T \oplus T^*\to M$ that is
 maximally isotropic relative to $\hinn$.
An {\em almost Dirac manifold\/}
$(M,E)$
consists of a manifold $M$ and an almost Dirac structure $E \to M$
on $M$.

Let $\lambda$ be a closed $3$-form on $M$.
The notation in  \cite{MR2642360} is $\eta$
for the present $\lambda$. 
The $\lambda$-{\em twisted Courant algebroid\/} over $M$
is the vector bundle
$\tau_M \oplus_M \tau^*_M\colon \TT M \oplus_M \TT^* M \to M$
together with the non-degenerate split  bilinear form
\eqref{hinnTn} (for $T = \TT M$),
anchor $\rho$ coming from the projection to $\TT M$, and
$\lambda$-{\em twisted Courant  bracket\/} 
$\bra_\lambda$ on the space $\Gamma(\tau_M \oplus_M \tau^*_M)$ 
of sections of $\tau_M \oplus_M \tau^*_M$
which the identity
\begin{equation}
[(X_1,\alpha_1), (X_2, \alpha_2)]_\lambda =
\left([X_1,X_2], \mathcal L_{X_1} \alpha_2 
- i_{X_2}d \alpha_1
-i_{X_1 \wedge X_2}\lambda\right)
\label{CB1t}
\end{equation} 
characterizes. Here the sign convention for $\lambda$ is that
in \cite{MR2642360}, opposite to the sign in
\cite{MR2023853},
 see footnote 3 in \cite{MR2642360}.
A $\lambda$-{\em twisted Dirac structure\/}
on $M$ is an almost Dirac structure
 ${\eta \colon E \to M}$ on $\tau_M \colon \TT M \to M$
such that the space $\Gamma(\eta)$ of sections of $\eta$
is closed under  $\bra_\lambda$
\cite[(3)~p.~147]{MR2023853},
\cite{MR2068969},
\cite[(3.2)~p.~10]{MR2103001}.
When $\lambda$ is zero,
\eqref{CB1t} comes down to the ordinary
 {\em Courant
bracket\/} $\bra$,
and
${(\tau_M \oplus_M \tau^*_M, \bra, \rho)}$ is the ordinary
{\em Courant algebroid\/} over $M$
\cite{MR998124, MR2023853};
further, 
a $0$-twisted Dirac structure
on $M$ is an ordinary Dirac structure.
A {\em twisted Dirac manifold\/}
$(M,E, \lambda)$
consists of a manifold $M$, an almost Dirac structure $E \to M$
on $M$, and a closed $3$-form on $M$ that turns
$(M,E, \lambda)$ into a $\lambda$-twisted Dirac structure;
a  twisted Dirac manifold of the kind
$(M,E, 0)$ is an ordinary  Dirac manifold.

Let $\sigma$ be a $2$-form on $M$.
The graph $\Gr_\sigma \subseteq \TT M \oplus \TT^* M$
of the adjoint $\sigma^\fflat \colon \TT M \to \TT^* M$
of $\sigma$ is an almost Dirac structure on $M$;
it is also common to say
$\sigma$ {\em defines\/} the almost Dirac structure
$\Gr_\sigma$. 
Such an  almost Dirac structure $\Gr_\sigma$ on $M$
is
a $\lambda$-twisted Dirac structure 
 if and only if $d \sigma=\lambda$
\cite[Section 2]{MR2642360}
(with a minus sign in
\cite{MR2023853}).
In the same vein,
consider an skew-symmetric bivector  $P$ over $M$.
The graph $\Gr_P \subseteq \TT M \oplus \TT^* M$
of the adjoint $P^\sharp \colon \TT^* M \to \TT^* M$
of $P$ is an almost Dirac structure on $M$,
and $P$ is said to {\em define\/}
this almost Dirac structure. 
The adjoint $P^\sharp \colon \TT^* M \to \TT M$
of $P$
determines the alternating $3$-tensor
$(\LAL^{\mrc,3}P^\sharp) (\lambda)$ on $M$,
and the almost Dirac structure
 $\Gr_P$
is a $\lambda$-twisted Dirac structure on $M$
 if and only if 
$[P,P]+2 (\LAL^{\mrc,3}P^\sharp) (\lambda)=0$
\cite[Section 2]{MR2642360}; see also
\cite[Section 2  p.~554 ff]{MR2068969},
\cite{MR2023853}, with the sign convention for $\lambda$ there.

An {\em almost Dirac morphism\/} 
\begin{equation}
(\Phi, \sigma)\colon (M_1,E_1) \longrightarrow
(M_2,E_2)
\end{equation}
between two almost Dirac manifolds $(M_1,E_1)$ and
$(M_2,E_2)$ 
consists of an admissible map
$\Phi \colon M_1 \to M_2$ and a $2$-form $\sigma$
on $M_1$ such that
\begin{equation*}
(d\Phi_q,\sigma_q)\colon(\TT _q M_1, E_{1,q})\longrightarrow
(\TT _{\Phi(q)} M_2, E_{2,\Phi(q)})
\end{equation*}
is a  morphism of Dirac spaces for every $q \in M_1$;
the morphism  $(\Phi, \sigma)$ is {\em strong\/}
when
$(d\Phi_q, \sigma_q)$ is strong for every $q \in M_1$.
A {\em Dirac morphism\/}
\begin{equation}
(\Phi, \sigma)\colon (M_1,E_1, \lambda_1) \longrightarrow
(M_2,E_2, \lambda_2)
\end{equation}
between two twisted Dirac manifolds
$(M_1,E_1, \lambda_1)$ and
$(M_2,E_2, \lambda_2)$
 is
an  almost Dirac morphism $(\Phi, \sigma)\colon (M_1,E_1) \to
(M_2,E_2)$
such that
$\lambda_1 + d \sigma = \Phi^* \lambda_2$, and a Dirac morphism 
is {\em strong\/} when it is strong as an almost Dirac morphism.

For a $\GG$-manifold $M$ 
and two vector bundles $E \to M$ and $F \to M$, extending the notation
\eqref{write},
write  a vector bundle endomorphism $\chi$ of $E \oplus_M F$ 
over $M$ as
\begin{equation}
\chi =\left[ 
\begin{matrix}
\chi_{1,1} & \chi_{1,2}\\ \chi_{2,1} & \chi_{2,2} 
\end{matrix}
\right] \colon E \oplus_M F \to E \oplus_M F,
\end{equation}
for vector bundle endomorphisms
\begin{equation*}
\chi_{1,1} \colon E \longrightarrow E,\ 
 \chi_{1,2}\colon F \longrightarrow E,\ 
 \chi_{2,1} \colon E \longrightarrow F,\ 
 \chi_{2,2} \colon F \longrightarrow F 
\end{equation*}
over $M$.

\subsubsection{Group case}
\label{grcase}
View the group $\GG$ as a $\GG$-manifold via conjugation.
A straightforward verification establishes the following.

\begin{prop}
\label{projsEF}
The vector bundle endomorphism
\begin{equation}
\begin{aligned}
\pp&= 
\left[ 
\begin{matrix}
\pp_{1,1} & \pp_{1,2}\\ \pp_{2,1} & \pp_{2,2} 
\end{matrix}
\right]= \tfrac 14
\left[ 
\begin{matrix}
(L-R)(L^{-1}- R^{-1}) &  (L-R)(L^{-1}+R^{-1})
\\ (L+R)(L^{-1}- R^{-1}) &  (L+R)(L^{-1}+R^{-1}) 
\end{matrix}
\right]
\end{aligned}
\label{pp}
\end{equation}
of  $\TT \GG \oplus_\GG \TT \GG$ in $\gg \otimes \gg$
satisfies the identity
\begin{equation*}
\pp \circ\left[ \begin{matrix} L-R\\ L+R\end{matrix}\right] 
=\left[ \begin{matrix} L-R\\ L+R\end{matrix}\right]
\end{equation*}
and is therefore idempotent,
and
the vector bundle endomorphism
\begin{equation}
\begin{aligned}
\qqp&= 
\left[ 
\begin{matrix}
\qqp_{1,1} & \qqp_{1,2}\\ \qqp_{2,1} & \qqp_{2,2} 
\end{matrix}
\right]= \tfrac 14
\left[ 
\begin{matrix}
(L+R)(L^{-1}+R^{-1}) & (L+R)(L^{-1}- R^{-1}) 
\\  
 (L-R)(L^{-1}+R^{-1})& (L-R)(L^{-1}- R^{-1})  
\end{matrix}
\right]
\end{aligned}
\label{qqp}
\end{equation}
of  $\TT \GG \oplus_\GG \TT \GG$ in $\gg \otimes \gg$
satisfies the identity
\begin{equation*}
\qqp \circ\left[ \begin{matrix} L+R\\ L-R\end{matrix}\right] 
=\left[ \begin{matrix} L+R\\ L-R\end{matrix}\right]
\end{equation*}
and is therefore idempotent.
Furthermore,
\begin{equation*}
\pp + \qqp = \Id\colon  \TT \GG \oplus_\GG \TT \GG 
\to\TT \GG \oplus_\GG \TT \GG.
\end{equation*}
\end{prop}

Let $\,\form\,$ be a non-degenerate  $\Ad$-invariant
symmetric bilinear form on $\gg$,
let $\psidot\colon \gg \to \gg^*$ denote its adjoint
as before,
 let
 $\psi_\GG^{\form}\colon \TT \GG \to \TT^* \GG$
denote the isomorphism of vector bundles on $\GG$
which $\,\form\,$ induces,
and let the $2$-tensor
$\oomega$ in $\gg \otimes \gg$ arise from 
 $\,\form\,$ 
as the image of
the corresponding $2$-tensor in $\gg^* \otimes \gg^*$
under the inverse of the  adjoint $\psidot \colon \gg \to \gg^*$.
For later reference we note that the diagrams
\begin{equation}
\begin{gathered}
\xymatrix{
\GG \times \gg \ar[d]_{\Id \times \psidot} \ar[r]^L & 
\TT \GG
\ar[d]^{\psi_\GG^{\form}}
&\GG \times \gg \ar[d]_{\Id \times \psidot} \ar[r]^R 
& \TT \GG\ar[d]^{\psi_\GG^{\form}}
\\
\GG \times \gg^* \ar[r]_{L^{*,-1}} & \TT^* \GG
&\GG \times \gg^* \ar[r]_{R^{*,-1}} & \TT^* \GG
}
\end{gathered}
\label{laterr}
\end{equation}
are commutative.

Let $\Delta^{\form} \colon \gg \to \gg \oplus \gg^*$ denote the composite
\begin{equation}
\Delta^{\form} \colon \gg 
\stackrel{\Delta}\longrightarrow  \gg \oplus \gg
\stackrel{\left(\Id, \tfrac 12 \psidot\right)}
\longrightarrow  \gg \oplus \gg^* .
\end{equation}
The image $E_\GG \subseteq \TT \GG \oplus_\GG \TT^* \GG$ of the injection
\begin{gather}
\xymatrixcolsep{8pc}
\xymatrix{
\mathrm e \colon \GG \times \gg \ar[r]^{\Id \times \Delta^{\form}}& 
\GG \times (\gg \oplus \gg^*)
\ar[r]^{\left (L-R, L^{*,-1} + R^{*,-1}\right) }& 
\TT \GG \oplus_\GG \TT^* \GG
}
\end{gather}
and the image  $F_\GG \subseteq \TT \GG \oplus_\GG \TT^* \GG$
of the injection
\begin{gather}
\xymatrixcolsep{8pc}
\xymatrix{
\mathrm f \colon \GG \times \gg \ar[r]^{\Id \times  \Delta^{\form}}& 
\GG \times (\gg \oplus \gg^*)
\ar[r]^{\left (L+R, L^{*,-1} - R^{*,-1}\right) }&  
\TT \GG \oplus_\GG \TT^* \GG
}
\end{gather}
yield vector bundles $E_\GG \to \GG$ and $F_\GG \to \GG$, respectively, on 
$\GG$.
The vector bundle
$E_\GG \to \GG$ on $\GG$ is referred to in the literature as the
{\em Cartan-Dirac structure\/} of $\GG$
(with respect to $\,\form\,$)
\cite[Section 7.2~p.~591]{MR2068969},
\cite[(3.7)~p.~12]{MR2103001}, \cite[Section 3]{MR2642360},
\cite[Ex. 5.2~p.~151]{MR2023853}.
Relative to the non-degenerate bilinear pairing
$\pairi$
 on
${\TT \GG \oplus_\GG \TT^* \GG}$,
cf. \eqref{hinnTn},
the constituents $E_\GG$ and $F_\GG$
are mutually orthogonal, and the splitting
$ E_\GG \oplus_G F_\GG$
of $\TT \GG \oplus_\GG \TT^* \GG$
is  Lagrangian.

Let $\Phi \colon M \to \GG$
be an admissible $\GG$-equivariant map and let
 $\psi_\Phi^{\form}\colon \TT_\Phi \GG \to \TT_\Phi^* \GG$
denote the isomorphism of vector bundles on $M$
which $\,\form\,$ induces.
The following is immediate.
\begin{prop}
\label{via1}
Via $\Phi$, the vector bundles $E_\GG \to \GG$ and
$F_\GG \to \GG$ induce vector subbundles
$E_\Phi \to M$ and
$F_\Phi \to M$
of
$\TT_\Phi \GG \oplus_M \TT_\Phi^* \GG \to M$
in such a way that,
relative to the non-degenerate bilinear pairing
$\pairi$
 on
${\TT_\Phi \GG \oplus_M\TT^*_\Phi \GG}$,
cf. {\rm \eqref{hinnTn}}, the
splitting
$E_\Phi \oplus_M F_\Phi$
of $\TT_\Phi \GG \oplus_M \TT_\Phi^* \GG$ is Lagrangian.
Furthermore, the
idempotent vector bundle endomorphisms \eqref{pp}
and \eqref{qqp} over $\GG$ induce the   retractions
\begin{gather}
\begin{aligned}
\pp_\Phi&= 
\left[ 
\begin{matrix}
\pp_{\Phi,1,1} & \pp_{\Phi,1,2}\\ \pp_{\Phi,2,1} & \pp_{\Phi,2,2} 
\end{matrix}
\right] \colon \TT_\Phi \GG \oplus_M \TT^*_\Phi \GG 
\twoheadrightarrow E_\Phi
\rightarrowtail
\TT_\Phi \GG \oplus_M \TT^*_\Phi \GG
\\
&= \tfrac 14
\left[ 
\begin{matrix}
(L_\Phi-R_\Phi)(L_\Phi^{-1}- R_\Phi^{-1}) 
&  
2(L_\Phi-R_\Phi)(L_\Phi^{-1}+R_\Phi^{-1})\psi^{\form,-1}_\Phi
\\ 
\tfrac 12 \psi^{\form}_\Phi(L_\Phi+R_\Phi)(L_\Phi^{-1}- R_\Phi^{-1}) 
&  
(L^{*,-1}_\Phi+R^{*,-1}_\Phi)(L_\Phi^*+R_\Phi^*) 
\end{matrix}
\right]
\end{aligned}
\label{ppphi}
\\
\begin{aligned}
\qqp_\Phi&= 
\left[ 
\begin{matrix}
\qqp_{\Phi,1,1} & \qqp_{\Phi,1,2}\\ \qqp_{\Phi,2,1} & \qqp_{\Phi,2,2} 
\end{matrix}
\right] \colon \TT_\Phi \GG \oplus_M \TT^*_\Phi \GG 
\twoheadrightarrow F_\Phi
\rightarrowtail
\TT_\Phi \GG \oplus_M \TT^*_\Phi \GG
\\
&
= \tfrac 14
\left[ 
\begin{matrix}
(L_\Phi+R_\Phi)(L_\Phi^{-1}+R_\Phi^{-1}) 
& 
2(L_\Phi+R_\Phi)(L_\Phi^{-1}- R_\Phi^{-1}) \psi^{\form,-1}_\Phi
\\  
\tfrac 12 \psi^{\form}_\Phi (L_\Phi-R_\Phi)(L_\Phi^{-1}+R_\Phi^{-1})
& (L_\Phi-R_\Phi)(L_\Phi^{-1}-R_\Phi^{-1}) 
\end{matrix}
\right]
\end{aligned}
\label{qqpphi}
\end{gather}
of vector bundles on $M$, and
\begin{equation*}
\pp_\Phi + \qqp_\Phi = \Id \colon 
\TT_\Phi \GG \oplus_M \TT^*_\Phi \GG 
\longrightarrow
\TT_\Phi \GG \oplus_M \TT^*_\Phi \GG .
\end{equation*}
\end{prop}
In view of {\rm \eqref{laterr}},
the off-diagonal terms of \eqref{ppphi} and
\eqref{qqpphi} admit the following equivalent expressions:
\begin{align*}
\pp_{\Phi,2,1}&=\tfrac 18
(L^{*,-1}_\Phi+R^{*,-1}_\Phi)(\Id \times \psidot) (L_\Phi^{-1}- R_\Phi^{-1})
\colon& \TT_\Phi \GG \to \TT^*_\Phi \GG
\\
\pp_{\Phi,1,2}&=\tfrac 12
(L_\Phi-R_\Phi)(\Id \times \psi^{\form,-1})(L_\Phi^*+R_\Phi^*)
\colon& \TT^*_\Phi \GG \to \TT_\Phi \GG
\\
\qqp_{\Phi,2,1}&= \tfrac 18 
(L_\Phi^{*,-1}-R^{*,-1}_\Phi)(\Id \times  \psidot) (L_\Phi^{-1}+R_\Phi^{-1})
\colon& \TT_\Phi \GG \to \TT^*_\Phi \GG
\\
\qqp_{\Phi,1,2}&=\tfrac 12
(L_\Phi+R_\Phi)(\Id \times\psi^{\form,-1})(L_\Phi^*- R_\Phi^*)
\colon& \TT^*_\Phi \GG \to \TT_\Phi \GG 
\end{align*}

\begin{rema}
{\rm
The expressions in \eqref{ppphi} and \eqref{qqpphi}
for the off-diagonal terms of
 $\pp_\Phi$ and $\qqp_\Phi$, respectively,
involve the non-degeneracy of $\,\form\,$
explicitly.
Hence there is no way to subsume
the theory of weakly quasi Hamiltonian structures
and Hamiltonian weakly quasi Poisson structures 
under Dirac structures.
}
\end{rema}
For $\Phi = \Id$, write the projection \eqref{ppphi} as
$ \pp \colon\TT \GG \oplus_\GG \TT^* \GG \to E_\GG$.
By construction, \eqref{ppphi}
determines the $\GG$-invariant bivector
$P_\GG \in \LAL^{\mrc,2}[\GG]$
as 
\begin{equation}
P_\GG^\sharp = \pp_{1,2} = \tfrac 12 (L-R)(L^{-1}+R^{-1}) \psi_\GG^{\form,-1}
= \tfrac 12 (L-R)(\Id \times \psi^{\form,-1})(L^*+R^*), 
\end{equation}
cf. 
\eqref{Psharp} for the notation ${\cdot\,}^\sharp$ and
\eqref{laterr} for the third equality;
indeed,  diagram \eqref{CD03}
being commutative characterizes  $P_\GG^\sharp$,
and the adjoint  $\psi^\oomega \colon \gg^* \to \gg$  
of the $2$-tensor  $\oomega$ in $\gg \otimes \gg$ coincides with
 $ \psidot^{,-1}$.

\subsubsection{Dirac  vs quasi structures}
\label{dvsq}
Maintain the hypthesis that the $2$-form  $\form$ on $\gg$ be non-degenerate
and retain the notation
 $\oomega$ be the symmetric $\Ad$-invariant $2$-tensor in 
$\gg \otimes \gg$ which $\form$ determines.
Theorems \ref{dvsq1} and \ref{dvsq2} below implicitly involve the fact that,
cf. Proposition \ref{cruc},
 an  almost Dirac morphism $(\Phi,\sigma) \colon (M,E_M) \to (\GG, E_\GG)$ is strong
if and only if
the $(\Phi,\sigma)$-backward image $F^{\Phi,\sigma}$ of $F_\GG$
is transverse to $E_M$.

We summarize a 
version of \cite[Theorem 7.6(i)]{MR2068969} (statement (ii) 
in this theorem is not relevant
in the present paper), see also
\cite[Theorem 3.15]{MR2103001},
reproved and generalized as
\cite[Theorem 5.2]{MR2642360},
together with the reasoning in the proofs
in a form tailored to our purposes:

\begin{thm}
\label{dvsq1}
Let $M$ be a $\GG$-manifold together with an admissible $\GG$-equivariant
map $\Phi \colon M \to \GG$ and a $\GG$-invariant $2$-form
$\sigma$. 

\begin{enumerate}
\item
The map $\Phi$ is a $\GG$-momentum for
$\sigma$ 
relative to $\form$
if and only if
\begin{equation}
(\Phi,\sigma) \colon (M, \TT M) \longrightarrow (\GG, E_\GG)
\label{adm1}
\end{equation}
is an almost Dirac morphism.

\item
When $(\Phi,\sigma)$ is an almost Dirac morphism, 
the $2$-form
$\sigma$ is $\Phi$-quasi closed if and only if
\begin{equation}
(\Phi,\sigma) \colon (M, \TT M,0) \longrightarrow (\GG, E_\GG,\lambda_\GG)
\label{adm1l}
\end{equation}
is a Dirac morphism (of twisted Dirac manifolds).
\item
The following are equivalent:
\begin{enumerate}
\item
The map $\Phi$ is a $\GG$-momentum for
$\sigma$ 
relative to $\form$ and the $2$-form
$\sigma$ on $M$ is $\Phi$-non-degenerate;

\item
the map $\Phi$ and $2$-form $\sigma$ combine to a strong
almost Dirac morphism of the kind
{\rm \eqref{adm1}};

\item
the
 $(\Phi,\sigma)$-backward image $F^{\Phi,\sigma}$ of $F_\GG$
is transverse to $\TT M$ and hence yields the  $\GG$-invariant
Lagrangian splitting
$\TT M \oplus _M F^{\Phi,\sigma}$ 
of $\TT M \oplus _M \TT^* M$;

\item
the
 $\Phi$-backward image $F^\Phi$ of $F_\GG$
is transverse to the graph $\Gr_\sigma$ of $\sigma$
and hence yields the  $\GG$-invariant
Lagrangian splitting
$\Gr_\sigma \oplus _M F^\Phi$ 
of $\TT M \oplus _M \TT^* M$. \qed
\end{enumerate}

\end{enumerate}
\end{thm}

\begin{rema}
{\rm
As for Theorem \ref{dvsq1}(3), the equivalence
of (a) and (b) is precisely
\cite[Theorem 5.2]{MR2642360} for the group case.
The equivalence with (c) and (d)
is a consequence of Proposition \ref{cruc}.
}
\end{rema}

In the statement of Theorem \ref{dvsq1}, when the group is trivial,
$(0,\sigma)$ is necessarily an almost Dirac morphism,
the Lagrangian subbundle $F^{\Phi,\sigma}$ of $\TT M \oplus_M \TT^* M$
comes down to $\ker(0,\sigma) \subseteq \TT M \oplus_M \TT^* M$,
and (3) says that $\sigma$ is non-degenerate if and only if
$\TT M$ is transverse to  $\ker(0,\sigma)$.

\begin{compl}
\label{comp2}
Under the circumstances of Theorem {\rm \ref{dvsq1}},
when $\sigma$ is $\Phi$-non-degenerate, 
the $\GG$-invariant skew-symmetric $2$-tensor $P$ on $M$ which
the projection from
$\TT M \oplus_M \TT^* M$ to $\Gr_\sigma$
relative to the $\GG$-invariant Lagrangian splitting
$\Gr_\sigma \oplus _M F^\Phi$ 
of $\TT M \oplus _M \TT^* M$ defines is
the $\Phi$-momentum dual of $\sigma$.
\end{compl}

\begin{proof}
The projection 
\begin{equation}
\pp = \left[ 
\begin{matrix}
\pp_{1,1} & \pp_{1,2}\\ \pp_{2,1} & \pp_{2,2} 
\end{matrix}
\right] \colon \TT M\oplus_M \TT^* M \longrightarrow \Gr_\sigma \subseteq \TT M\oplus_M \TT^* M
\label{projjj}
\end{equation}
to $\Gr_\sigma$
relative to the Lagrangian splitting $\Gr_\sigma \oplus_M F^\Phi$
of $\TT M \oplus_M \TT^* M$
defines
the $\GG$-invariant skew-symmetric $2$-tensor $P$ on $M$ as 
$P^\sharp =\pp_{1,2}$.

Since $\Gr_\sigma =\{(X, \sigma^\fflat(X)\} \subseteq \TT M\oplus_M \TT^* M$,
\begin{align*}
\pp_{2,1} &= \sigma^\fflat \circ \pp_{1,1} 
\\
\pp_{2,2} &= \sigma^\fflat \circ \pp_{1,2}= \sigma^\fflat \circ P^\sharp.
\end{align*}
The  projection to $F^\Phi$ 
relative to the Lagrangian splitting $\Gr_\sigma \oplus_M F^\Phi$
of $\TT M \oplus_M \TT^* M$
reads
\begin{align*}
\qqp = \Id - \pp =\left[ 
\begin{matrix}
\Id -\pp_{1,1} & -P^\sharp\\-\sigma^\fflat\circ\pp_{1,1}&\Id- \sigma^\fflat \circ P^\sharp 
\end{matrix}
\right]
&\colon \TT M \oplus_M \TT^* M 
\twoheadrightarrow F^\Phi \rightarrowtail \TT M \oplus_M \TT^* M.
\end{align*}
Let $\alpha \in \TT^* M$. Then
$
\qqp\left[ 
\begin{matrix}
0\\ \alpha
\end{matrix}
\right]
=
\left[ 
\begin{matrix}
 - P^\sharp(\alpha) \\ 
\alpha -(\sigma^\fflat \circ P^\sharp)(\alpha)
\end{matrix}
\right] \in F^\Phi .
$
Hence, by the definition of $F^\Phi$, for some $\beta \in\TT^*_\Phi\GG$,
necessarily unique, such that
$ (-(d \Phi)_M( P^\sharp(\alpha)),\beta) \in F_\Phi$,
\begin{align*}
\alpha -(\sigma^\fflat \circ P^\sharp)(\alpha) &= \beta \circ (d \Phi)_M .
\end{align*}
Since $\Phi$ is a $\GG$-momentum mapping for
$P$ relative to $\oomega$, 
diagram \eqref{CD01d} being commutative,
\begin{align*}
-(d \Phi)_M( P^\sharp(\alpha))&= 
\tfrac 12 (L_\Phi + R_\Phi)(\Id \times \Psi^\oomega)(\fund_M^*(\alpha)) .
\end{align*}
In view of the definition of 
the almost Dirac structure
$F_\Phi \subseteq  \TT_\Phi\GG \oplus_M \TT^*_\Phi\GG$,
cf. Proposition \ref{via1},
necessarily
\begin{align*}
\beta&=\tfrac 14 (L^{*,-1}_\Phi + R^{*,-1}_\Phi)(\fund_M^*(\alpha))
\\
\alpha -\sigma^\fflat P^\sharp(\alpha) &= \beta \circ (d \Phi)_M 
= (d \Phi)^*_M (\beta)
\\
&=\tfrac 14 \left((d \Phi)^*_M \circ
(L^{*,-1}_\Phi + R^{*,-1}_\Phi) \circ \fund_M^*\right)(\alpha).
\end{align*}
Since $\alpha$ is arbitrary, we conclude
\begin{align*}
\Id -\sigma^\fflat P^\sharp &=\tfrac 14 \rho_\Phi^*
\\
\Id - P^\sharp \sigma^\fflat&=\tfrac 14 \rho_\Phi . \qedhere
\end{align*}
\end{proof}

The following restates a version of
\cite[Theorem 3.16]{MR2103001}, reproved and generalized as
 \cite[Theorem 5.22]{MR2642360},
tailored to our purposes; our version is
weaker than \cite[Theorem 5.22]{MR2642360} since we already build in
the requisite $\GG$-symmetries (which in that theorem
result from infinitesimal symmetries)
but spells out a precise bijective correspondence
which those theorems  claim to exist and which
only the proofs render explicit.

\begin{thm}
\label{dvsq2}
Let 
 $M$  be a $\GG$-manifold and 
 $\Phi \colon M \to \GG$  an admissible $\GG$-equivariant map.
The assignment to a $\GG$-invariant almost Dirac structure
$E_M\subseteq \TT M \oplus \TT^* M$ on $M$
having
$\Phi \colon(M,E_M) \to (\GG, E_\GG)$
as a $\GG$-invariant strong almost Dirac morphism
of the $\GG$-invariant skew-symmetric $2$-tensor $P$ 
on $M$ 
which
the projection from  $\TT M \oplus_M \TT^*M$ to $E_M$
relative to the $\GG$-invariant
Lagrangian splitting  $E_M \oplus_M F^\Phi$ of
$\TT M \oplus_M \TT^*M$
defines
(cf. {\rm \eqref{projP}})
establishes a bijective correspondence
between such $\GG$-invariant 
almost Dirac structures
$E_M$ on $M$ and $\GG$-invariant skew-symmetric $2$-tensors $P$ 
on $M$ having $\Phi$ as $\GG$-momentum mapping relative to $\oomega$
and enjoying the following property:

The image of $E_M$
in $\TT M$ under the projection from $\TT M \oplus_M \TT^*M$ to $\TT M$
coincides with
the image of the 
vector bundle morphism \eqref{surj21}, viz.
\begin{equation*}
P^\sharp+ \fund_M\colon
\TT ^* M \oplus_M M \times \gg \to \TT M,
\end{equation*}
{\rm (the property  \lq $\ran(E_M)= \ran(\fund_M) + \ran(P^\sharp)$\rq\ 
in \cite[Theorem 5.22]{MR2642360})}.

Under this correspondence, the data 
$(M,E_M, \Phi^* \lambda_\GG)$ constitute a twisted
Dirac manifold and
$\Phi \colon(M,E_M, \Phi^* \lambda_\GG ) \to (\GG, E_\GG,\lambda_\GG)$
is a $\GG$-invariant strong  Dirac morphism
if and only if the associated skew-symmetric $2$-tensor $P$ on $M$
is a $\GG$-quasi Poisson structure relative to $\oomega$. \qed
\end{thm}

Under the circumstances of Theorem \ref{dvsq2},
given the $\GG$-invariant skew-symmetric $2$-tensor $P$ 
on $M$ having $\Phi$ as $\GG$-momentum mapping relative to $\oomega$,
we denote by
$E_P$  the corresponding 
$\GG$-invariant almost Dirac structure on $M$
written there as $E_M$ and refer to it as
being {\em associated to\/} $P$.

\begin{compl}
\label{comp1}
Under the circumstances of Theorem {\rm \ref{dvsq2}},
the $\GG$-invariant skew-sym\-metric $2$-tensor $P$ 
on $M$ is $\GG$-quasi non-degenerate if and only if
the $\GG$-invariant almost Dirac structure
$E_P$ associated to $P$ is transverse to $\TT ^* M$.
Furthermore,
if this happens to be the case,
the $\GG$-invariant  $2$-form $\sigma$ on $M$ which
the projection from
$\TT M \oplus_M \TT^* M$ to $E_P$
relative to the Lagrangian splitting
$E_P \oplus_M\TT ^* M$
of $\TT M \oplus _M \TT^* M$ defines is
the $\Phi$-momentum dual of $P$, and
$E_P$ coincides with the graph $\Gr_\sigma$ of $\sigma$.
\end{compl}

\begin{proof}
By definition, the skew-sym\-metric $2$-tensor $P$ 
is $\GG$-quasi non-degenerate if and only if
the morphism \eqref{surj21}
of vector bundles is onto $\TT M$.
Since, by Theorem \ref{dvsq2}, 
the image of  \eqref{surj21} in $\TT M$ coincides with
the image of $E_P$ in $\TT M$ under the projection from
$\TT M \oplus_M\TT ^* M$ to $\TT M$,
we conclude that
$P$ 
is $\GG$-quasi non-degenerate if and only if $E_P$
is transverse to $\TT ^*M$.

Suppose that $P$ is $\GG$-quasi non-degenerate. 
The projection to $E_P$
relative to the Lagrangian splitting
$E_P \oplus_M\TT ^* M$
of $\TT M \oplus _M \TT^* M$ reads
\begin{equation}
\pp= \left[ 
\begin{matrix}
\Id & 0
\\ 
\sigma^\fflat & 0 
\end{matrix}
\right] \colon  \TT M \oplus_M \TT^* M \twoheadrightarrow E_P
\rightarrowtail \TT M \oplus_M \TT^* M
\end{equation}
and defines the $\GG$-equivariant $2$-form $\sigma$
on $M$ 
having $\Phi$ as $\GG$-momentum mapping  relative to $\,\form\,$.
It is immediate that $E_P$ coincides with the graph $\Gr_\sigma$
of $\sigma$.
By Complement \ref{comp2},
$\sigma^\fflat \circ P^\sharp =\Id - \tfrac 14 \rho_\Phi^*$,
that is,
the $2$-form $\sigma$
on $M$ 
is
the $\Phi$-momentum dual of $P$.
\end{proof}

In  Theorem \ref{dvsq2}, when the group is trivial,
the backward image $F^0$ amounts to $\TT M$,
the almost Dirac structures of the kind $E_M$ amount to
graphs  of  skew-symmetric $2$-tensors, 
and the statement comes down  the bijective correspondence
between  skew-symmetric $2$-tensors $P$ on $M$ and the graphs
$\Gr_P$ defining them via the projection from
$\TT M \oplus_M \TT^* M$ to $\Gr_P$ relative to the Lagrangian 
splitting $\TT M \oplus_M \Gr_P$ of $\TT M \oplus_M \TT^* M$, cf.
\eqref{projP}.

The following is a Dirac theoretic version of Theorem \ref{existence}.

\begin{thm}
\label{dvsq3}
Let
 $M$  be a $\GG$-manifold and 
 $\Phi \colon M \to \GG$  an admissible $\GG$-equivariant map.
Further, let $\sigma$ be a $\GG$invariant $2$-form 
and $P$ a
$\GG$ invariant skew-symmetric $2$-tensor
on $M$ both having $\Phi$ as $\GG$-momentum mapping
(relative to $\form$ and $\oomega$, respectively).
Then $\sigma$ and $P$ are $\Phi$-momentum dual to each other 
(and hence, in particular, $\sigma$ is  $\Phi$-non-degenerate
and $P$ is $\GG$-quasi non-degenerate)
if and only if
the graph 
$\Gr_\sigma$ of $\sigma$ coincides with the
$\GG$-invariant almost Dirac structure $E_P$ associated to $P$
in Theorem {\rm \ref{dvsq2}}.

Furthermore, when
$\sigma$ and $P$ are $\Phi$-momentum dual to each other,
the $2$-form $\sigma$ is $\Phi$-quasi closed if and only if
the skew-symmetric $2$-tensor $P$ is $\GG$-quasi Poisson.

\end{thm}

\begin{proof}
Suppose that
$\sigma$ and $P$ are $\Phi$-momentum dual to each other.
By Proposition \ref{compar32},
the $2$-form $\sigma$ is $\Phi$-non-degenerate and
$P$ is $\GG$-quasi-non-degenerate.
By Theorem \ref{dvsq1}(3),
the
 $\Phi$-backward image $F^\Phi$ of $F_\GG$
is transverse to the graph $\Gr_\sigma$ of $\sigma$
 and,
by Complement \ref{comp1},
the $\GG$-invariant
almost Dirac structure
$E_P$ on $M$ 
 is transverse to $\TT ^* M$. 
Hence 
$\Gr_\sigma \oplus _M F^\Phi$ 
and
$E_P \oplus_M\TT ^* M$
yield   $\GG$-invariant
Lagrangian splittings
of $\TT M \oplus _M \TT^* M$.
In view of the uniqueness of the 
$\Phi$-momentum duals, cf. Proposition \ref{unique},
by Complement \ref{comp2},
the projection to 
$\Gr_\sigma$ relative to the former decomposition
defines the $\Phi$-momentum dual $P$ of $\sigma$ and,
by Complement \ref{comp1},
the projection to 
$E_P$ relative to the latter decomposition
defines the $\Phi$-momentum dual $\sigma$ of $P$.
By  Complement \ref{comp1},
the graph $\Gr_\sigma$ of $\sigma$ coincides with $E_P$.

Conversely, suppose
the graph 
$\Gr_\sigma$ of $\sigma$ coincides with the
$\GG$-invariant almost Dirac structure $E_P$ associated to $P$.
Then
$\Gr_\sigma \oplus _M F^\Phi$ 
and
$E_P \oplus_M\TT ^* M$
yield   $\GG$-invariant
Lagrangian splittings
of $\TT M \oplus _M \TT^* M$,
and the projection to 
$\Gr_\sigma$ relative to the former decomposition
defines  $P$  and
the projection to 
$E_P$ relative to the latter decomposition
defines $\sigma$.
By Complement \ref{comp2},
the $2$-tensor $P$
is the $\Phi$-momentum dual of $\sigma$ and,
by Complement \ref{comp1},
the $2$-form $\sigma$
is the $\Phi$-momentum dual of $P$.

Suppose that $\sigma$ and $P$ are $\Phi$-dual to each other.
Suppose that, furthermore,
$\sigma$ is $\Phi$-quasi closed,
that is, $d \sigma = \Phi^* \lambda_\GG$, cf. \eqref{qh1}.
Then $(M,\Gr_\sigma, \Phi^*\lambda_\GG)$
is a twisted Dirac manifold, cf. Subsection \ref{dsom}, 
and
\begin{equation*}
\Phi\colon (M,\Gr_\sigma, \Phi^*\lambda_\GG) \longrightarrow (\GG,E_\GG, \lambda_\GG)
\end{equation*}
is  a strong 
Dirac morphism.
By Theorem \ref{dvsq2},
the $2$-tensor
$P$ is a $\GG$-quasi Poisson structure.

In the same vein,
suppose that, furthermore, the $2$-tensor
$P$ on $M$ is $\GG$-quasi Poisson. 
By Theorem \ref{dvsq2}, the data 
$(M,\Gr_\sigma,\Phi^* \lambda_\GG)$
then constitute a twisted Dirac manifold, 
that is, the $2$-form $\sigma$ on $M$ satisfies the identity
$d\sigma =\Phi^* \lambda_\GG$,
cf. Subsection \ref{dsom}, i.e., 
the $2$-form $\sigma$ is $\Phi$-quasi closed.
\end{proof}

In  Theorem \ref{dvsq3}, when the group is trivial,
the statement comes down to a Dirac theoretic 
characterization of the non-degeneracy of a $2$-form on $M$
and of that of the corresponding skew-symmetric $2$-tensor on $M$
of the kind in
Proposition \ref{latref1}.

\subsubsection{Proof of Theorem {\rm \ref{existence}}}
Complement \ref{comp2} establishes the
existence of the $\Phi$-momentum dual $P$ of $\sigma$
and 
Complement \ref{comp1} that  of the $\Phi$-momentum dual $\sigma$ 
of $P$. Once the existence of $P$ and that of $\sigma$
is established, it is immediate that
\eqref{render1} characterizes
$P$ in terms of $\sigma$ and 
 that
\eqref{render2} characterizes
$\sigma$ in terms of $P$.
Furthermore,
when
$\sigma$ and $P$ are $\Phi$-momentum dual to each other,
by Theorem \ref{dvsq3},
the $2$-form $\sigma$ is $\Phi$-quasi closed if and only if
the skew-symmetric $2$-tensor $P$ is $\GG$-quasi Poisson.
This completes the proof of Theorem \ref{existence}.

\subsection{Double}
\label{double}
As before, we use the notation $\GG$ for a Lie group and $\gg$ for 
its Lie algebra.

\begin{prop}
Suppose the $\Ad$-invariant symmetric $2$-tensor $\oomega$ on $\gg$
is non-degenerate.
Then the (external)  Hamiltonian   quasi 
Poisson double of $(\GG,\oomega)$
is $(\GG \times \GG)$-quasi non-degenerate,
and the internally fused  Hamiltonian   quasi 
Poisson double
of $(\GG,\oomega)$
is $\GG$-quasi non-degenerate.
\end{prop}

\begin{proof}
Left to the reader.
In view of Proposition \ref{compar32} (1),
the claim is also a consequence of Proposition \ref{d.4}
below.
\end{proof}

While the following observation is a consequence of Theorem \ref{existence},
we give a proof that is independent of that theorem.
\begin{prop}
\label{d.4}
Suppose
the symmetric $\Ad$-invariant $2$-tensor 
 $\oomega$ in $\gg \otimes \gg$ arises from 
a non-degenerate  $\Ad$-invariant symmetric bilinear form $\,\form\,$ on $\gg$
as the image of
the corresponding $2$-tensor in $\gg^* \otimes \gg^*$
under the inverse of the  adjoint $\psidot \colon \gg \to \gg^*$
of $\,\form\,$.
The 
weakly
$(\GG \times \oGG)$-quasi closed $2$-form
$\sigma^\times$ relative to $\,\form\,$ on $\GG \times \GG$, 
cf. {\rm \eqref{sigmatimes}},
and the 
$(\GG \times \oGG)$-quasi
Poisson structure
$P^\times_\oomega$ relative to $\oomega^\times$
on $\GG \times \GG$,
cf. Theorem {\rm \ref{momfusG}},
are dual to each
other via the $(\GG \times \oGG)$-momentum mapping 
\begin{equation*}
(\mult,\omult)\colon \GG \times \GG \to \GG \times \oGG
\end{equation*}
for 
$\sigma^\times$
relative to $\,\form\,$ and for $P^\times_\oomega$ 
relative to  $\oomega^\times$.
The same statement holds for
 the
internally fused double
$(\GG \times \GG,P_1,\sigma_1,\Phi_1)$:
The 
weakly
$\Phi$-quasi closed $2$-form
$\sigma_1$ relative to $\,\form\,$ on $\GG \times \GG$, 
cf. Subsection {\rm \eqref{fusq}},
and the 
$\GG$-quasi
Poisson structure
$P_1$ relative to $\oomega$
on $\GG \times \GG$,
cf. Subsection {\rm \ref{internallyfused}},
are dual to each
other via the $\GG $-momentum mapping 
$\Phi_1\colon \GG \times \GG \to \GG$
for 
$\sigma_1$
relative to $\,\form\,$ and for $P_1$ 
relative to  $\oomega$.
\end{prop}

\begin{proof}
Let $e_1,\ldots, e_\ell, e_{\ell +1},\ldots, e_d$
be a basis of $\gg$
and, cf. \eqref{Ptimes},
write
\begin{align*}
\oomega &= \eeta^{j,k} e_j \otimes e_k
\\
P^\times_\oomega &= \tfrac 12 \left(L^1 \wedge R^2 + R^1 \wedge L^2 \right)(\oomega)  
= \tfrac 12 
\eeta^{j,k}\left (e^{1,L}_j \wedge  e^{2,R}_k + e^{1,R}_j \wedge  e^{2,L}_k\right).
\end{align*}
Let $\eta^1,\ldots, \eta^\ell, \eta^{\ell +1},\ldots, \eta^d$
be the dual  basis of $\gg^*$ and,
for $1 \leq j \leq d$,
 introduce the notation
$\omega^j = \eeta^j \circ \omega$
and
$\ovomega^j = \eeta^j \circ \ovomega$
for the components of $\omega$ and $\ovomega$ in the basis
$e_1,\ldots, e_d$ of $\gg$, so that
\begin{equation*}
\omega = \omega^j e_j \colon \TT \GG \to \gg,\quad 
\ovomega = \ovomega^j e_j\colon \TT \GG \to \gg.
\end{equation*}
In terms of this notation, for  $\alpha \in \Form^1(\GG^1 \times \GG^2)$,
\begin{align*}
P^{\times, \sharp}_\oomega(\alpha)&=\tfrac 12 \eeta^{j,k}
\left ( 
\langle e^{1,L}_j,\alpha\rangle e^{2,R}_k 
-
\langle  e^{2,R}_k,\alpha\rangle e^{1,L}_j
+ 
\langle e^{1,R}_j, \alpha\rangle e^{2,L}_k
-
\langle e^{2,L}_k, \alpha\rangle e^{1,R}_j
\right).
\end{align*}
Next, recall
\begin{align*}
\sigma^\times &=-\tfrac 12
\left( \omega_1 \form \ovomega_2 + \ovomega_1 \form \omega_2
\right)
=
-\tfrac 12
\left(
(\omega_1^j e_j) \form(\ovomega_2^k e_k)
+(\ovomega_1^j e_j) \form(\omega_2^k e_k)
\right)
\\
&=
-\tfrac 12
\eeta_{j.k}
\left(
\omega_1^j \wedge \ovomega_2^k +\ovomega_1^j \wedge \omega_2^k
\right).
\end{align*}
Let $X \in \Vect(\GG^1 \times \GG^2)$. Then
\begin{align*}
\sigma^{\times,\flat}(X)&=-\tfrac 12
\eeta_{j.k}
\left(
\langle \omega_1^j,X\rangle\ovomega_2^k 
-
\langle \ovomega_2^k ,X\rangle\omega_1^j
+
\langle\ovomega_1^j,X\rangle \omega_2^k
-
\langle\omega_2^k,X\rangle \ovomega_1^j
\right),\ 
\\
&=
-\tfrac 12
\eeta_{u.v}
\left(
\langle \omega_1^u,X\rangle\ovomega_2^v 
-
\langle \ovomega_2^v ,X\rangle\omega_1^u
+
\langle\ovomega_1^u,X\rangle \omega_2^v
-
\langle\omega_2^v,X\rangle \ovomega_1^u
\right),
\\
&=
\tfrac 12
\eeta_{u.v}
\left(
 \langle \ovomega_2^v ,X\rangle\omega_1^u
-\langle \omega_1^u,X\rangle\ovomega_2^v 
+\langle\omega_2^v,X\rangle \ovomega_1^u
-\langle\ovomega_1^u,X\rangle \omega_2^v
\right).
\end{align*}
A calculation yields the following:
\begin{align*}
P^{\times, \sharp}_\oomega(\sigma^{\times,\flat}(X))
&=\tfrac 14 \eeta^{j,k} \eeta_{u,v}
\left( 
\begin{cases}
\phantom{-}
\delta_j^u \langle \ovomega_2^v ,X\rangle
e^{2,R}_k 
\\
+
\delta_k^v \langle \omega_1^u,X\rangle 
e^{1,L}_j 
\\
+
\delta_j^u\langle\omega_2^v,X\rangle
e^{2,L}_k 
\\
+
\delta_k^v \langle\ovomega_1^u,X\rangle
e^{1,R}_j
\end{cases}
+
\begin{cases}
\phantom{-}
\langle e^{1,L}_j, \ovomega_1^u\rangle\langle\omega_2^v,X\rangle
e^{2,R}_k 
\\
+
\langle e^{2,R}_k, \omega_2^v\rangle \langle\ovomega_1^u,X\rangle
e^{1,L}_j 
\\
+
\langle e^{1,R}_j,\omega_1^u\rangle  \langle \ovomega_2^v ,X\rangle
e^{2,L}_k 
\\
+
\langle e^{2,L}_k,\ovomega_2^v \rangle  \langle \omega_1^u,X\rangle
e^{1,R}_j
\end{cases}
\right)
\end{align*}

\begin{align*}
\tfrac 14 \eeta^{j,k} \eeta_{u,v}
\begin{cases}
\phantom{-}
\delta_j^u \langle \ovomega_2^v ,X\rangle
e^{2,R}_k 
\\
+
\delta_k^v \langle \omega_1^u,X\rangle 
e^{1,L}_j 
\\
+
\delta_j^u\langle\omega_2^v,X\rangle
e^{2,L}_k 
\\
+
\delta_k^v \langle\ovomega_1^u,X\rangle
e^{1,R}_j
\end{cases}
&= \tfrac 12 X
\end{align*}

\begin{align*}
\tfrac 14 \eeta^{j,k} \eeta_{u,v}
\begin{cases}
\phantom{-}
\langle e^{1,L}_j, \ovomega_1^u\rangle\langle\omega_2^v,X\rangle
e^{2,R}_k 
\\
+
\langle e^{2,R}_k, \omega_2^v\rangle \langle\ovomega_1^u,X\rangle
e^{1,L}_j 
\\
+
\langle e^{1,R}_j,\omega_1^u\rangle  \langle \ovomega_2^v ,X\rangle
e^{2,L}_k 
\\
+
\langle e^{2,L}_k,\ovomega_2^v \rangle  \langle \omega_1^u,X\rangle
e^{1,R}_j
\end{cases}
&= \tfrac 12 X - \tfrac 14 \rho_{\mult,\omult}(X) \qedhere
\end{align*}
\end{proof}

\subsection{Conjugacy classes}
\label{conjclasses}
As before, $\GG$ denotes a Lie group and $\gg$ its Lie algebra.

\begin{prop}
Suppose the $\Ad$-invariant symmetric $2$-tensor $\oomega$ on $\gg$
is non-degenerate.
Then the $\GG$-quasi Poisson structure $P_{\CcC}$
relative to $\oomega$ on a conjugacy class $\CcC$ in $\GG$
is $\GG$-quasi non-degenerate.
\end{prop}

\begin{proof}
Left to the reader.
In view of Proposition \ref{compar32} (1),
the claim is also a consequence of Proposition \ref{d.5}
below.
\end{proof}

Again  the following observation is a consequence of Theorem \ref{existence},
but we give a proof that is independent of that theorem.
\begin{prop}
\label{d.5}
Suppose
the symmetric $\Ad$-invariant $2$-tensor 
 $\oomega$ in $\gg \otimes \gg$ arises from 
a non-degenerate  $\Ad$-invariant symmetric bilinear form $\,\form\,$ on $\gg$
as the image of
the corresponding $2$-tensor in $\gg^* \otimes \gg^*$
under the inverse of the  adjoint $\psidot \colon \gg \to \gg^*$
of $\,\form\,$ and let ${\CcC}$ be a conjugacy class in $\GG$.
Then the 
$\iota$-quasi closed $2$-form
$\tau_\CcC$ relative to $\,\form\,$ on ${\CcC}$, 
cf. Propositiion {\rm \ref{prop1c}},
and the 
$\GG$-quasi
Poisson structure
$P_{\CcC}$ relative to $\oomega$ on ${\CcC}$,
cf. Proposition {\rm \ref{momG}},
are dual to each
other via the $\GG$-momentum mapping 
$\iota \colon {\CcC} \to \GG$
for 
$\tau_\CcC$
relative to $\,\form\,$ and for $P_{\CcC}$ 
relative to  $\oomega$.
\end{prop}

\begin{proof}
Let $q$ be a point of $\CcC$ and $X,Y \in \gg$.
The morphism $\rho_\iota \colon \TT \CcC \to \TT \CcC$
of vector bundles on $\CcC$, cf. \eqref{3rhowrite},
sends the vector
$Xq-qX \in \TT_q \CcC$
to the vector
\begin{align*}
{}&2(Xq-qX) 
- \left( (\Ad_{q^{-1}} X)q-q\Ad_{q^{-1}}X \right)
- \left( (\Ad_q X)q-q\Ad_qX \right)
\\
&=2(Xq-qX) 
- \left( (\Ad_{q^{-1}} X)q-q\Ad_{q^{-1}}X \right)
- \left( (\Ad_q X)q-q\Ad_qX \right)
\end{align*}
in $\TT_q \CcC$. Hence
\begin{align*}
(4\Id - \rho_\iota)(Yq-qY)&=2 (Yq-qY)
+\left( (\Ad_{q^{-1}} Y)q-q\Ad_{q^{-1}}Y \right)
+ \left( (\Ad_q Y)q-q\Ad_qY \right)
\\
&=
2 (Yq-qY)
+ q^{-1} Yq^2- Yq
+ qY -q^2Y q^{-1}
\\
&=
Yq-qY+ q^{-1} Yq^2 -q^2Y q^{-1} .
\end{align*}

By
construction,
cf.  \eqref{PG} and \eqref{PG2}
as well as
Proposition \ref{momG},
\begin{align*}
2P_G&=   (R \otimes L - L \otimes R)(\oomega)
\\
2 P_{\CcC}&=  (R \otimes L - L \otimes R)(\oomega).
\end{align*}
Let $\alpha_q, \beta_q \in \TT^*_q \GG$. Then
\begin{align*}
2P_G(\alpha_q,\beta_q)&
=\langle \psi^\oomega(\alpha_q \circ R_q),\beta_q \circ L_q\rangle
-\langle \psi^\oomega(\alpha_q \circ L_q),\beta_q \circ R_q\rangle .
\end{align*}
The $\GG$-momentum property of $\Id \colon \GG \to \GG$ for $P_\GG$
relative to $\oomega$, equivalently diagram \eqref{CD03}
being commutative, implies
\begin{align*}
2P_G^\sharp(\alpha_q)&=
q(\Psi^\oomega(\alpha_q\circ (L_q + R_q)))
  -(\Psi^\oomega(\alpha_q\circ (L_q + R_q)))q .
\end{align*}
By construction, cf.  Proposition \ref{prop1c},
\begin{align*}
2 (\tau^\flat_\CcC(Xq-qX))(Yq-qY)=2 \tau_\CcC(Xq-qX,Yq-qY)&=X \form \Ad_q Y - Y \form \Ad_q X
\end{align*}
Let $X=\psi^\oomega(\alpha_q\circ (L_q + R_q))$, so that $2P_\GG^\sharp(\alpha_q)=qX-Xq$.
Then
\begin{align*}
4\tau^\flat_\CcC(P_\GG^\sharp(\alpha_q))(Yq-qY)&=
-2\tau_\CcC(Xq-qX,Yq-qY)
\\
&=
-(\psi^\oomega(\alpha_q\circ (L_q + R_q)))\form\Ad_q Y
+
(\psi^\oomega(\alpha_q\circ (L_q + R_q)))\form\Ad_{q^{-1}} Y
\\
&=
-(\alpha_q\circ (L_q + R_q))(\Ad_q Y)
+
(\alpha_q\circ (L_q + R_q))(\Ad_{q^{-1}} Y)
\\
&=
\alpha_q(-q^2 Yq^{-1} -qY +Yq +q^{-1} Y q^2)
\\
&=
\alpha_q(Yq -qY +q^{-1} Y q^2-q^2 Yq^{-1})
\\
&=
\alpha_q(
2 (Yq-qY)
+ q^{-1} Yq^2- Yq
+ qY -q^2Y q^{-1})
\\
&=
\alpha_q\left((4 \Id - \rho_\iota)(Yq-qY)\right).
\end{align*}
Consequently
\begin{equation*}
4\tau^\flat_\CcC \circ P_\GG^\sharp =4 \Id - \rho_\iota^*
\colon \TT^* \CcC \to \TT^* \CcC,
\end{equation*}
that is, \eqref{31d} holds for $(\CcC,\tau_\CcC, P_\CcC,\iota)$,
whence
$\tau_\CcC$ and $ P_\CcC$ are dual to each other via $\iota$
as asserted.
\end{proof}

\section{Moduli spaces over a Riemann surface revisited}
\label{msrrev}

As before, $\GG$ denotes a Lie group and $\oomega$ an $\Ad$-invariant 
symmetric $2$-tensor in the tensor quare $\qq \otimes \gg$ of
 the Lie algebra $\gg$ of $\GG$.
We recall 
from Subsection \ref{internallyfused}
that the internally fused double $(\GG \times \GG, P_1,\Phi_1)$
is a $\GG$-quasi Poisson manifold.
Successively fusing $\ell \geq 2$ copies of 
$(\GG \times \GG, P_1,\Phi_1)$ of $\GG$ yields a $\GG$-quasi Poisson 
structure $P_\ell$ on $\GG^{2 \ell}$ relative to $\oomega$ and a
$\GG$-momentum mapping $\Phi_\ell \colon \GG^{2 \ell} \to \GG$ for $P_\ell$
relative to $\oomega$. Including in the fusion process $n\geq 1$ 
conjugacy classes 
$(\CcC_1,P_{\CcC_1}, \iota_1),\ldots, (\CcC_n,P_{\CcC_n}, \iota_n)$ in 
$\GG$, cf. Proposition \ref{momG}, we obtain a
$\GG$-quasi Poisson structure $P_{\ell,n}$ on 
$\GG^{2 \ell} \times \CcC_1 \times \ldots \times \CcC_n$ relative to $\oomega$  and a
$\GG$-momentum mapping  
\begin{equation}
\Phi_{\ell,n} \colon \GG^{2 \ell}\times \CcC_1 \times \ldots \times \CcC_n 
\longrightarrow \GG
\label{momelln}
\end{equation}
for $P_{\ell,n}$ relative to $\oomega$.
Including, instead,
in the fusion process $n\geq 1$ copies of 
$(\GG,P_\GG, \Id)$, cf. Proposition \ref{momG}, we obtain a
$\GG$-quasi Poisson structure ${}^\GG P_{\ell,n}$ on 
$\GG^{2 \ell} \times \GG^n$ relative to $\oomega$  and a
$\GG$-momentum mapping  
\begin{equation}
{}^\GG\Phi_{\ell,n} \colon \GG^{2 \ell}\times \GG^n 
\longrightarrow \GG
\label{momellnG}
\end{equation}
for ${}^\GG P_{\ell,n}$ relative to $\oomega$.
In the same vein, successively fusing $n \geq 2$  
conjugacy classes 
$(\CcC_1,P_{\CcC_1}, \iota_1),\ldots, (\CcC_n,P_{\CcC_n}, \iota_n)$ 
in 
$\GG$ yields
 a
$\GG$-quasi Poisson structure $P_{0,n}$ on 
$\CcC_1 \times \ldots \times \CcC_n$ relative to $\oomega$  and a
$\GG$-momentum mapping  
\begin{equation}
\Phi_{0,n} \colon  \CcC_1 \times \ldots \times \CcC_n 
\longrightarrow \GG
\label{momellnn}
\end{equation}
for $P_{0,n}$ relative to $\oomega$.

Let $\pi$ be the fundamental group of a compact, connected, and oriented
(real) topological surface of genus $\ell\geq 0$ with 
$n \geq 0$ boundary circles.
Recall the presentation \eqref{standpre2}
of $\pi$  in Section \ref{reps} and, 
as there, suppose $n \geq 3$ when $\ell = 0$.
Consider the twisted representation spaces of the kind
$\mathrm{Rep}_X(\Gamma,\GG)$
and
$\mathrm{Rep}(\pi,\GG)_{\mathbf C}$
explained there,
cf.  as well  Subsection \ref{recognition} and Remark \ref{qhreduction} above. 
Recall that 
the $\GG$-quasi Poisson structures yield ordinary
Poisson structures on 
$\AA[\GG^{2 \ell}\times \CcC_1 \times \ldots \times \CcC_n]^\GG$
and on 
$\AA[\GG^{2 \ell}\times \GG^n]^\GG$, cf. Proposition \ref{prop1}.
We denote by
 $\Rep(\pi,\GG)$ the quotient of $\Hom(\pi, \GG)$ by $\GG$, the quotient
being suitably defined when $\GG$ is not compact, e.g., as 
a categorical (analytic or algebraic as the case may be) quotient;
see also Theorem \ref{qprt}.
For $n=0$, the space $\Rep(\pi,\GG)$ comes down to
a space of the kind $\mathrm{Rep}_X(\Gamma,\GG)$ for $X=0$.
The following is an immediate consequence of Theorem \ref{qprt};
it applies
to the analytic as well as  to the algebraic case.

\begin{prop}
\label{ms2}

\begin{enumerate}
\item
Applying the quasi Poisson reduction procedure in Theorem {\rm \ref{qprt}}
to the Hamiltonian $\GG$-quasi Poisson manifold
$(\GG^{2 \ell},P_\ell, \Phi_\ell)$ relative to $\oomega$ with respect to a 
point $z$ of the center of $\GG$ which lies in the image of $\Phi_\ell$
yields a Poisson algebra 
$(\AA[\mathrm{Rep}_X(\Gamma,\GG)],\pbra)$
of functions on a twisted representation space
of the kind $\mathrm{Rep}_X(\Gamma,\GG)$. 
\item
The canonical algebra map 
$
\AA[\GG^{2 \ell}]^\GG\to
\AA[\mathrm{Rep}_X(\Gamma,\GG)]
$
is compatible with the Poisson structures.

\item
Applying that quasi Poisson reduction 
procedure to the Hamiltonian $\GG$-quasi Poisson manifold
$(\GG^{2 \ell} \times \CcC_1 \times \ldots \times \CcC_n,P_{\ell,n}, 
\Phi_{\ell,n})$ relative to $\oomega$ with respect to  $e\in \GG $
yields a Poisson algebra 
$(\AA[\mathrm{Rep}(\pi,\GG)_{\mathbf C}],\pbra)$
of functions on a representation space of the kind
$\mathrm{Rep}(\pi,\GG)_{\mathbf C}$.
\item
The canonical algebra map 
$
\AA[\GG^{2 \ell}\times \CcC_1 \times \ldots \times \CcC_n]^\GG\to
\AA[\mathrm{Rep}(\pi,\GG)_{\mathbf C}]
$
is compatible with the Poisson structures.

\item
Applying that quasi Poisson reduction 
procedure to the Hamiltonian $\GG$-quasi Poisson manifold
$(\GG^{2 \ell} \times \GG^n,{}^\GG P_{\ell,n}, 
{}^\GG \Phi_{\ell,n})$ relative to $\oomega$ with respect to  $e\in \GG $
yields a Poisson algebra 
$(\AA[\mathrm{Rep}(\pi,\GG)],\pbra)$
of functions on a representation space of the kind
$\mathrm{Rep}(\pi,\GG)$.
\item
The canonical algebra map 
$
\AA[\GG^{2 \ell}\times \GG^n]^\GG\to
\AA[\mathrm{Rep}(\pi,\GG)]
$
is compatible with the Poisson structures.
\item
The canonical algebra map 
$
\AA[\mathrm{Rep}(\pi,\GG)]
\to \AA[\mathrm{Rep}(\pi,\GG)_{\mathbf C}]
$
by restriction
is compatible with the Poisson structures. 

\item
The canonical algebra map 
$
\AA[\GG^{2 \ell}\times \GG^n]^\GG\to
\AA[\GG^{2 \ell}\times \CcC_1 \times \ldots \times \CcC_n]^\GG
$
is compatible with the Poisson structures. \qed

\end{enumerate}

\end{prop}

\begin{rema}
{\rm
Maintain the circumstances of
Proposition \ref{ms2} and suppose 
the symmetric $\Ad$-invariant $2$-tensor  $\oomega$ in 
$\gg \otimes \gg$ is non-degenerate, i.e.,  arises from 
an $\Ad$-invariant symmetric bilinear form
on $\gg$.
Then the
various 
$\GG$-quasi Poisson structures under discussion
in Proposition \ref{ms2} are 
$\GG$-quasi non-degenerate.
One can think of
the spaces of the kind
$\mathrm{Rep}_X(\Gamma,\GG)$ 
as symplectic leaves in the Poisson variety
$\GG^{2 \ell}/\GG$
and of those of the kind
$\mathrm{Rep}(\pi,\GG)_{\mathbf C}$
 as symplectic leaves in 
$\mathrm{Rep}(\pi,\GG)$
and in
a Poisson variety of the kind
\begin{equation}
\left(\GG^{2 \ell}\times \CcC_1 \times \ldots \times \CcC_n\right)/\GG
\subseteq
\left(\GG^{2 \ell}\times \GG^n\right)/\GG .
\end{equation}
One can render this observation precise
in terms of suitable regularity assumption and/or
suitable smooth open submanifolds of
$\mathrm{Rep}_X(\Gamma,\GG)$ and
$\mathrm{Rep}(\pi,\GG)_{\mathbf C}$, etc.
Proposition \ref{ms2}  (7)
includes a description of 
the variation of the Poisson structure
in the transverse directions when the conjugacy classes
$\mathbf C$ are allowed to move, cf. \cite{MR1815112}.
This answers in particular a question raised at the bottom of
p.~271 of \cite{MR1815112}.
One can also view
Proposition \ref{ms2}  (2)
 as describing 
the variation of the Poisson structure
in the transverse directions when the parameter $X$
is  allowed to move
but the union in $\GG^{2 \ell}/\GG$ of the spaces of the kind
$\mathrm{Rep}_X(\Gamma,\GG)$
is disconnected in the classical topology.

}
\end{rema}

Write the bracket \eqref{pb2} on $\GG$-invariant admissible $\bK$-valued
functions on $\GG^{2\ell}\times \CcC_1 \times \ldots \times \CcC_n$
which the bivector $P_{\ell,n}$ determines as $\pbra_{\ell,n}$ such that
$\{f,h\}_{\ell,n} =\langle P_{\ell,n}, d f \wedge d h\rangle$, for two 
functions $f$ and $h$, and interpret the notation $\pbra_{\ell,0}$
as the bracket $\pbra_{\ell}$  relative to $P_\ell$.

\begin{prop}
\label{alternate}
Let $\,\form\,$ be a non-degenerate 
$\Ad$-invariant symmetric bilinear form on $\gg$,
and suppose the symmetric $\Ad$-invariant $2$-tensor  $\oomega$ in 
$\gg \otimes \gg$ arises from 
$\,\form\,$ on $\gg$ as the image of the corresponding 
$2$-tensor in $\gg^* \otimes \gg^*$ under the inverse of the adjoint 
$\psidot \colon \gg \to \gg^*$ of $\,\form\,$.
Then the Hamiltonian $\GG$-quasi Poisson and weakly $\GG$-quasi 
Hamiltonian structures
$(P_{\ell,n},\Phi_{\ell,n})$
and
$(\sigma_{\ell,n},\Phi_{\ell,n})$
are $\Phi_{\ell,n}$-momentum dual to each other. Hence,
by Proposition {\rm \ref{compar32} (1)},
they are necessarily non-degenerate in the quasi sense.
\end{prop}

\begin{proof}
This results from combining Propositions \ref{6.3}, \ref{d.4}, and  \ref{d.5}.
\end{proof}

\subsection{Analytic case}
\label{ac}
Suppose $\,\form\,$ non-degenerate. Then a twisted representation space
of the kind $\mathrm{Rep}_X(\Gamma,\GG)$ also arises by ordinary symplectic 
reduction applied to an extended moduli space of the kind {\rm \eqref{exten2}}
and a representation space of the kind $\mathrm{Rep}(\pi,\GG)_{\mathbf C}$
arises by ordinary symplectic reduction applied to an extended moduli space 
of the kind  {\eqref{exten3}}, and thereby such a space acquires a Poisson 
algebra of functions\cite[Section 6 p.~754]{MR1370113},
 \cite[Theorem 9.1 p.~403]{MR1460627}.
Propositions \ref{ms2} an \ref{alternate} entail at once the following.

\begin{thm}
\label{ms1}
Suppose the symmetric $\Ad$-invariant $2$-tensor  $\oomega$ in 
$\gg \otimes \gg$ arises from the non-degenerate  $\Ad$-invariant symmetric 
bilinear form $\,\form\,$ on $\gg$ as the image of the corresponding 
$2$-tensor in $\gg^* \otimes \gg^*$ under the inverse of the adjoint 
$\psidot \colon \gg \to \gg^*$ of $\,\form\,$. Then the $\Phi$-quasi 
closed $2$-form $\sigma_{\ell,n}$ relative to $\,\form\,$ and the  
$\GG$-quasi Poisson structure $P_{\ell,n}$ relative to $\oomega$  are dual 
to each other via $\Phi_{\ell,n}$. Hence:
\begin{enumerate}
\item
For a $\GG$-invariant $\bK$-valued admissible function $f$ on 
$\GG^{2 \ell} \times \CcC_1 \times \ldots \times \CcC_n$, the quasi 
Hamiltonian 
vector field $X_f = \{f,\,\cdot\,\}_{\ell,n}$ satisfies the identity
\begin{equation}
\sigma_{\ell,n}(X_f,\,\cdot\,) = df
\end{equation}
and, for two $\GG$-invariant $\bK$-valued admissible functions $f,h$ on 
$\GG^{2 \ell} \times \CcC_1 \times \ldots \times \CcC_n$, the quasi
 Hamiltonian 
vector fields $X_f = \{f,\,\cdot\,\}_{\ell,n}$ and 
$X_h = \{h,\,\cdot\,\}_{\ell,n}$ satisfy the identity
\begin{equation}
\sigma_{\ell,n}(X_f,X_h) = \{h,f\}_{\ell,n}.
\end{equation}
\item
The reduced Poisson algebra of functions on $\mathrm{Rep}_X(\Gamma,\GG)$
arising from Theorem {\rm \ref{qprt} (3)} via $\oomega$ coincides with the 
reduced Poisson algebra of functions arising, via $\,\form\,$, from
symplectic reduction applied to the extended moduli space of the kind 
{\rm \eqref{exten2}};  the reduced Poisson algebra of functions on 
$\mathrm{Rep}(\pi,\GG)_{\mathbf C}$ arising from Theorem {\rm \ref{qprt} (3)}
via $\oomega$ likewise coincides with the reduced Poisson algebra of functions
arising, via $\,\form\,$, from symplectic reduction applied to
the  extended moduli space of the kind {\rm \eqref{exten3}}. \qed
\end{enumerate}
\end{thm}

\begin{rema}
{\rm
While Theorem \ref{ms1} is also a consequence of Theorem \ref{existence},
the above reasoning  does not involve this theorem.
}
\end{rema}

\subsection{Algebraic case}

Let $\bK$ be an algebraically closed field of characteristic zero. Suppose 
that $\GG$ is a reductive algebraic group defined over $\bK$. Let 
$\oomega$ be an $\Ad$-invariant symmetric $2$-tensor in the tensor square
$\gg \otimes \gg$ of the  Lie algebra 
$\gg$ of $\GG$, not necessarily non-degenerate. 
Consider the resulting algebraic Hamiltonian
$\GG$-quasi Poisson manifold
\begin{equation}
\left(\GG^{2 \ell} \times \CcC_1 \times \ldots \times \CcC_n,
\sigma_{\ell,n},\Phi_{\ell,n}\right).
\label{algeb}
\end{equation}
Proposition \ref{ms2} implies at once the following.
\begin{thm}
\label{ms3}
The choice of the $\Ad$-invariant symmetric $2$-tensor $\oomega$ 
in $\gg \otimes \gg$
 determines, on a twisted representation space
of the kind $\mathrm{Rep}_X(\Gamma,\GG)$ (realized as an algebraic quotient)
and on a representation space of the kind 
$\mathrm{Rep}(\pi,\GG)_{\mathbf C}$ (realized as an algebraic quotient),
an affine  Poisson variety structure. \qed
\end{thm}

\subsection{Loop group}
\label{loop}
The loop group is a special case of a group of gauge transformations,
and the construction in Example \ref{examp3}
endows the loop group with a quasi Poisson structure.

Thus consider a compact Lie group $\GG$ and let $L\GG$
be the loop group of $\GG$, that is, $L \GG = \Map(S^1,\GG)$,
the space of smooth maps from the circle $S^1$ to $\GG$,
endowed with the Fr\'echet topology and pointwise 
group operations.
Consider the trivial principal $\GG$-bundle $\prin$
on $S^1$ and let $\Conn$ denote its space of
smooth connections, endowed with the Fr\'echet topology,
by construction an affine space.

Maintain the choice of an  $\Ad$-invariant symmetric $2$-tensor 
$\oomega$ in the tensor square 
of the  Lie algebra $\gg$  of $\GG$ and, as in Example 
\ref{examp3},  write the resulting $3$-tensor
over $\Map(S^1,\GG)$ as $\phi_\oomega \colon S^1 \to \LAL^{\mrc,3}[\gg]$
and consider the resulting 
$\Map(S^1,\GG)$-quasi Poisson structure
$P_{\Map(S^1,\GG)}$ on $\Map(S^1,\GG)$ relative to 
$\phi_\oomega$.
It is now an interesting task to 
reinterpret and extend,
for the special case where
$\oomega$ arises from
a non-degenerate  $\Ad$-invariant symmetric 
bilinear form on $\gg$, 
the equivalence theorem
\cite[Theorem 8.3]{MR1638045}
in our setting.
Also  I suspect one can, even for general $\oomega$,
substitute $\Conn$
for $L \gg^*$ in \cite[Appendix A]{MR1880957}
and render
the formal Poisson structure
on $L \gg^*$ 
rigorous  on
$\Conn$ in the Fr\'echet setting.

Last but not least, perhaps the  quasi Poisson structure
on the loop group sheds new light on the results in
\cite{MR1289830}.

\subsection{Stokes data revisited}
\label{stokesd}
The quasi Poisson approach
applies to \lq\lq wild character varieties\rq\rq\ 
and leads, on such a space, to 
 an
 algebraic Poisson variety structure
relative to a not necessarily
quasi non-degenerate
quasi Poisson structure,
that is, Poisson structures
that do not necessarily arise from a quasi Hamiltonian structure
as in \cite{MR3126570}:

Consider a complex algebraic group $\GG$ and let $\oomega$ be
an $\Ad$-invariant symmetric $2$-tensor
in the tensor square  of  the Lie algebra of $\GG$.
Return to the circumstances of Subsection \ref{alternateq}.
After suitably readjusting the data if need be, 
the $2$-tensor $\oomega$ determines 
an $\Ad$-invariant symmetric $2$-tensor
in the tensor square  the Lie algebra of $\mathbf H$
and
 a Hamiltonian $\mathbf H$-quasi Poisson structure
$(P, \Phi)$ on the $\mathbf H$-manifold $\Hom_{\mathbb S}(\Pi, \GG)$
(notation in  \cite{MR3126570}), and
the corresponding bracket \eqref{pb2}
yields a Poisson structure on the affine coordinate ring that
turns the
affine algebraic
quotient
$\mathbf M_B(\Sigma) =\Hom_{\mathbb S}(\Pi,\GG)// \mathbf H$,
the \lq\lq wild character variety\rq\rq,
into an algebraic Poisson variety.
When 
$\oomega$ arises from
an $\Ad$-invariant symmetric bilinear form $\,\form\,$
on the Lie algebra of $\GG$,
the $2$-tensor
$P$ is $\Phi$-quasi non-degenerate, and
this $2$-tensor and
the $2$-form $\sigma$ which underlies the resulting
 $\mathbf H$-quasi Hamiltonian structure $(\sigma,\mmu)$
are $\Phi$-dual to each other, 
cf. Section \ref{comparison} for the terminology.
Justifying this observation does not involve
Theorem \ref{existence}.
In this particular case, the Poisson structure on
$\mathbf M_B(\Sigma) =\Hom_{\mathbb S}(\Pi,\GG)// \mathbf H$
which $P$ determines coincides with that
in \cite[Corollary 8.3 p.~43]{MR3126570}
which the $\GG$-quasi Hamiltonian structure
$(\sigma,\mmu)$ determines,
reproduced in Subsection \ref{alternateq}.
The algebraic Poisson structure on
$\Hom_{\mathbb S}(\Pi,\GG)// \mathbf H$
that arises from 
a general degenerate $\Ad$-invariant symmetric $2$-tensor
$\oomega$ 
is not available via an
 $\mathbf H$-quasi Hamiltonian structure,
 however.

Retain the choice
$\CcC_1\subseteq C_{\GG(Q_1)}, \ldots,\CcC_m\subseteq C_{\GG(Q_m)}$ of
$m$ conjugacy classes
and consider the 
affine algebraic
quotient
$\mmu^{-1}(\CcC_1 \times \ldots \times \CcC_m)// \mathbf H$.
By
Theorem \ref{qprt},
the Poisson bracket on
 the affine coordinate ring
of $\Hom_{\mathbb S}(\Pi,\GG)// \mathbf H$
induces  a Poisson structure on the affine coordinate ring 
of the
affine algebraic
quotient
$\mmu^{-1}(\CcC_1 \times \ldots \times \CcC_m)// \mathbf H$
that 
turns 
it into an affine algebraic Poisson variety.
When the quotient
$\mmu^{-1}(\CcC_1 \times \ldots \times \CcC_m)// \mathbf H$
arises within the quasi Hamiltonian approach, i.e.,
the $2$-tensor $\oomega$ at the start arises from
a non-degenerate  $\Ad$-invariant symmetric bilinear form on the Lie algebra
of the target group, 
that Poisson structure is not available 
by quasi Hamiltonian reduction unless
the variety under discussion is non-singular.
In particular, in the presence of singularities,
the reasoning in \cite{MR3126570} 
in terms of Thm 1.1  and Prop. 2.8 of that paper
does not explain whether and how
the Poisson structure descends
to one on the coordinate ring
 of
$\mmu^{-1}(\CcC_1 \times \ldots \times \CcC_m)// \mathbf H$.
Theorem \ref{qprt} implies 
that the vanishing ideal of $\mmu^{-1}(\CcC_1 \times \ldots \times \CcC_m)$
is a Poisson ideal and thereby
settles this issue.
Indeed, the search for 
a conclusive construction of
such an algebraic Poisson structure
in the singular case
prompted me to compose the present paper.

\section{Scholium}
\label{schol}
This section complements Subsections \ref{alternateq} and \ref{stokesd}
as well as Remarks 
\ref{qhreduction}, \ref{interp1}, \ref{reconcile}, and \ref{lurking}.

\subsection{}
Let $\GG$ be a complex algebraic group
and
consider an affine complex  algebraic  quasi Hamiltonian 
$\cgroup$-manifold $(M,\sigma,\mmu)$. 
Write the coordinate ring
of $M$ as $\CC[M]$
(rather than $\mathcat A[M]$).
Suppose that, furthermore, $\GG$ is reductive.
\cite[Proposition 2.8 p.~11]{MR3126570}
says that the affine geometric invariant theory quotient
$M// \GG$ \lq\lq is a Poisson variety\rq\rq, that is,
{\em the quasi Hamiltonian structure 
induces a Poisson bracket\/} on the affine coordinate ring
of $M// \GG$, by definition,
the algebra 
$\CC[M]^\GG$ of $\GG$-invariant functions in $\CC[M]$. 
The proof of \cite[Proposition 2.8]{MR3126570}
attributes the existence of the kind of Poisson structure under discussion  to
\cite[\S 5.4]{MR2642360},
\cite[\S 6]{MR1880957},
\cite[Proposition 4.6]{MR1638045}.
\cite{MR1880957} and \cite{MR1638045} apply to
the case where the group is compact.
Proposition \ref{4.6} above 
straightforwardly extends
\cite[Proposition 4.6]{MR1638045} 
to the case of a general Lie group 
(including an algebraic group)
with a non-degenerate
invariant symmetric bilinear form on its Lie algebra
and thereby
settles the existence of Hamiltonian vector fields
of invariant functions when the group is not necessarily compact.
The injectivity of
\eqref{inj1}
then entails that the vector field 
of an invariant  algebraic function is algebraic.
Hence the
 bracket \eqref{unpois1} of two invariant algebraic functions
is an invariant algebraic function.
By  Theorem \ref{unpois}, this bracket
satisfies the Jacobi identity.
Some care is necessary here
since constructing the Poisson bracket
via exponentiation does not lead
to an algebraic Poisson bracket for invariant algebraic functions.

The paper \cite{MR1880957}
develops quasi Poisson structures with respect to  a compact Lie group;
 \cite[Theorem 10.3 p.~24]{MR1880957}
establishes a
bijective correspondence between
quasi Hamiltonian structures and non-degenerate
Hamiltonian quasi Poisson structures for that case. 
The proof involves a construction 
which, locally,  factors through 
exponentiation
and uses, furthermore, the slice theorem. This proof
 does therefore not carry over for non-compact groups;
indeed, the etale slice theorem 
\cite{MR0342523}
applies only to
points in the group generating a semisimple conjugacy class.
\cite{MR2642360} does not say anything about the kind of
Poisson structure under discussion here.

Theorem \ref{existence} yields
that bijective correspondence
relative to a general group in an explicit manner.
An immediate consequence of Theorem \ref{existence}
is the fact that the bracket of two algebraic functions is algebraic.
Constructing such an algebraic  Poisson structure 
as in Theorem \ref{ms3}
directly from
an algebraic quasi Poisson manifold 
of the kind \eqref{algeb}
completely  gets away with those difficulties.
This also yields Poisson structures
arising from more general Hamiltonian quasi Poisson structures,
not necessarily relative to a
 non-degenerate  $\Ad$-invariant symmetric bilinear form on the Lie algebra.

\subsection{}
\label{8.2}
For  $\GG$ compact,
\cite[Section 9]{MR1638045}
contains a comparison of the moduli spaces arising in
 \cite{MR1638045} from quasi Hamiltonian spaces
with the corresponding gauge theory description
in  \cite{MR702806}, cf. Theorem \ref{visible} above.
The paper \cite{MR1370113}
offers a comparison of 
the construction of
such  moduli spaces
from extended moduli spaces 
with the corresponding gauge theory description
in  \cite{MR702806}.
Conclusions \ref{concl1}, \ref{concl2} and \ref{concl3}
above
render the comparison between the extended moduli spaces
and quasi Hamiltonian spaces explicit,
including 
the comparison for non-compact $\GG$.

For non-compact $\GG$, it is not clear 
to what extent
 a gauge theoretic description
of such moduli spaces is in general available, however:
Let $\prin\colon \Prin \to M$ be a principal $\GG$-bundle,
 $\Conn$ the affine space of connections on
$\prin$, $\mathrm{Flat}_\prin \subseteq \Conn$
the subspace of flat connections (supposed to be non-empty),  $\GG_\prin$
the group of gauge transformations, and
let $\pi$ denote the fundamental group of $M$.
In terms of suitable holonomies, it is  straightforward
to construct a bijective map
$\beta\colon \mathrm{Flat}_\prin / \GG_\prin \to \Hom(\pi,G)/\GG$
that is continuous relative to the 
 $ C^{\infty}$ topology
 but the continuity of the inverse map is more subtle.
Also, in general, the problem of interpreting the quotients 
on both sides of $\beta$ arises.
For $\GG$ and $M$ compact, there is no such problem, and
 Uhlenbeck compactness \cite{MR648356}
implies that a suitable Sobolev version of
$\mathrm{Flat}_\prin / \GG_\prin$
is compact 
whence $\beta$ is then a homeomorphism.
A quote from  the second paragraph of the introduction  of 
\cite{MR2030823}  says:
\lq\lq An elementary observation in gauge theory is that the moduli space
 of flat connections over a compact manifold with compact structure group
 is compact in the $ C^{\infty}$ topology. This is obvious from the fact
 that the gauge equivalence classes of flat connections are in one-to-one
 correspondence with conjugacy classes of representations of the
 fundamental group.\lq\lq\  I do not see how
to prove this directly without further technology.
\cite[14.11]{MR702806} quotes \cite{MR648356}.
Using more sophisticated Fr\'echet space slice techniques,
one can actually avoid Uhlenbeck's compactness theorem
or, more precisely, recover that statement 
in the  Fr\'echet setting for ordinary smooth functions,
but only for  $M$ and $\GG$ compact;
see \cite[Section 7]{MR3836789} and the references there for details.
For $\GG$ complex reductive,
the gauge theoretic description in \cite{MR887284}
of the self-duality moduli spaces under discussion
overcomes such analytical  difficulties
by reducing to a maximal compact subgroup;
indeed this is how hyperkaehler reduction works
in the infinite-dimensional setting in 
\cite{MR887284}.
See also \cite{self}.
In light of these observations, there is a question mark whether
the argument in
\cite{MR762512} 
(quoted again in \cite[Section 1.2]{MR2094117})
for the closedness of the symplectic structure under discussion
is complete when the target Lie group
is not compact.
 \cite[p.133]{MR3155540} explicitly claims
\lq\lq Goldman extended this construction to non-compact groups\rq\rq,\ and
\cite{MR1461568} naively extends the Atiyah-Bott approach
to principal bundles with complex structure group.
Also, \cite[p.~2]{MR3126570} refers to the \lq\lq complexification
of the viewpoint of Atiyah-Bott \cite{MR702806}\rq\rq\ 
(without further explanation)
but, fortunately, this \lq\lq complexification\rq\rq\ 
is only of heuristic significance in  \cite{MR3126570}.
To my understanding,
the closedness of that symplectic structure
is a consequence of the results in \cite{MR887284}
for $\GG$ complex reductive (\cite{MR887284} 
considers the case $\GG = \mathrm{SL}_2(\CC)$ 
but the reasoning is valid for a general 
complex reductive group),
and the first conclusive argument for that closedness
in the general case
is in \cite{MR1112494}. 
The results in 
\cite{MR1460627}, \cite{MR1370113}, \cite{MR1670408}, \cite{MR1815112}, 
\cite {MR1277051}, \cite{MR1470732} also include this closedness.

\subsection{}
Let $\pi$ be a finitely presented discrete group.
The terminology {\em representation variety\/}
for $\Hom(\pi,\mathrm{GL}(n,\CC))$ 
goes back at least to 
\cite[Definition 1.18 p.~18]{MR818915};
more generally, for a complex algebraic group $\GG$, 
the terminology {\em representation variety\/}
for $\Hom(\pi,\GG)$ 
or for the affine categorical
quotient $\Rep(\pi, \GG)= \Hom(\pi,\GG)// \GG$
as well as  {\em character variety\/} for
$\Rep(\pi, \GG)$ is nowadays standard.
The paper \cite{MR818915} extensively studies this quotient for
$\GG =\mathrm{GL}(n,\CC)$ without introducing a name for it.
It has become common,
for a not necessarily complex Lie group $\GG$,
also to refer to a quotient space
of the kind $\Rep(\pi,\GG)$
or even to a subspace
thereof as a 
{\em representation variety\/}
\cite{MR1880957}, \cite{MR762512}
 or as a {\em character variety\/} \cite{MR3155540}.
This is unfortunate since, e.g.,
when $\GG$ is compact,
$\Rep(\pi,\GG)$ is locally semi-algebraic
(even globally) but
$\Rep(\pi,\GG)$ does not coincide with the real points of the variety
$\Rep(\pi,\GG^\CC)$.
Some clean up of the situation is in \cite{MR1938554}
and in \cite{MR2883413}.

\subsection{} Once
\cite{MR1638045}
had appeared, among the Poisson community
a belief has evolved to the effect that the quasi Hamiltonian 
spaces supersede the extended moduli spaces
developed in
\cite{MR1460627}, \cite{MR1370113}, \cite {MR1277051}, 
\cite{MR1470732}.
The more recent papers 
\cite{MR3838878},
\cite{MR2884041}
indicate that this belief is not called for.
In fact, these papers 
exploit extended moduli spaces to
elaborate on an approach to
equivariant Floer theory.
This is interesting  because
\cite{MR1362845} in particular arose out of an attempt to
understand Floer theory by
using the symplectic structure of an appropriate finite-dimensional
moduli space to 
search for a replacement 
of the gradient \lq flow\rq\  of the Chern-Simons invariant
on a space of connections.
Moreover,
the generalization of extended moduli spaces in
\cite{MR1670408} does not match the quasi Hamiltonian picture.
An application of this generalization
is a purely finite-dimensional
characterization of the Chern-Simons function, 
solving a problem posed in \cite{MR1078014}.
Within this framework,
one can as well develop 
a purely finite-dimensional approach to Donaldson polynomials
(mentioned in the Introduction of \cite{MR1670408} but unpublished).
Also, extended moduli spaces play a major role in \cite{2007.13285}
and
\cite{MR4220999}.

\section*{Acknowledgements}
I am indebted to A. Weinstein for sharing insight
and for generous support over the years.
His paper \cite{MR1362845}
prompted the development of 
 extended moduli spaces and thereafter that of
 quasi Hamiltonian structures.
Also, A. Weinstein has been a driving force behind 
 Dirac structures.
I gratefully acknowledge support by the CNRS and by the
Labex CEMPI (ANR-11-LABX-0007-01).

\section*{Data availability statement}

Data sharing is not applicable to this article as no datasets were generated 
or analyzed during the current study.

\section*{No conflict of interest statement}

The author has no conflicts of interest to declare.


\begin{thebibliography}{GHJW97}

\bibitem[AM78]{MR515141}
Ralph Abraham and Jerrold~E. Marsden.
\newblock {\em Foundations of mechanics}.
\newblock Benjamin/Cummings Publishing Co. Inc. Advanced Book Program, Reading,
  Mass., 1978.
\newblock Second edition, revised and enlarged, With the assistance of Tudor
  Ra{\c{t}}iu and Richard Cushman.



\bibitem[ABM09]{MR2642360}
Anton Alekseev, Henrique Bursztyn, and Eckhard Meinrenken.
\newblock Pure spinors on {L}ie groups.
\newblock {\em Ast\'{e}risque}, (327):131--199 (2010), 2009.

\bibitem[AKSM02]{MR1880957}
Anton Alekseev, Yvette Kosmann-Schwarzbach, and Eckhard Meinrenken.
\newblock Quasi-{P}oisson manifolds.
\newblock {\em Canad. J. Math.}, 54(1):3--29, 2002.


\bibitem[AMM98]{MR1638045}
Anton Alekseev, Anton Malkin, and Eckhard Meinrenken.
\newblock Lie group valued moment maps.
\newblock {\em J. Differential Geom.}, 48(3):445--495, 1998.

\bibitem[Ati90]{MR1078014}
Michael Atiyah.
\newblock {\em The geometry and physics of knots}.
\newblock Lezioni Lincee. [Lincei Lectures]. Cambridge University Press,
  Cambridge, 1990.

\bibitem[AB83]{MR702806}
Michael~F. Atiyah and Raoul Bott.
\newblock The {Y}ang-{M}ills equations over {R}iemann surfaces.
\newblock {\em Philos. Trans. Roy. Soc. London Ser. A}, 308(1505):523--615,
  1983.


\bibitem[Aud97]{MR1461568}
Mich\`ele Audin.
\newblock Lectures on gauge theory and integrable systems.
\newblock In {\em Gauge theory and symplectic geometry ({M}ontreal, {PQ},
  1995)}, volume 488 of {\em NATO Adv. Sci. Inst. Ser. C: Math. Phys. Sci.},
  pages 1--48. Kluwer Acad. Publ., Dordrecht, 1997.



\bibitem[Boa14]{MR3126570}
Philip~P. Boalch.
\newblock Geometry and braiding of {S}tokes data; fission and wild character
  varieties.
\newblock {\em Ann. of Math. (2)}, 179(1):301--365, 2014.
\newblock \url{https://arxiv.org/abs/1111.6228}.

\bibitem[Boa18]{MR3931781}
Philip Boalch.
\newblock Wild character varieties, meromorphic {H}itchin systems and {D}ynkin
  diagrams.
\newblock In {\em Geometry and physics. {V}ol. {II}}, pages 433--454. Oxford
  Univ. Press, Oxford, 2018.

\bibitem[BC05]{MR2103001}
Henrique Bursztyn and Marius Crainic.
\newblock Dirac structures, momentum maps, and quasi-{P}oisson manifolds.
\newblock In {\em The breadth of symplectic and {P}oisson geometry}, volume 232
  of {\em Progr. Math.}, pages 1--40. Birkh\"{a}user Boston, Boston, MA, 2005.

\bibitem[BCWZ04]{MR2068969}
Henrique Bursztyn, Marius Crainic, Alan Weinstein, and Chenchang Zhu.
\newblock Integration of twisted {D}irac brackets.
\newblock {\em Duke Math. J.}, 123(3):549--607, 2004.

\bibitem[CJ20]{2007.13285}
Suhyoung Choi and Hongtaek Jung.
\newblock Symplectic coordinates on the deformation spaces of convex projective
  structures on 2-orbifolds.
\newblock {\em Transformation groups}, 2022.
\newblock \url{https://arxiv.org/abs/2007.13285}.

\bibitem[CJK19]{MR4220999}
Suhyoung Choi, Hongtaek Jung, and Hong~Chan Kim.
\newblock Symplectic coordinates on {$\rm PSL_3(\Bbb R)$}-{H}itchin components.
\newblock {\em Pure Appl. Math. Q.}, 16(5):1321--1386, 2020.
\newblock \url{https://arxiv.org/abs/1901.04651}.


\bibitem[Cor88]{MR965220}
Kevin Corlette.
\newblock Flat {$G$}-bundles with canonical metrics.
\newblock {\em J. Differential Geom.}, 28(3):361--382, 1988.


\bibitem[Cou90]{MR998124}
Theodore~James Courant.
\newblock Dirac manifolds.
\newblock {\em Trans. Amer. Math. Soc.}, 319(2):631--661, 1990.

\bibitem[DF18]{MR3838878}
Aliakbar Daemi and Kenji Fukaya.
\newblock Atiyah-{F}loer conjecture: a formulation, a strategy of proof and
  generalizations.
\newblock In {\em Modern geometry: a celebration of the work of {S}imon
  {D}onaldson}, volume~99 of {\em Proc. Sympos. Pure Math.}, pages 23--57.
  Amer. Math. Soc., Providence, RI, 2018.

\bibitem[DH18]{MR3836789}
Tobias Diez and Johannes Huebschmann.
\newblock Yang--{M}ills moduli spaces over an orientable closed surface via
  {F}r\'echet reduction.
\newblock {\em J. Geom. Phys.}, 132:393--414, 2018.
\newblock \url{https://arxiv.org/abs/1704.01982}.

\bibitem[Don87]{MR887285}
Simon~K. Donaldson.
\newblock Twisted harmonic maps and the self-duality equations.
\newblock {\em Proc. London Math. Soc. (3)}, 55(1):127--131, 1987.

\bibitem[Gab81]{MR618321}
Ofer Gabber.
\newblock The integrability of the characteristic variety.
\newblock {\em Amer. J. Math.}, 103(3):445--468, 1981.


\bibitem[Gol84]{MR762512}
William~M. Goldman.
\newblock The symplectic nature of fundamental groups of surfaces.
\newblock {\em Adv. in Math.}, 54(2):200--225, 1984.

\bibitem[Gol04]{MR2094117}
William~M. Goldman.
\newblock The complex-symplectic geometry of {${\rm SL}(2,\mathbb
  C)$}-characters over surfaces.
\newblock In {\em Algebraic groups and arithmetic}, pages 375--407. Tata Inst.
  Fund. Res., Mumbai, 2004.

\bibitem[GHJW97]{MR1460627}
Krishnamurthi Guruprasad, Johannes Huebschmann, Lisa Jeffrey, and Alan
  Weinstein.
\newblock Group systems, groupoids, and moduli spaces of parabolic bundles.
\newblock {\em Duke Math. J.}, 89(2):377--412, 1997.
\newblock \url{https://arxiv.org/abs/dg-ga/9510006}.


\bibitem[Hit87]{MR887284}
Nigel~J. Hitchin.
\newblock The self-duality equations on a {R}iemann surface.
\newblock {\em Proc. London Math. Soc. (3)}, 55(1):59--126, 1987.


\bibitem[HK86]{MR823849}
Karl~H. Hofmann and Verena~S. Keith.
\newblock Invariant quadratic forms on finite-dimensional {L}ie algebras.
\newblock {\em Bull. Austral. Math. Soc.}, 33(1):21--36, 1986.

\bibitem[Hue95]{MR1370113}
Johannes Huebschmann.
\newblock Symplectic and {P}oisson structures of certain moduli spaces. {I}.
\newblock {\em Duke Math. J.}, 80(3):737--756, 1995.
\newblock \url{https://arxiv.org/abs/hep-th/9312112}.

\bibitem[Hue99]{MR1670408}
Johannes Huebschmann.
\newblock Extended moduli spaces, the {K}an construction, and lattice gauge
  theory.
\newblock {\em Topology}, 38(3):555--596, 1999.
\newblock \url{https://arxiv.org/abs/dg-ga/9505005},
  \url{https://arxiv.org/abs/dg-ga/9506006}.

\bibitem[Hue01a]{MR1815112}
Johannes Huebschmann.
\newblock On the variation of the {P}oisson structures of certain moduli
  spaces.
\newblock {\em Math. Ann.}, 319(2):267--310, 2001.
\newblock {\tt DG-GA/9710033}.

\bibitem[Hue01b]{MR1938554}
Johannes Huebschmann.
\newblock Singularities and {P}oisson geometry of certain representation
  spaces.
\newblock In {\em Quantization of singular symplectic quotients}, volume 198 of
  {\em Progr. Math.}, pages 119--135. Birkh\"auser, Basel, 2001.
\newblock \url{https://arxiv.org/abs/math/0012184}.

\bibitem[Hue11]{MR2883413}
Johannes Huebschmann.
\newblock Singular {P}oisson-{K}\"ahler geometry of stratified {K}\"ahler
  spaces and quantization.
\newblock In {\em Geometry and quantization}, volume~19 of {\em Trav. Math.},
  pages 27--63. Univ. Luxemb., Luxembourg, 2011.
\newblock \url{https://arxiv.org/abs/1103.1584}.

\bibitem[Hue21]{self}
Johannes Huebschmann.
\newblock Finite-dimensional construction of self-duality and related moduli
  spaces over a closed {R}iemann surface as stratified holomorphic symplectic
  spaces.
\newblock 2021.
\newblock \url{https://arxiv.org/abs/2108.01016}.

\bibitem[HJ94]{MR1277051}
Johannes Huebschmann and Lisa~C. Jeffrey.
\newblock Group cohomology construction of symplectic forms on certain moduli
  spaces.
\newblock {\em Internat. Math. Res. Notices}, (6):245--249, 1994.


\bibitem[Jef97]{MR1470732}
Lisa~C. Jeffrey.
\newblock Symplectic forms on moduli spaces of flat connections on
  {$2$}-manifolds.
\newblock In {\em Geometric topology ({A}thens, {GA}, 1993)}, volume~2 of {\em
  AMS/IP Stud. Adv. Math.}, pages 268--281. Amer. Math. Soc., Providence, RI,
  1997.

\bibitem[Kar92]{MR1112494}
Yael Karshon.
\newblock An algebraic proof for the symplectic structure of moduli space.
\newblock {\em Proc. Amer. Math. Soc.}, 116(3):591--605, 1992.

\bibitem[KNR94]{MR1289830}
Shrawan Kumar, Mudumbai~S. Narasimhan, and Annamalai Ramanathan.
\newblock Infinite {G}rassmannians and moduli spaces of {$G$}-bundles.
\newblock {\em Math. Ann.}, 300(1):41--75, 1994.

\bibitem[Lab13]{MR3155540}
Fran\c{c}ois Labourie.
\newblock {\em Lectures on representations of surface groups}.
\newblock Z\"urich Lectures in Advanced Mathematics. European Mathematical
  Society (EMS), Z\"urich, 2013.

\bibitem[LBS11]{MR2806566}
David Li-Bland and Pavol \v{S}evera.
\newblock Quasi-{H}amiltonian groupoids and multiplicative {M}anin pairs.
\newblock {\em Int. Math. Res. Not. IMRN}, (10):2295--2350, 2011.

\bibitem[LM85]{MR818915}
Alexander Lubotzky and Andy~R. Magid.
\newblock Varieties of representations of finitely generated groups.
\newblock {\em Mem. Amer. Math. Soc.}, 58(336):xi+117, 1985.

\bibitem[Lun73]{MR0342523}
Domingo Luna.
\newblock Slices \'etales.
\newblock {\em Bull. Soc. Math. France}, {\sl M\'emoire} 33:81--105, 1973.

\bibitem[Lun75]{MR364272}
Domingo Luna.
\newblock Sur certaines op\'{e}rations diff\'{e}rentiables des groupes de
  {L}ie.
\newblock {\em Amer. J. Math.}, 97:172--181, 1975.

\bibitem[Lun76]{MR423398}
Domingo Luna.
\newblock Fonctions diff\'{e}rentiables invariantes sous l'op\'{e}ration d'un
  groupe r\'{e}ductif.
\newblock {\em Ann. Inst. Fourier (Grenoble)}, 26(1):ix, 33--49, 1976.

\bibitem[Mil63]{MR0163331}
John~W. Milnor.
\newblock {\em Morse theory}.
\newblock Based on lecture notes by M. Spivak and R. Wells. Annals of
  Mathematics Studies, No. 51. Princeton University Press, Princeton, N.J.,
  1963.

\bibitem[MW12]{MR2884041}
Ciprian Manolescu and Christopher Woodward.
\newblock Floer homology on the extended moduli space.
\newblock In {\em Perspectives in analysis, geometry, and topology}, volume 296
  of {\em Progr. Math.}, pages 283--329. Birkh\"{a}user/Springer, New York,
  2012.

\bibitem[RS90]{MR1087217}
Roger~W. Richardson, Jr. and Peter~J. Slodowy.
\newblock Minimum vectors for real reductive algebraic groups.
\newblock {\em J. London Math. Soc. (2)}, 42(3):409--429, 1990.


\bibitem[{\v{S}}W01]{MR2023853}
Pavol {\v{S}}evera and Alan Weinstein.
\newblock Poisson geometry with a 3-form background.
\newblock Number 144, pages 145--154. 2001.
\newblock Noncommutative geometry and string theory (Yokohama, 2001).


\bibitem[Sim92]{MR1179076}
Carlos~T. Simpson.
\newblock Higgs bundles and local systems.
\newblock {\em Inst. Hautes \'Etudes Sci. Publ. Math.}, (75):5--95, 1992.

\bibitem[Sim94]{MR1307297}
Carlos~T. Simpson.
\newblock Moduli of representations of the fundamental group of a smooth
  projective variety. {I}.
\newblock {\em Inst. Hautes \'Etudes Sci. Publ. Math.}, (79):47--129, 1994.

\bibitem[Sim95]{MR1320603}
Carlos~T. Simpson.
\newblock Moduli of representations of the fundamental group of a smooth
  projective variety. {II}.
\newblock {\em Inst. Hautes \'Etudes Sci. Publ. Math.}, (80):5--79 (1995),
  1994.


\bibitem[SL91]{MR1127479}
Reyer Sjamaar and Eugene Lerman.
\newblock Stratified symplectic spaces and reduction.
\newblock {\em Ann. of Math. (2)}, 134(2):375--422, 1991.


\bibitem[Uhl82]{MR648356}
Karen~K. Uhlenbeck.
\newblock Connections with {$L^{p}$} bounds on curvature.
\newblock {\em Comm. Math. Phys.}, 83(1):31--42, 1982.

\bibitem[Weh04]{MR2030823}
Katrin Wehrheim.
\newblock {\em Uhlenbeck compactness}.
\newblock EMS Series of Lectures in Mathematics. European Mathematical Society
  (EMS), Z\"urich, 2004.

\bibitem[Wei95]{MR1362845}
Alan Weinstein.
\newblock The symplectic structure on moduli space.
\newblock In {\em The {F}loer memorial volume}, volume 133 of {\em Progr.
  Math.}, pages 627--635. Birkh\"auser, Basel, 1995.

\end{thebibliography}


\def\cprime{$'$} \def\cprime{$'$} \def\cprime{$'$} \def\cprime{$'$}
  \def\cprime{$'$} \def\cprime{$'$} \def\cprime{$'$} \def\cprime{$'$}
  \def\dbar{\leavevmode\hbox to 0pt{\hskip.2ex \accent"16\hss}d}
  \def\cprime{$'$} \def\cprime{$'$} \def\cprime{$'$} \def\cprime{$'$}
  \def\cprime{$'$} \def\Dbar{\leavevmode\lower.6ex\hbox to 0pt{\hskip-.23ex
  \accent"16\hss}D} \def\cftil#1{\ifmmode\setbox7\hbox{$\accent"5E#1$}\else
  \setbox7\hbox{\accent"5E#1}\penalty 10000\relax\fi\raise 1\ht7
  \hbox{\lower1.15ex\hbox to 1\wd7{\hss\accent"7E\hss}}\penalty 10000
  \hskip-1\wd7\penalty 10000\box7}
  \def\cfudot#1{\ifmmode\setbox7\hbox{$\accent"5E#1$}\else
  \setbox7\hbox{\accent"5E#1}\penalty 10000\relax\fi\raise 1\ht7
  \hbox{\raise.1ex\hbox to 1\wd7{\hss.\hss}}\penalty 10000 \hskip-1\wd7\penalty
  10000\box7} \def\polhk#1{\setbox0=\hbox{#1}{\ooalign{\hidewidth
  \lower1.5ex\hbox{`}\hidewidth\crcr\unhbox0}}}
  \def\polhk#1{\setbox0=\hbox{#1}{\ooalign{\hidewidth
  \lower1.5ex\hbox{`}\hidewidth\crcr\unhbox0}}}
  \def\polhk#1{\setbox0=\hbox{#1}{\ooalign{\hidewidth
  \lower1.5ex\hbox{`}\hidewidth\crcr\unhbox0}}}

\end{document}